\documentclass[EJP,preprint]{ejpecp} % replace ECP by EJP if needed.
\SHORTTITLE{Regularization by noise for $1d$ mean field games}

\TITLE{Intrinsic regularization by noise for $1d$ mean field games} % \thanks is optional. Insert line breaks with \\

\AUTHORS{%
  François Delarue\footnote{Université Côte d'Azur, CNRS, Nice, France
    \EMAIL{francois.delarue@univ-cotedazur.fr}}
  \and %% remove this line and below if single author
  Youssef Ouknine 
  \footnote{Cadi Ayyad University \& Mohammed VI Polytechnic University, Morocco. \BEMAIL{ouknine@uca.ac.ma}}}%AUTHORS
 % Separate items with ;

\AMSSUBJ{49N80, 91A16, 60H15, 60G57} % Edit. Separate items with ;
\SUBMITTED{January 6, 2024} % Edit.
\VOLUME{0}
\YEAR{2020}
\PAPERNUM{0}
\DOI{10.1214/YY-TN}

\ABSTRACT{The purpose of this article is to show that an intrinsic noise with values in the space ${\mathcal P}({\mathbb R})$ of $1d$ probability measures 
may force uniqueness to first order mean field games. 
The structure of the noise is inspired from the earlier work \cite{delarueHammersley2022rshe}. It reads as a coloured Ornstein-Uhlenbeck process 
with reflection on the boundary of quantile functions on the $1d$ torus, with the elements of the latter playing the role of indices for the continuum of players
underpinning the game. In \cite{delarueHammersley2022rshe}, the semi-group generated by the noise is shown to enjoy smoothing properties that 
become key in the study carried out here. Although the analysis 
is limited to the 1d setting, this is the first example of uniqueness forcing for generic mean field games set over an infinite dimensional set of probability measures and this 
may be one step forward towards  
a more systematic regularization by noise theory for mean field games.     }

\usepackage[T1]{fontenc}
\usepackage[utf8]{inputenc}
\usepackage{lmodern}
\usepackage[english]{babel}
\usepackage{graphicx,mathrsfs}
\usepackage{cases}
\usepackage{upgreek}
\usepackage{amsmath}
\usepackage{amsfonts}
\usepackage{amssymb} 
\usepackage{amsthm}
\usepackage{bbm}
\usepackage{bm} 
\usepackage{lipsum}
\usepackage{ulem}
\usepackage{subcaption}
\usepackage{xcolor}
\usepackage{url}
\usepackage{hyperref}
\usepackage{csquotes}
\usepackage{stmaryrd}
\usepackage{pgf,tikz}
\usepackage{pgfplots,pgfplotstable}
\usetikzlibrary{calc,trees,positioning,arrows,chains,shapes.geometric,%
    decorations.pathreplacing,decorations.pathmorphing,shapes,%
    matrix,shapes.symbols}
\usepackage{multirow}
\usepackage{tikz}
\usepackage{wrapfig}
\usepackage{graphicx}
\graphicspath{ {./images/} }
\usepackage{subcaption}
\usepackage{blindtext}
\usepackage[ruled,vlined]{algorithm2e}
\usepackage{xcolor}

\tikzset{
>=stealth',
  punktchain/.style={
    rectangle, 
    rounded corners, 
    draw=black, very thick,
    text width=10em, 
    minimum height=3em, 
    text centered, 
    on chain},
  line/.style={draw, thick, <-},
  element/.style={
    tape,
    top color=white,
    bottom color=blue!50!black!60!,
    minimum width=8em,
    draw=blue!40!black!90, very thick,
    text width=10em, 
    minimum height=3.5em, 
    text centered, 
    on chain},
  every join/.style={->, thick,shorten >=1pt},
  decoration={brace},
  tuborg/.style={decorate},
  tubnode/.style={midway, right=2pt},
}

\pgfplotsset{contents/.style={axis x line=none, axis y line=none, scale = 2,area style,ymin=-2,ymax=22,enlargelimits=true}}

\allowdisplaybreaks
\numberwithin{equation}{section}

\theoremstyle{plain}
\newtheorem{thm}{Theorem}[section]
\newtheorem{lem}[thm]{Lemma}%[section]
\newtheorem{prop}[thm]{Proposition}%[section]

\newtheorem{defn}[thm]{Definition}%[section]
\newtheorem{rem}[thm]{Remark}%[section]

\def \be {\begin{equation}}
\def \ee {\end{equation}}

\def\ud{\mathrm{d}}

\def\cC{{\mathcal C}}

\renewcommand{\phi}{\varphi}
\renewcommand{\epsilon}{\varepsilon}
\renewcommand{\tilde}{\widetilde}
\renewcommand{\hat}{\widehat}

\newtheorem{innercustomgeneric}{\customgenericname}
\providecommand{\customgenericname}{}

\begin{document}

%\color{red}
%TO BE DONE:
%\begin{enumerate}
%\item Insert references Bertoin, Yor, Ouknine...
%\end{enumerate}
%
%\color{black}

\section{Introduction} 

Mean field game theory was introduced in concomitant works by Lasry and Lions and by Huang, Caines and Malhamé, see \cite{HuangCainesMalhame1,huangMalhameCaines2007,lasryLions2006jeuxAchampMoyen,lasryLions2006jeuxAchampMoyen2,lasryLions2007}. 
Its purpose is to describe asymptotic versions of dynamic games with many weakly interacting agents.  In this approach, 
equilibria (in the sense of Nash) are described by means of a time-dependent path with values in the space of probability measures. 
Those probability measures describe
the successive states (as time goes by) 
of the population 
of players when all of them form a Nash equilibrium. 

Solutions (or equilibria) to mean field games 
are known to be characterized through a forward-backward system of two Partial Differential Equations (PDEs): the forward equation 
is a Fokker-Planck equation describing the evolution of the in-equilibrium population whilst the backward equation 
is a Hamilton-Jacobi equation for the value function to one reference player within the population, see \cite{cardaliaguet2019master,cardaliaguetporretta-cetraro,Lionscollege2}. Alternatively, 
equilibria can be described through a forward-backward system of two (possibly stochastic) differential equations, 
the forward one being for the state of one typical player inside the in-equilibrium population and the backward one being 
either for the optimal cost or the optimal strategy to the same reference player, see \cite{CarmonaDelarueI}. 

One main difficulty  with mean field games is that uniqueness is rare. Conditions are known under which uniqueness is indeed satisfied, but 
there are rather demanding. 
The most famous ones are the so-called monotonicity conditions, either stated in the sense of Lasry and Lions (or equivalently in the space of signed measures) 
or in the displacement monotone sense (or equivalently in the space of random variables lying above the space of probability measures), see   the additional references 
\cite{MR4509653,MR4499277} for the latter. 

The key objective of this article is to show that one can guarantee uniqueness by forcing the dynamics of the population randomly. 
This fact was already observed in previous contributions, but in some limited cases when the equilibria are known to live in 
a finite-dimensional space (for instance, when equilibria are Gaussian, see \cite{fog2018}, or when the mean field game is set over a finite state space, see \cite{mfggenetic}). 
When the framework is truly infinite dimensional, uniqueness has been shown to hold in \cite{delarue2019restoring}, but in presence of an infinite dimensional noise that destroys the mean field structure of the solutions. 
In contrast to the latter one, 
the main result here is to show that we can indeed force uniqueness and (at the same time) preserve the mean field structure of the solutions with a suitable form of common noise. 

The form of our common noise is taken from the previous work \cite{delarueHammersley2022rshe}. 
Therein, the authors study the construction of a diffusion process with values in ${\mathcal P}({\mathbb R})$ (the space of probability measures over ${\mathbb R}$).
The key point in this approach is to identify elements of 
${\mathcal P}({\mathbb R})$
with quantile functions defined on the $1d$ torus and then to 
consider, as dynamics, a coloured stochastic 
heat equation with reflection on the boundary of
the set of   
quantile functions.
 Importantly, 
the semi-group of the resulting process (acting on real-valued functions defined on ${\mathcal P}({\mathbb R})$) is shown to have a strong smoothing effect, mapping bounded onto Lipschitz functions, with a time integrable singularity in small time. The core of our approach is to 
subject the players 
of the game to this form of noise and then to use the smoothing 
properties established for the latter one in \cite{delarueHammersley2022rshe} in order to decouple the two forward and backward equations 
characterizing the equilibrium. In particular, 
the limitation of the analysis to the $1d$ setting directly stems from the 
approach introduced in \cite{delarueHammersley2022rshe}. If we could extend 
\cite{delarueHammersley2022rshe} to the higher dimensional setting, 
we could consider in turn games with $d$-dimensional players. 

Similar 
to solutions to standard mean field games, 
the equilibrium that is constructed is shown to be distributed, 
meaning that the equilibrium strategy of a reference player in the population 
is a feedback function of the private state of the reference player and of the statistical state of the
population. As such, the feedback function does not depend on the index of the reference player 
in the continuum of players. 
Based on this observation, 
we say that 
the regularization phenomenon enforcing uniqueness 
is \textit{intrinsic}. In fact, this feedback function 
is expected to be, at least formally, the solution of a second order analogue of the 
master equation 
for mean field games 
(see 
\cite{cardaliaguet2019master}). 
At this stage, the understanding of the Kolomogorov equation associated with 
the reflected stochastic heat equation constructed in 
\cite{delarueHammersley2022rshe} remains however too limited to 
make this guess rigorous. We hope to come back to this question in a future work.

Our article is organized as follows. 
The form of the game together with the main results
are presented in Section 
\ref{sec:1} (see in particular Theorem \ref{main statement}). 
The forward-backward system characterizing the equilibria is shown to be uniquely solvable 
in Section 
\ref{sec:FB}. 
In Section 
\ref{sec:intrinsic}, we prove that the equilibrium feedback has a distributed structure. 
Finally, 
we establish in Appendix some auxiliary results related  
to the noise itself.

\vspace{5pt} 

\noindent {\bf Notations}. We let ${\mathbb S} := {\mathbb R}/{\mathbb Z}$, which we equip 
 with the Lebesgue measure 
 ${\rm Leb}_{\mathbb S}$. We call $(e_0:=1,(e_k:=\sqrt{2} \cos( 2 \pi k \cdot))_{k \in {\mathbb N} \setminus \{0\})}$ the standard cosinus functions on 
 ${\mathbb S}$. For an exponent $\lambda \in (1/2,1)$  (which is fixed throughout the note), we let $Q$ be the operator mapping
 $e_k$ onto $k^{- \lambda} e_k$, for $
 k \in {\mathbb N}$. Also, we call  $(W_t)_{t \geq 0}$
 an $L^2({\mathbb S})$-valued $Q$-Brownian motion  defined on some usual filtered probability space $(\Omega,{\mathcal A},{\mathbb F},{\mathbb P})$, 
 with the following expansion:  
  \begin{equation*} 
 W_t(x) = \sum_{k \in {\mathbb N}} (Q e_k)(x) B_t^k, \quad x \in {\mathbb S}, \quad t \geq 0,
 \end{equation*} 
  for a collection $(B^k)_{k \in {\mathbb N}}$ of independent ${\mathbb F}$-Brownian motions. For purposes that are clarified below, we assume ${\mathbb F}$ to be generated by 
 ${\mathcal F}_0$ and the family 
  $(B^k)_{k \in {\mathbb N}}$. 
We say that a 1-periodic function 
from ${\mathbb R}$ to $[-\infty,+\infty]$ 
is 
non-decreasing if it is symmetric with respect $0$ and non-increasing on $[-1/2,0]$. 
Identifying 1-periodic functions with functions defined on ${\mathbb S}$, 
we call 
$U^2({\mathbb S})$ the class of those non-decreasing functions that are also square-integrable
(with respect to ${\rm Leb}_{\mathbb S}$).  
Importantly (see for instance\footnote{In fact, in 
 \cite{delarueHammersley2022rshe}, elements of $U^2({\mathbb S})$ are non-increasing on $[0,1/2]$ and non-decreasing on $[-1/2,0]$.
 We may easily pass from one choice to the other by means of the change of variable $x \mapsto x+1/2$.}
 \cite[Proposition 2.1]{delarueHammersley2022rshe}), any element 
of $U^2({\mathbb S})$ has a unique version in $L^2({\mathbb S})$ that is right-continuous in $0$ and left-continuous on 
$(0,1/2]$ (by symmetry such a version is in fact continuous in $0$). We say that this version is `canonical'. Also, the mapping 
\begin{equation} 
\label{eq:cadlag:version}
(x,h) \in {\mathbb S}
\times U^2({\mathbb S}) \mapsto 
\left\{
\begin{array}{ll}
\displaystyle \lim_{\delta \searrow 0} 
\frac1{\delta} \int_{x- \delta}^{x} h(r) \ud r \quad x \in (0,1/2] + {\mathbb Z}
\\
\displaystyle 
\lim_{\delta \searrow 0} 
\frac1{\delta} \int_x^{x+ \delta} h(r) \ud r \quad x \in [-1/2,0] + {\mathbb Z}
\end{array}
\right.
\end{equation}
that maps 
an element $h \in U^2({\mathbb S})$ onto the evaluation of the canonical version of $h$ at $x$ is jointly measurable in 
$(x,h)$.

Without any further 
comment, we will always consider this version. 
In the same spirit, we denote by $H^2_{\rm sym}({\mathbb S})$ the set of 
symmetric 1-periodic functions that belong to the standard Sobolev space $H^2({\mathbb S})$
(i.e., the second-order derivatives are in $L^2({\mathbb S})$). The dual of 
$H^2_{\rm sym}({\mathbb S})$ is denoted 
$H^{-2}_{\rm sym}({\mathbb S})$. 
 
Moreover, we call ${\mathcal P}_2({\mathbb R})$ the space of probability measures over ${\mathbb R}$
with a finite second moment. We equip it with the $2$-Wasserstein distance ${\mathbb W}_2$, see 
\cite[Chapter 5]{CarmonaDelarueI}. We recall in particular
from \cite{delarueHammersley2022rshe} that 
$({\mathcal P}_2({\mathbb R}),W_2)$ and $(U^2({\mathbb S}),\| \cdot \|_2)$ (where $\| \cdot \|_2$ is the $L^2({\mathbb S})$-norm) 
are isometric. In words, elements of $U^2({\mathbb S})$ should be regarded as `periodic' quantile functions and are 
canonically identified with probability measures on ${\mathbb R}$.
Indeed, for a probability measure $\mu \in {\mathcal P}_2({\mathbb R})$ 
with $F_\mu$ as cumulative distribution function
and with 
\begin{equation*}
F_\mu^{-1}(x) := \inf \bigl\{ t \in {\mathbb R} : F_\mu(t) \geq x \bigr\}, \quad x \in [0,1], 
\end{equation*} 
as generalized inverse of the cumulative distribution function, 
the periodic function given by 
\begin{equation} 
\label{eq:inverse:generalise}
x \mapsto 
\left\{ 
\begin{array}{ll}
F_\mu^{-1}(2x),&\quad x \in [0,1/2], 
\\
F_\mu^{-1}(-2x),&\quad x \in [-1/2,0],
\end{array}
\right. 
\end{equation} 
is the representative of $\mu$ in $U^2({\mathbb S})$. With a slight abuse of notation, we still denote this representative by $F_\mu^{-1}$. 

Lastly, we also recall that the Borel $\sigma$-field on ${\mathcal P}_2({\mathbb R})$ is generated 
by the mappings $A \mapsto \mu(A)$, for $A$ running in the Borel $\sigma$-field over ${\mathbb R}$.  
And, 
for any metric space 
${\mathcal X}$, we denote by ${\mathcal B}({\mathcal X})$ the 
Borel $\sigma$-field on ${\mathcal X}$.

\section{Mean field game with common noise and main result} 
\label{sec:1}

\subsection{Candidate for being an equilibrium: reflected SDE}
Following \cite{mfggenetic} (which addresses mean field games over a finite state space), the challenge is to define a suitable of class of probability-measure-valued dynamics on the top of which the mean field game will be eventually constructed. 
We thus introduce the following prototype of reflected SDEs, the driftless version of which is taken from \cite{delarueHammersley2022rshe}:
\begin{equation}
\label{eq:V}
\ud X_t(x) = - {\mathcal V}\bigl(t,x,\text{\rm Leb}_{\mathbb S} \circ X_t^{-1} \bigr) \ud t + \Delta X_t(x) \ud t + \ud W_t(x) + \ud \eta_t(x), \quad x \in {\mathbb S}, 
\end{equation} 
where ${\mathcal V}$ is a measurable mapping from $[0,T] \times {\mathbb S} \times {\mathcal P}_2({\mathbb R})$ to ${\mathbb R}$. We assume the latter to satisfy the following properties:

\begin{defn} 
\label{def:mathcalC}
We denote by ${\mathcal C}$ the class of measurable functions ${\mathcal V}:[0,T] \times {\mathbb S} \times {\mathcal P}_2({\mathbb R}) \rightarrow {\mathbb R}$ such that, for a certain constant $C \geq 0$, 
\begin{enumerate}
\item The function ${\mathcal V}$ is bounded by $C$;
\item  For any $(t,\mu) \in [0,T] \times {\mathcal P}_2({\mathbb R})$, the function $x \in {\mathbb S} \mapsto  {\mathcal V}(t,x,\mu)$ belongs to 
$U^2({\mathbb S})$; 
\item The function ${\mathcal V} : (t,\mu) \in [0,T] \times {\mathcal P}_2({\mathbb R}) \rightarrow L^2({\mathbb S})$ is $C$-Lipschitz, i.e. 
\begin{equation*} 
\int_{\mathbb S} 
\bigl\vert {\mathcal V}(t,x,\mu) - {\mathcal V}(t,x,\nu) \bigr\vert^2 \ud x \leq C {\mathbb W}_2^2 \bigl( \mu,\nu\bigr). 
\end{equation*} 
\end{enumerate}
\end{defn}

In Equation \eqref{eq:V}, the process $(\eta_t(x))_{0\le t \le T,x \in {\mathbb S}}$ is a reflection term that forces the solution to stay in 
the cone $U^2({\mathbb S})$ of square-integrable mappings from ${\mathbb S}$ to ${\mathbb R}$. The definition is as follows:

\begin{defn}
\label{def:existence}
On  $(\Omega,{\mathcal A},{\mathbb F},{\mathbb P})$ equipped with 
the $Q$-Brownian motion $(W_t)_{0 \le t \le T}$ 
and with a deterministic initial condition $X_0$ with values in $
U^2({\mathbb S})$, 
a process  
$(X_t,\eta_t)_{0 \le t \le T}$ solves the 
rearranged SHE 
\eqref{eq:V}
 if 
\begin{enumerate}
\item $(X_t)_{0 \le t \le T}$ is a continuous ${\mathbb F}$-adapted process with values in $
U^2({\mathbb S})$;
\item $(\eta_t)_{0 \le t \le T}$ is a continuous ${\mathbb F}$-adapted process with values in $H^{-2}_{\rm sym}({\mathbb S})$, starting from $0$ at $0$, such that, with probability 1, for any $u \in H^2_{\rm sym}({\mathbb S}) \cap U^2({\mathbb S})$, the path
$(\langle \eta_t,u \rangle)_{0 \le t \le T}$ is non-decreasing;
\item with probability 1, for any $u \in H^{2}_{\rm sym}({\mathbb S})$
and for any $t \in [0,T]$, 
\begin{equation*} 
	\begin{split}
		\langle  X_t-X_0,u \rangle    =  -  \int_0^t \langle {\mathcal V}\bigl(r,\cdot,{\rm Leb}_{\mathbb S} \circ X_r^{-1} \bigr),u \bigr\rangle \ud r +  \int_0^t \langle   {X}_r  ,\Delta     u \rangle dr +\langle W_t  ,u \rangle + \langle \eta_t ,u\rangle.
\end{split} 
\end{equation*}
\item for any $t \in [0,T]$, 
\begin{equation*}
	\lim_{\varepsilon\searrow0}\mathbb{E}\left[ \int_0^t  e^{\varepsilon\Delta}X_r   \cdot d  \eta_r \right]= 0.
\end{equation*}
\end{enumerate}
\end{defn}

One key technical result in our work is to show that, for a fixed choice of ${\mathcal V}$, Equation \eqref{eq:V} has a unique strong solution. This is an improvement of the results studied in 
\cite{delarueHammersley2022rshe}, in which ${\mathcal V}$ is just $0$. 

\begin{thm}
\label{thm:main:existence:uniqueness}
Given 
${\mathcal V} \in {\mathcal C}$ and a deterministic initial condition $X_0 \in U^2({\mathbb S})$, 
there exists a unique solution $(X_t,\eta_t)_{0 \le t \le T}$ to the rearranged SHE \eqref{eq:V} that satisfies 
Definition 
\ref{def:existence}.
Moreover, 
there exist constants $C,\varepsilon >0$
 such that 
\begin{equation} 
\label{eq:main:existence:uniqueness}
\begin{split}
&{\mathbb E} 
\Bigl[\bigl\|\nabla X_{t} \bigr\|_2^2 \Bigr]
\leq C + C \min \Bigl( \frac{1}{t} \| X_0 \|_2^2 , \bigl\| \nabla X_0 \bigr\|_2^2 \Bigr), \quad t \in (0,T],
\\
&{\mathbb E} 
\Bigl[ \exp \Bigl( \varepsilon \sup_{0 \le t \le T} \| X_t\|_2^2 \Bigr) 
\Bigr]
\leq C. 
\end{split}
\end{equation} 
\end{thm} 

We just provide the proof of uniqueness at this stage of the paper, because the argument is repeatedly used in the paper. Existence, which is much more 
difficult to address, is treated 
in Appendix (Section \ref{appendix}).

\begin{proof} 
The proof of uniqueness is just a variant of the proof of 
\cite[Proposition 4.14]{delarueHammersley2022rshe}. For this reason, we just  present a sketch of it, the details 
being left to the reader.
Intuitively, for any two solutions $(X_t(x),\eta_t(x))_{t \geq 0,x \in {\mathbb S}}$
and $(X_t'(x),\eta_t'(x))_{t \geq 0,x \in {\mathbb S}}$, we have
\begin{equation} 
\label{eq:uniqueness:1}
\begin{split}
\ud_t \| X_t - X_t' \|_2^2 
&= -  2 \bigl\langle X_t - X_t', {\mathcal V}(t,\cdot,{\rm Leb}_{\mathbb S} \circ X_t^{-1}) - 
{\mathcal V}(t,\cdot,{\rm Leb}_{\mathbb S} \circ (X_t')^{-1})
\bigr\rangle_2 \ud t 
\\
&\hspace{15pt} 
- \| \nabla (X_t - X_t') \|_2^2 \ud t  
 + 
2 \bigl\langle X_t - X_t', \ud \eta_t - \ud \eta_{t}' \bigr\rangle_2,
\end{split}
\end{equation}
with the bracket $\langle \cdot, \cdot \rangle_2$ standing for the inner product in $L^2({\mathbb S})$. 
Rigorously, this argument is false, as the 
last integral with respect to $\ud \eta_t - \ud \eta_t'$ is properly defined only when the integrand takes
values in $H^2_{\rm sym}({\mathbb S})$. 
The proof performed in 
\cite[Proposition 4.14]{delarueHammersley2022rshe} 
consists in mollifying 
$X_t - X_t'$ by $e^{\varepsilon \Delta}$ and then in letting $\varepsilon$ tend $0$. In this way, we can 
do as if the term in the right-hand side were well-defined. 
The next step is to invoke items 
$(2)$ and $(4)$ in Definition 
\ref{def:existence}: altogether, they guarantee that the contribution of 
$\langle X_t - X_t', \ud \eta_t - \ud \eta_{t}' \rangle_2$
is negative. 
The proof is then easily completed by using the Lispchitz property of ${\mathcal V}$ 
(item $(3)$ 
in Definition 
\ref{def:mathcalC}), which says that 
$\| 
 {\mathcal V}(t,\cdot,{\rm Leb}_{\mathbb S} \circ X_t^{-1}) - 
{\mathcal V}(t,\cdot,{\rm Leb}_{\mathbb S} \circ (X_t')^{-1}) \|_2 
\leq C \| X_t - X_t' \|_2$, 
together with Gronwall's lemma. 
\end{proof}

\begin{rem}
\label{rem:measurability}
By 
\eqref{eq:cadlag:version},
we can easily change 
the version of $(X_t)_{t \geq 0}$
in Definition 
\ref{eq:cadlag:version} in such a way that
it still satisfies all the items of the definition and, for any $(t,\omega) \in [0,T]   \times \Omega$, 
the function $x \in {\mathbb S} \mapsto X_t(x)$ is left-continuous on $(0,1/2]$, right-continuous on $[-1/2,0)$ 
and continuous in $0$. 

In fact, more can be said from 
Theorem 
\ref{thm:main:existence:uniqueness}. 
Indeed, we know that, for any $t>0$ and with probability 1, $\| \nabla X_t \|_2 < \infty$, from which we deduce that, 
for any $t>0$ and
with probability 1, 
the mapping  
$x \in {\mathbb S} \mapsto X_t(x)$ 
is continuous. 
In particular, 
for any bounded and continuous 
function $\varphi : {\mathbb R} \rightarrow {\mathbb R}$, 
the mapping $
x \in {\mathbb S} \mapsto {\mathbb E} [ \varphi(X_t(x)) ]$ is continuous. 
\end{rem}

\begin{rem} 
Below, we often denote by $(X^0_t)_{0 \le t \le T}$ the solution to 
\eqref{eq:V}
when ${\mathcal V} \equiv 0$. Obviously, it can be defined in arbitrary time (even for $t > T$). 
Also, it induces the following semi-group, acting on bounded measurable functions 
$\phi : {\mathcal P}_2({\mathbb R}) \rightarrow {\mathbb R}$: 
\begin{equation*} 
{\mathscr P}^0_t \phi : \mu \mapsto 
{\mathbb E} \bigl[ \phi \bigl( 
\text{\rm Leb}_{\mathbb S} \circ X_t^{-1}\bigr) 
\bigr], \quad t \geq 0. 
\end{equation*} 
\end{rem}

\subsection{Definition of the mean field game}
\label{subse:def:mean field game}
With a given ${\mathcal V} \in {\mathcal C}$, we associate the following deterministic control problem (in the random environment formed by the solution 
$(X_t)_{0 \le t \le T}$ to 
\eqref{eq:V}):
\begin{equation*}
\inf 
J\Bigl((\gamma_t)_{0 \le t \le T};(X_t)_{0 \le t \le T} \Bigr), 
\end{equation*}
with
\begin{equation}
\label{eq:cost} 
\begin{split}
J\Bigl((\gamma_t)_{0 \le t \le T};(X_t)_{0 \le t \le T} \Bigr)
&:=
 {\mathbb E} \int_{\mathbb S} \biggl[ g\bigl(\Gamma_T(x),{\rm Leb}_{\mathbb S} \circ X_T^{-1} \bigr) 
\\
&+  \int_0^T 
\Bigl( 
f\bigl(\Gamma_t(x),{\rm Leb}_{\mathbb S} \circ X_t^{-1} \bigr) 
+
\frac12 
\vert \gamma_t(x) \vert^2 \Bigr)\ud t \biggr] \ud x,
\end{split}
\end{equation} 
where $f$ and $g$ are (measurable) real-valued function functions over 
${\mathbb R} \times {\mathcal P}({\mathbb R})$, at most of linear growth with respect to 
the spatial variable
and to the root of the second-order moment of the measure argument,  
i.e., $\vert f(x,\mu) \vert, \ \vert g(x,\mu) \vert \leq C ( 1 + \vert x \vert^2 + \int_{\mathbb R} \vert x \vert^2 \ud \mu(x) )^{1/2}$, 
and
the infimum is taken over pairs of ${\mathbb F}$-progressively-measurable random fields $\Gamma,\gamma$ from $[0,T] \times {\mathbb S}$
to ${\mathbb R}$, with ${\mathbb E} \int_0^T \int_{\mathbb S} \vert \gamma(x) \vert^2 \ud s \ud x< \infty$, that solve the dynamics
\begin{equation}
\label{eq:controlled:dynamics}
\ud\Gamma_t(x) = \ud X_t(x) +  \Bigl( \gamma_t(x) + {\mathcal V}\bigl(t,x,{\rm Leb}_{\mathbb S} \circ X_t^{-1} \bigr) \Bigr) \ud t, \quad t \in [0,T] \ ; 
\quad \Gamma_0(x) =X_0(x), \quad x \in {\mathbb S}. 
\end{equation} 
The key idea in this formula is very simple: without the forcing term $\Delta X_t(x) \ud t + \ud W_t(x)$, the equation for 
$(X_t(x))_{0 \le t \le T}$ just becomes
\begin{equation*}
\ud X_t(x) = - {\mathcal V}(t,x,{\rm Leb}_{\mathbb S} \circ X_t^{-1}) \ud t + \ud \eta_t(x).
\end{equation*} 
Still,
if 
the drift
can be written in 
feedback form, i.e., 
${\mathcal V}(t,x,{\rm Leb}_{\mathbb S} \circ X_t^{-1})
= \widetilde{\mathcal V}(t,X_t(x),{\rm Leb}_{\mathbb S} \circ X_t^{-1})$
as is shown below 
in 
the main statement (Theorem 
\ref{main statement}) for the mean field game under study, then 
the reflection term should be forgotten because the ordinary differential equation is expected to preserve monotonicity of the initial condition. 
In other words, 
the solution to the non-reflected equation
\begin{equation*}
\ud X_t(x) = - {\mathcal V}(t,x,{\rm Leb}_{\mathbb S} \circ X_t^{-1}) \ud t 
\end{equation*} 
is then expected to live in ${\mathcal C}$ when $X_0$ itself is in ${\mathcal C}$, even when there is no reflection in the dynamics. This says that, in absence of random forcing 
equation \eqref{eq:controlled:dynamics}
boils down to 
a mere controlled curve
$\dot{\Gamma}_t(x) = \gamma_t(x)$, the label $x$ here parametrizing the state of the initial condition. 
This explains intuitively  the connection with standard mean field games, all these explanations being addressed in a more rigorous manner in 
Proposition 
\ref{prop:MFG}. 

For the time being, 
we introduce the following definition.

\begin{defn}
\label{def:mfg:common}
We say that a field ${\mathcal V} \in {\mathcal C}$ forms a mean field equilibrium if 
the problem control \eqref{eq:cost}--\eqref{eq:controlled:dynamics} admits as optimal trajectory the curve
\begin{equation*} 
\bigl( \Gamma_t(x),\gamma_t(x) \bigr) = \bigl( X_t(x),- {\mathcal V}(t,x,{\rm Leb}_{\mathbb S} \circ X_t^{-1} ) \bigr), 
\quad x \in {\mathbb S}, 
\quad t \in [0,T]. 
\end{equation*} 
\end{defn}

Here is then our main statement: 

\begin{thm}
\label{main statement}
Assume that
$f$ and $g$ are $x$-differentiable and that 
$\partial_x f$
and
$\partial_x g$ are bounded over ${\mathbb R} \times {\mathcal P}_2({\mathbb R})$, non-decreasing in the argument $x$ and 
jointly Lipschitz continuous in $(x,\mu)$ (with ${\mathcal P}_2({\mathbb R})$ being equipped with 
${\mathbb W}_2$). Then,
 for any initial condition $X_0 \in U^2({\mathbb S})$, there exists a unique mean 
field equilibrium $(X_t(x))_{0 \leq t \leq T,x \in {\mathbb S}}$ in the sense of 
Definition 
\ref{def:mfg:common}
and 
there exists a certain ${\mathcal U} \in {\mathcal C}$ such that, for any initial condition $X_0 \in U^2({\mathbb S})$, the unique 
equilibrium is driven by ${\mathcal U}$.  

Moreover, there exists a measurable function 
$\tilde {\mathcal U} : [0,T] \times {\mathbb R} \times {\mathcal P}_2({\mathbb R}) \rightarrow {\mathbb R}$ (only depending on ${\mathcal U}$) such that,
for any 
mean field equilibrium $(X_t(x))_{0 \leq t \leq T,x \in {\mathbb S}}$ as above, 
 for almost every 
$t \in [0,T]$ and ${\mathbb P}$ almost surely,  
\begin{equation*} 
{\mathcal U}\bigl(t,x,{\rm Leb}_{\mathbb S} \circ X_t^{-1} \bigr) = \tilde {\mathcal U}\bigl(t,X_t(x),{\rm Leb}_{\mathbb S} \circ X_t^{-1} \bigr), \quad x \in {\mathbb S}. 
\end{equation*} 
\end{thm}

The following comments are in order: 
\begin{enumerate}
\item Here, the functions $f$ and $g$ are convex in $x$ because
the derivatives are non-decreasing in the argument $x$. This is substantially different from assuming that 
$f$ and $g$ have any form of monotonicity in the measure argument. In particular, our functions $f$ and $g$ here do NOT satisfy the Lasry-Lions monotonicity condition. Our statement should be thus regarded as a regularization by noise result. 

Actually, the need to assume $f$ and $g$ to be convex in $x$ comes from the fact there is no idiosyncractic noise in our model. In turn, there is no smoothing effect in the 
$x$-variable and we must rely on standard convexity properties in $x$ to prevent the emergence of singularities
(and thus lack of uniqueness). 
Obviously, this is part of our project to extend our results to models with an 
idiosyncratic noise, but this would require first to address 
the analogue of 
\cite{delarueHammersley2022rshe} for models with 
an idiosyncratic noise. 

\item The existence of a function $\tilde {\mathcal U}$ is a crucial point in our approach. It says that the equilibrium feedback in the game with common noise can be computed in terms of the sole states of the agent (which is here encoded through $X_t(x)$) and of the distribution (which is here encoded through ${\rm Leb}_{\mathbb S}
\circ X_t^{-1}$). The specific value of the label $x$ does not need to be known. In this sense, the equilibrium 
has the same distributed form as in standard mean field games with common noise. 
This is an important feature.

Furthermore, we insist on the fact that the two mappings ${\mathcal U}$ and $\tilde{\mathcal U}$ are independent of the initial condition. This is 
a form of Markov property that is very connected to the notion of master equation in mean field games. 
Would the forcing be absent, the reader who is aware of the theory would indeed easily identify $\tilde{\mathcal U}$ with the derivative in $x$ of the solution to 
the master equation (see \cite{cardaliaguet2019master}). 
In presence of the forcing $ \Delta X_t(x) \ud t + \ud W_t(x)$, which plays the role of a common noise, 
we are still lacking 
a similar sharp 
correspondance, but we conjecture 
that it holds true. In fact, the form of the master equation could be inferred from the It\^o rule established in 
the parallel 
work \cite{delarueHammersley2022ergodicrshe}, but a rigorous analysis of the associated Kolmogorov 
equation is still missing at this stage. This is our plan to come back to this point in the future. 

\item The result is just stated for a cost functional with a linear control and a quadratic Lagrangian.  We have indeed chosen to limit the presentation to this simpler framework. 
The reason is that almost all the technical aspects due to the random forcing are already present in the current framework. In particular, our result should 
be considered as a proof of concept. However, it would make perfect sense, for practical purposes, to extend the analysis to more general convex Hamiltonians $H$, depending on 
both the state variable $x$ and the momentum $p$ (as in the usual formulation of mean field games). Typically, one would then require $H$ to be convex in $(x,p)$, with 
strictly positive second order derivatives in $p$ (uniformly in $x$) and bounded derivatives in $x$ (uniformly in $p$).  Subsequently, the term $-{\mathcal V}(t,x,\textrm{\rm Leb}_{\mathbb S} 
\circ X_t^{-1})$ in 
\eqref{eq:V} would be replaced by 
$-\partial_p H(X_t(x), {\mathcal V}(t,x,\textrm{\rm Leb}_{\mathbb S} 
\circ X_t^{-1}))$. This would make slightly more difficult the analysis of the equation \eqref{eq:V}  (in particular, one should adapt the arguments given in 
Section 
\ref{appendix}), but the spirit of the analysis would remain very much the same. As for the representation of ${\mathcal V}$ 
(see Proposition 
\ref{prop:2:7:new} below), this would yield an additional term $-\partial_x H(X_s(x), {\mathcal V}(s,x,\textrm{\rm Leb}_{\mathbb S} 
\circ X_s^{-1}))$ in the integral appearing in the right-hand side of 
\eqref{eq:pontryagin}, but 
the principles underpinning the main smoothing result (see Lemma \ref{lem:phi})
would also remain the same. 
\end{enumerate}

Throughout the analysis below, the assumption of Theorem 
\ref{main statement} is in force. 

\subsection{Characterization of mean field equilibria}

The following result is an adaptation of the stochastic Pontryagin principle and 
is key in our proof of Theorem 
\ref{main statement}: 

\begin{prop}
\label{prop:2:7:new}
A field
${\mathcal V} \in {\mathcal C}$ forms a
mean field equilibrium if and only if, 
for almost every $t \in [0,T]$,
${\mathbb P}$-almost surely, 
for all $x \in {\mathbb S}$,
\begin{equation}
\label{eq:pontryagin}
\begin{split}
&{\mathcal V}\bigl(t,x,{\rm Leb}_{\mathbb S} \circ X_t^{-1} 
\bigr)
 \\
&= {\mathbb E} \biggl[ \partial_x g \bigl( X_T(x), 
{\rm Leb}_{\mathbb S} \circ X_T^{-1} \bigr)
+ \int_t^T 
\partial_x f \bigl( X_s(x), 
{\rm Leb}_{\mathbb S} \circ X_s^{-1} \bigr)
\ud s
 \, \vert \, {\mathcal F}_t \biggr], 
 \end{split} 
\end{equation} 
with $(X_t)_{0 \le t \le T}$ being the
 solution of 
 \eqref{eq:V}. 
 \end{prop}
%for any $(t,x,\mu) \in [0,T] \times {\mathbb R} 
%\times {\mathcal P}_2({\mathbb R})$, 
%\begin{equation}
%\label{eq:eq:U} 
%\begin{split}
%V(t,x,\mu) = {\mathbb E} \bigl[ \partial_x g \bigl( X_T(x), {\rm Leb}_{\mathbb S} \circ X_T^{-1} \bigr) \bigr], 
%\end{split} 
%\end{equation} 
%whenever $(X_s)_{t \leq s \leq T}$ solves the equation \eqref{eq:V} but with 
%$X_t \sim \mu$ as initial condition. 
%\end{prop}

We recall from 
Remark 
\ref{rem:measurability} that 
the mapping $(s,\omega,x) \mapsto X_s(x)$ is measurable, which guarantees that 
the term appearing inside the conditional expectation in the right-hand side of 
\eqref{eq:pontryagin}
is indeed a random variable. 

Also, 
the reader should regard 
Proposition
\ref{prop:2:7:new}
as a fixed point condition on the mapping ${\mathcal V}$. This condition is reminiscent of the fixed point condition in standard mean field games.

\begin{proof}
We start with the necessary condition, assuming that ${\mathcal V}$ forms a mean field equilibrium. 
Solving 
\eqref{eq:V}, 
with $(X_t)_{0 \le t \le T}$ as solution, 
we then consider 
the control problem 
\eqref{eq:cost}--\eqref{eq:controlled:dynamics}. 

We know that, for any pair control $(\gamma_t)_{0 \le t \le T}$
\begin{equation*} 
\frac{\ud}{\ud \varepsilon}_{\vert \varepsilon=0} 
J \Bigl( \bigl(- {\mathcal V}(t,\cdot,{\rm Leb}_{\mathbb S} \circ X_t^{-1}) + \varepsilon \gamma_t \bigr)_{0 \le t \le T};
\bigl( X_t\bigr)_{0 \le t \le T} 
\Bigr) =0.
\end{equation*}
Above, $(X_t)_{0 \le t \le T}$ does not depend on $(\gamma_t)_{0 \le t \le T}$ because it is the equilibrium. 
By 
\eqref{eq:controlled:dynamics}, the controlled curve 
associated with 
$( - {\mathcal V}(t,\cdot,{\rm Leb}_{\mathbb S} \circ X_t^{-1}) + \varepsilon \gamma_t)_{0 \le t \le T}$ 
 is $(X_t + \varepsilon \Gamma_t)_{0 \le t \le T}$ (pay attention that the controlled curve is not 
 $(\Gamma_t)_{0 \le t \le T}$ as one may think from 
\eqref{eq:controlled:dynamics}), where 
\begin{equation*} 
\Gamma_t(x)=  \int_0^t
\gamma_s(x)  \ud s, \quad t \in [0,T], \quad x \in {\mathbb S}, 
\end{equation*} 
from which we get (by differentiating w.r.t. $\varepsilon$) 
that, for any $t \in [0,T]$, 
\begin{equation}
\label{eq:proof:pontryagin} 
\begin{split}
&{\mathbb E}
\int_{\mathbb S}
 \biggl[ \partial_x g\bigl(X_T(x),{\rm Leb}_{\mathbb S} \circ X_T^{-1} \bigr) 
 \Gamma_T(x)
 \\
&\hspace{5pt}  + \int_0^T \Bigl( 
\partial_x f\bigl(X_t(x),{\rm Leb}_{\mathbb S} \circ X_t^{-1} \bigr) 
\Gamma_t(x)
-
{\mathcal V}\bigl(t,X_t(x),{\rm Leb}_{\mathbb S} \circ X_T^{-1} \bigr)
 \gamma_t(x) \ud t \Bigr) \biggr] \ud x = 0, 
\end{split} 
\end{equation} 
which may be rewritten as 
\begin{equation*} 
\begin{split}
&{\mathbb E}
\int_0^T 
\int_{\mathbb S}
 \Bigl[ 
 \partial_x g\bigl(X_T(x),{\rm Leb}_{\mathbb S} \circ X_T^{-1} \bigr)
 + \int_t^T 
 \partial_x f\bigl(X_s(x),{\rm Leb}_{\mathbb S} \circ X_s^{-1} \bigr)
  \ud s
  \\
&\hspace{45pt} - {\mathcal V}\bigl(t,X_t(x),{\rm Leb}_{\mathbb S} \circ X_T^{-1} \bigr) 
 \Bigr] \gamma_t(x)   \ud x \, \ud t = 0.
 \end{split}
\end{equation*} 
We deduce that, for almost every $t \in [0,T]$, ${\mathbb P}$ almost surely, for almost every $x \in {\mathbb S}$, 
\eqref{eq:pontryagin}
is satisfied. By one-sided continuity in $x$ of the two left-hand sides, we deduce that 
\eqref{eq:pontryagin}
holds true 
for almost every $t \in [0,T]$, ${\mathbb P}$ almost surely, for any $x \in {\mathbb S}$. 

We now study the converse, assuming that 
\eqref{eq:pontryagin}
is satisfied.  
Reverting back the computations, 
we  easily obtain 
\eqref{eq:proof:pontryagin}. Using the convexity of 
$f$ and $g$ in $x$, we easily deduce that 
\begin{equation*}
\begin{split}
&J \Bigl( \bigl( - {\mathcal V}(t,\cdot,{\rm Leb}_{\mathbb S} \circ X_t^{-1}) + \gamma_t \bigr)_{0 \le t \le T};
\bigl( X_t\bigr)_{0 \le t \le T} 
\Bigr)
\\
&\geq
J \Bigl( \bigl( - {\mathcal V}(t,\cdot,{\rm Leb}_{\mathbb S} \circ X_t^{-1})  \bigr)_{0 \le t \le T};
\bigl( X_t\bigr)_{0 \le t \le T} 
\Bigr)
\\
&\hspace{15pt} +
{\mathbb E}
\int_{\mathbb S}
 \biggl[ \partial_x g\bigl(X_T(x),{\rm Leb}_{\mathbb S} \circ X_T^{-1} \bigr) 
 \Gamma_T(x)
 \\
&\hspace{20pt}  + \int_0^T \Bigl( 
\partial_x f\bigl(X_t(x),{\rm Leb}_{\mathbb S} \circ X_t^{-1} \bigr) 
\Gamma_t(x)
-
{\mathcal V}\bigl(t,X_t(x),{\rm Leb}_{\mathbb S} \circ X_t^{-1} \bigr)
 \gamma_t(x)  \Bigr) \ud t \biggr] \ud x
 \\
 &= J \Bigl( \bigl( - {\mathcal V}(t,\cdot,{\rm Leb}_{\mathbb S} \circ X_t^{-1})  \bigr)_{0 \le t \le T};
\bigl( X_t\bigr)_{0 \le t \le T} 
\Bigr),
\end{split}
\end{equation*} 
with the last line following from 
\eqref{eq:proof:pontryagin}. This holds 
for any choice of $(\gamma_t)_{0 \le t \le T}$. By linearity of the set of controls, we get the result. 
\end{proof}

The next statement elaborates on the previous one.
The proof is deferred to the next section. 

\begin{prop}
\label{prop:fixed:point} 
There exists a unique mean field equilibrium 
for any initial condition $X_0 \in U^2({\mathbb S})$ 
if there exists a unique field
${\mathcal V} \in {\mathcal C}$
satisfying,  
for any $(t,x,\mu) \in [0,T] \times {\mathbb R} 
\times {\mathcal P}_2({\mathbb R})$, 
\begin{equation}
\label{eq:eq:U} 
\begin{split}
{\mathcal V}(t,x,\mu) = {\mathbb E} \biggl[ \partial_x g \bigl( X_T(x), {\rm Leb}_{\mathbb S} \circ X_T^{-1} \bigr)
+ \int_t^T 
\partial_x f \bigl( X_s(x), {\rm Leb}_{\mathbb S} \circ X_s^{-1} \bigr) \ud s
 \biggr], 
\end{split} 
\end{equation} 
whenever $(X_s)_{t \leq s \leq T}$ 
solves the equation \eqref{eq:V} driven by ${\mathcal V}$ but with 
$X_t \sim \mu$ as initial condition. 
\end{prop}

Identity 
\eqref{eq:eq:U} 
should be regarded as a forward-backward fixed point problem since
$(X_s)_{t \leq s \leq T}$ 
depends on ${\mathcal V}$.

\section{Analysis of the forward-backward fixed point problem} 
\label{sec:FB}

The purpose of this section is to kill two birds with one stone: not only we prove 
that the fixed point problem 
\eqref{eq:eq:U} 
has a unique solution but we also establish 
Proposition 
\ref{prop:fixed:point}. 
Throughout the section, we thus consider the following map 
\begin{equation*} 
\begin{split}
\Phi &: {\mathcal V} \in {\mathcal C} 
\\
&\mapsto 
\biggl( {\mathcal U} : (t,x,\mu) \mapsto   {\mathbb E} \biggl[ \partial_x g \bigl( X_T(x), {\rm Leb}_{\mathbb S} \circ X_T^{-1} \bigr)
+ \int_t^T 
\partial_x f \bigl( X_s(x), {\rm Leb}_{\mathbb S} \circ X_s^{-1} \bigr) \ud s \biggr] \biggr), 
\end{split} 
\end{equation*} 
where $(X_s)_{t \leq s \leq T}$ in the right-hand side
 solves the equation \eqref{eq:V} with $X_t \sim \mu$ as initial condition.

\subsection{Short time analysis}

\begin{lem} 
\label{lem:U} 
For ${\mathcal V} \in {\mathcal C}$, the function ${\mathcal U}=\Phi({\mathcal V})$ is well-defined and for any $(t,\mu) \in [0,T) \times {\mathcal P}_2({\mathbb R})$, 
the function $x \in {\mathbb S} \mapsto {\mathcal U}(t,x,\mu)$ is bounded and continuous and belongs to ${\mathcal C}$. 
 Moreover, there exists $c>0$ such that, for 
$T \leq c$, we can find a constant $C_c$ with the following property: if ${\mathcal V}$ is $C_c$-Lipschitz in $\mu$ (when regarded as an $L^2({\mathbb S})$-valued function), then 
$\Phi({\mathcal V})$ is also $C_c$-Lipschitz continuous (also when it is regarded as an $L^2({\mathbb S})$-valued function). 
\end{lem}

\begin{proof}
The fact that $\Phi({\mathcal V})$ is well-defined follows from 
Remark 
\ref{rem:measurability} (which implies joint measurability of the 
function 
$(s,\omega,x) \mapsto X_s(x)$). 
The same remark guarantees that, 
for any initial condition $(t,\mu)
\in 
[0,T) \times {\mathcal P}_2({\mathbb R})$, 
for any $s \in (t,T]$ and any bounded and continuous function $\varphi : {\mathbb R} \rightarrow 
{\mathbb R}$, 
the function  $x \mapsto {\mathbb E}[ \varphi(X_s(x))]$ is continuous. 
 We here apply this observation with $\varphi = \partial_x g$ and $\varphi = \partial_x f$ respectively.
 We deduce  that $\Phi({\mathcal V})(t,\cdot,\mu)$ is also continuous (when $t<T$). 
 Since
$\partial_x f$ and   
$\partial_x g$ are non-decreasing and each $X_s$ has values in $U^2({\mathbb S})$, we  also deduce, by composition,
that  $\Phi({\mathcal V})(t,\cdot,\mu)$ 
 belongs to 
$U^2({\mathbb S})$. At time $t=T$, 
$\Phi({\mathcal V})(T,x,\mu)=\partial_x g(F_\mu^{-1}(x),\mu)$ and satisfies the same 
prescription of one-sided continuity as in 
Remark 
\ref{rem:measurability}.

Item $(3)$ in Definition \ref{def:mathcalC} may be seen as a consequence of the proof of the second part of the statement, 
which may be proved as follows. 
For $(t,\mu,\nu) \in [0,T] \times {\mathcal P}_2({\mathbb R}) \times {\mathcal P}_2({\mathbb R})$, 
\begin{equation*} 
\begin{split}
&\int_{{\mathbb S}} \vert {\mathcal U}(t,x,\mu) - {\mathcal U}(t,x,\nu) \vert^2 \ud x 
\\
&\leq 
{\mathbb E} \int_{{\mathbb S}}
\bigl\vert \partial_x g\bigl(X_T(x),{\rm Leb}_{\mathbb S} \circ X_T^{-1}\bigr) - 
\partial_x g\bigl(X_T'(x),{\rm Leb}_{\mathbb S} \circ (X_T')^{-1}\bigr)
\bigr\vert^2 
\ud x
\\
&\hspace{15pt}
+
{\mathbb E} 
\int_t^T 
\int_{{\mathbb S}}
\bigl\vert \partial_x f\bigl(X_s(x),{\rm Leb}_{\mathbb S} \circ X_s^{-1}\bigr) - 
\partial_x f\bigl(X_s'(x),{\rm Leb}_{\mathbb S} \circ (X_s')^{-1}\bigr)
\bigr\vert^2 
\ud x.
\end{split} 
\end{equation*} 
where $(X_s)_{t \leq s \leq T}$ stands for the solution initialized from $(t,\mu)$
and $(X_s')_{t \leq s \leq T}$ for the solution initialized from $(t,\nu)$. 
By the Lipschitz properties of $\partial_x f$ and  $\partial_x g$ (in both variables), we obtain 
\begin{equation*} 
\begin{split}
\int_{{\mathbb S}} \vert {\mathcal U}(t,x,\mu) - {\mathcal U}(t,x,\nu) \vert^2 \ud x 
&\leq 
C_g {\mathbb E} \int_{{\mathbb S}}
\bigl\vert
X_T(x) - X_T'(x) \bigr\vert^2 
\ud x 
\\
&\hspace{15pt} + C_f  {\mathbb E} \int_t^T \int_{{\mathbb S}}
\bigl\vert
X_s(x) - X_s'(x) \bigr\vert^2 
\ud x \, \ud s
\\
&\leq C_{f,g} \bigl( 1 + T \bigr) \sup_{t \leq s \leq T}{\mathbb E}\bigl[ \| X_s - X_s' \|_2^2 \bigr],
\end{split} 
\end{equation*}
for a constant $C_{f,g}$ only depending on $f$ and $g$. 

Next, by proceeding as in 
\eqref{eq:uniqueness:1}, 
we obtain (for $C_{\mathcal V}$ denoting the square Lipschitz constant of ${\mathcal V}$) 
\begin{equation*} 
\sup_{t \leq s \leq T} 
\| X_s - X_s' \|_2^2 %+ \int_t^T \| \nabla X_s - \nabla X_s' \|_2^2 \ud s 
\leq 
\| X_t - X_t' \|_2^2 +  2 C_{\mathcal V} T  \sup_{t \leq s \leq T} 
\| X_s - X_s' \|_2^2, 
\end{equation*} 
and then, choosing $T C_{\mathcal V} \leq 1/4$, we obtain 
$
\sup_{t \leq s \leq T}
\| X_s - X_s' \|_2^2  \leq 
2 \| X_t - X_t' \|_2^2$, 
so that 
\begin{equation*} 
\begin{split}
&\int_{{\mathbb S}} \vert {\mathcal U}(t,x,\mu) - {\mathcal U}(t,x,\nu) \vert^2 \ud x 
\leq 
2 C_{f,g} \bigl( 1+ T\bigr) 
\| X_t - X_t' \|_2^2.
\end{split} 
\end{equation*}
Since $X_t$ and $X_t'$ are quantile functions, we have 
$\| X_t - X_t' \|_2^2={\mathbb W}_2^2(\mu,\nu)$, see \cite[Subsection 2.2]{delarueHammersley2022rshe}. 
Therefore, for $T C_{\mathcal V} \leq 1/4$,
\begin{equation*} 
\begin{split}
&\int_{{\mathbb S}} \vert {\mathcal U}(t,x,\mu) - {\mathcal U}(t,x,\nu) \vert^2 \ud x 
\leq 
2  C_{f,g} \bigl( 1+ T\bigr)  {\mathbb W}_2^2(\mu,\nu).
\end{split} 
\end{equation*}
Assuming $C_{\mathcal V} \leq 2 C_{f,g} (1+T) $ and then $2 T C_{f,g}(1+T) \leq 1/4$, we complete the proof
with $c:= 1/(8(C_{f,g}(1+T)))$
and $C_c:=  \sqrt{2 C_{f,g} (1+T)}$.
\end{proof}
In the sequel, we denote by ${\mathcal C}_c$ the subclass of ${\mathcal C}$ containing fields ${\mathcal V}$ which are 
$C_c$-Lipschitz continuous in $\mu$ (when regarded as being $L^2({\mathbb S})$-valued). 
It now remains to prove that 
\begin{prop}
\label{prop:2:2}
With the same notation as in the statement of Lemma  \ref{lem:U}, the function 
$\Phi$ is a (strict) contraction from ${\mathcal C}_c$ into itself when $T \leq \min(c, \ln(2)/(1+2C_c^2))$ and
 ${\mathcal C}_c$ is equipped with the distance 
 \begin{equation*}
d({\mathcal V},{\mathcal V}') := 
\sup_{0 \le s \le T} \sup_{\mu \in {\mathcal P}_2({\mathbb R})}
\| 
{\mathcal V}(s,\cdot,\mu) 
- 
{\mathcal V}'(s,\cdot,\mu) \|_2.
\end{equation*}
In particular, the mapping $\Phi$ has a fixed point in small time. 
\end{prop} 

\begin{proof}
We proceed as in the proof of Lemma 
\ref{lem:U}. 
For two fields ${\mathcal V}$ and ${\mathcal V}'$ in  ${\mathcal C}_c$, we denote by ${\mathcal U}$ and ${\mathcal U}'$ the respective images by 
$\Phi$. For a given $(t,\mu) \in [0,T] \times {\mathcal P}({\mathbb R})$, we call $(X_s)_{t \leq s \leq T}$ and 
$(X_s')_{t \leq s \leq T}$ the two solutions to \eqref{eq:V}, driven by ${\mathcal V}$ and ${\mathcal V}'$ respectively, with 
$(t,\mu)$ as initial condition.
Following the proof of Lemma \ref{lem:U}, we have 
% We have, for $T \leq c$ (with the same $c$ as in the proof of Lemma
%\ref{lem:U})
\begin{equation*} 
\begin{split}
\int_{\mathbb S}\vert {\mathcal U}(t,x,\mu) - {\mathcal U}'(t,x,\mu) \vert^2 \ud x 
&\leq 
C_{f,g} \bigl( 1 + T \bigr) \sup_{t \leq s \leq T}{\mathbb E}\bigl[ \| X_s - X_s' \|_2^2 \bigr].
\end{split}
\end{equation*}
By 
\eqref{eq:uniqueness:1}, 
for any $s \in [t,T]$, 
\begin{equation*} 
\begin{split}
e^{-s} \| X_s - X_s' \|_2^2 
&\leq  
\int_t^s 
e^{-r}
\bigl\| 
{\mathcal V}\bigl(r,\cdot,{\rm Leb}_{\mathbb S} \circ X_r^{-1} \bigr) 
- 
{\mathcal V}'\bigl(r,\cdot,{\rm Leb}_{\mathbb S} \circ (X_r')^{-1} \bigr) 
\bigr\|_2^2 \ud r
\\
&\leq  
2\int_t^s 
e^{-r}
\bigl\| 
{\mathcal V}\bigl(r,\cdot,{\rm Leb}_{\mathbb S} \circ X_r^{-1} \bigr) 
- 
{\mathcal V}'\bigl(r,\cdot,{\rm Leb}_{\mathbb S} \circ X_r^{-1} \bigr) 
\bigr\|_2^2 \ud r
\\
&\hspace{15pt} + 
2C_c^2 \int_t^T 
e^{-r}
 \|X_r - X_r' 
 \|_2^2 \ud r
 \\
 &\leq 2e^{2 C_c^2 T} \int_t^T 
e^{-r}
\bigl\| 
{\mathcal V}\bigl(r,\cdot,{\rm Leb}_{\mathbb S} \circ X_r^{-1} \bigr) 
- 
{\mathcal V}'\bigl(r,\cdot,{\rm Leb}_{\mathbb S} \circ X_r^{-1} \bigr) 
\bigr\|_2^2 \ud r
\end{split}
\end{equation*}
with the last line following from Gronwall's lemma. 
And then,
\begin{equation*}
\begin{split}
&\int_{\mathbb S}\vert {\mathcal U}(t,x,\mu) - {\mathcal U}'(t,x,\mu) \vert^2 \ud x 
\\
&\leq   2C_{f,g} \bigl( 1 + T \bigr)  e^{(1+2C_c^2) T} {\mathbb E} \int_t^T 
\bigl\| 
{\mathcal V}\bigl(s,\cdot,{\rm Leb}_{\mathbb S} \circ X_s^{-1} \bigr) 
- 
{\mathcal V}'\bigl(s,\cdot,{\rm Leb}_{\mathbb S} \circ X_s^{-1} \bigr) 
\bigr\|_2^2 \ud s.
\\
&\leq \frac{T}{4c} e^{(1+2C_c^2) T} 
\sup_{0 \le s \le T}
\sup_{\nu \in {\mathcal P}_2({\mathbb R})}
\| 
{\mathcal V}(s,\cdot,\nu) 
- 
{\mathcal V}'(s,\cdot,\nu) \|_2^2,
\end{split} 
\end{equation*} 
where we used the fact that 
$c= 1/(8(C_{f,g}(1+T)))$. 

By assumption, $T \leq \min(c, \ln(2)/(1+2C_c^2))$ and then $T e^{(1+2C_c^2)T} \leq 2c$. We can easily complete the proof. 
\end{proof}

\subsection{Martingale representation} 

We now provide further properties of the mapping $\Phi$. 

\begin{lem}
\label{lem:martingale}
Let ${\mathcal V}$ be in ${\mathcal C}$ and ${\mathcal U}:=\Phi({\mathcal V})$. For a given initial condition $(t,\mu) \in [0,T] \times {\mathcal P}_2({\mathbb R})$ to \eqref{eq:V}, call 
$(X_s)_{t \leq s \leq T}$ the corresponding solution
(driven by ${\mathcal V}$). Then, 
for any $s \in [t,T]$, 
with probability 1 under ${\mathbb P}$, 
for any $x \in {\mathbb S}$, 
\begin{equation*} 
{\mathcal U} \bigl(s,x,{\rm Leb}_{\mathbb S} \circ X_s^{-1} 
\bigr)
=
{\mathbb E} \biggl[ \partial_x g \bigl( X_T(x),{\rm Leb}_{\mathbb S} \circ X_T^{-1} 
\bigr)
+
\int_s^T 
 \partial_x f \bigl( X_r(x),{\rm Leb}_{\mathbb S} \circ X_r^{-1} 
\bigr)
\ud r
 \, \vert \, {\mathcal F}_s \biggr]. 
\end{equation*}
%\textcolor{red}{where ${\mathbb F}$ is the filtration generated by the underlying noise}. 
\end{lem}

\begin{proof}
The proof relies on a variation of the Markov property for \eqref{eq:V}.
Without any loss of generality, we can assume that we are working on the canonical space $\Omega={\mathscr C}([0,T],L^2({\mathbb S}))$ (space of continuous functions from $[0,T]$ to 
$L^2({\mathbb S})$) equipped with the 
$Q$-Brownian motion law. Then, we can consider 
a regular conditional probability $({\mathbb P}_{s,\omega})_{\omega \in \Omega}$ of ${\mathbb P}$ given the $\sigma$-field generated by 
$\sigma(W_r(\cdot) ; r \leq s)$ (with $W$ being understood as the canonical process). 

Then, $\omega$ ${\mathbb P}$-almost surely, under ${\mathbb P}_{s,\omega}$, $(X_r)_{s \leq r \leq T}$ can be regarded as the unique solution of \eqref{eq:V} initialized from 
$(s,X_s(\omega))$. By uniqueness in law of the solution (which follows from pathwise uniqueness), it must hold, for any $x \in {\mathbb S}$: 
\begin{equation*} 
\begin{split}
{\mathcal U} \bigl(s,x,{\rm Leb}_{\mathbb S} \circ (X_s(\omega))^{-1} 
\bigr) &= 
\int_{\Omega} 
\partial_x g \bigl( 
X_T(\omega')(x),
{\rm Leb}_{\mathbb S} \circ (X_T(\omega'))^{-1} 
\bigr)
\ud{{\mathbb P}_{s,\omega}}(\omega')
\\
&+ 
\int_{\Omega}  
\biggl[ \int_s^T 
\partial_x f \bigl( 
X_r(\omega')(x),
{\rm Leb}_{\mathbb S} \circ (X_r(\omega'))^{-1} 
\bigr)
\ud r
\biggr]
\ud {{\mathbb P}_{s,\omega}}(\omega') ,
\end{split}
\end{equation*}
which is exactly the claim. 
\end{proof}
We notice by martingale representation theorem in infinite dimension (see \cite[Proposition 4.1]{FuhrmanTessitore}) that, within the framework of Lemma
\ref{lem:martingale}, there exists a collection of progressively-measurable processes
$(Z^{t,\mu,k}(x))_{k \geq 0,x\in {\mathbb S}}$ (with values in ${\mathbb R}$) such that 
\begin{equation}
\label{eq:representation} 
\begin{split}
&\partial_x g \bigl( X_T(x),{\rm Leb}_{\mathbb S} \circ X_T^{-1} \bigr)
+ \int_t^T
\partial_x f \bigl( X_s(x),{\rm Leb}_{\mathbb S} \circ X_s^{-1} \bigr)
\ud s
\\
&= 
{\mathcal U}(t,x,\mu) + \sum_{k \in {\mathbb N}} \int_t^T Z^{t,\mu,k}_s(x) \ud B_s^k, \quad x \in {\mathbb S}, 
\end{split}
\end{equation} 
with 
\begin{equation*} 
\sum_{k \in {\mathbb N}} {\mathbb E} \int_t^T \vert Z^{t,\mu,k}_s(x) \vert^2 \ud s < \infty. 
\end{equation*} 
By boundedness and $\omega$-wise continuity in $x$ of the left-hand side 
(see Remark 
\ref{rem:measurability}) and of 
the first term in the right-hand side, we notice that 
\begin{equation*} 
\forall x \in {\mathbb S}, \quad \lim_{x' \rightarrow x} 
{\mathbb E} 
\sum_{k \in {\mathbb N}}\int_t^T \vert Z^{t,\mu,k}_s(x) - Z^{t,\mu,k}_s(x') \vert^2 \ud s = 0,
\end{equation*} 
from which we can easily find a version of each $Z^{t,\mu,k}(x)$ such that the
field $(s,x) \mapsto Z_s^{t,\mu,k}(x)$ is ${\mathbb F}$-progressively measurable for 
each $k$ (i.e., the mapping 
$(s,x,\omega) \in [0,t] \times {\mathbb S} \times \Omega \mapsto Z_s^{t,\mu,k}(x,\omega)$ is 
${\mathcal B}([0,t]) \otimes {\mathcal B}({\mathbb S}) \otimes {\mathcal F}_t/{\mathcal B}({\mathbb R})$
measurable for any $t \in [0,T]$). Details are left to the reader because a stronger version of this result is given in Lemma 
\ref{lem:2:5} below.

The following statement is indeed a standard but crucial result in BSDE theory. 
It clarifies the structure of the martingale representation term $(Z_s^{t,\mu,k}(x))_{t \leq s \leq T}$ in 
\eqref{eq:representation}. 
Intuitively, it is a way to access the derivatives of ${\mathcal U}$ with respect to the measure argument, but 
without any preliminary study of the differentiability properties of ${\mathcal U}$.  

\begin{lem}
\label{lem:2:5}
With the same notation as in 
\eqref{eq:representation}, 
there exists a collection of measurable functions 
\begin{equation*} 
\psi^k : [0,T] \times {\mathbb R} \times {\mathcal P}_2({\mathbb R}) \rightarrow {\mathbb R}, \quad k \in {\mathbb N},
\end{equation*}
such that, for any initial condition $(t,\mu) \in [0,T] \times {\mathcal P}_2({\mathbb R})$, 
the process 
$(Z^{t,\mu,k}(x))_{k \geq 0,x\in {\mathbb S}}$ in 
\eqref{eq:representation} 
satisfies 
\begin{equation*} 
 \int_t^T 
 {\mathbb P}
 \Bigl( \bigl\{ Z_s^{t,\mu,k}(x) \not =  \psi^k\bigl(s,x, {\rm Leb}_{\mathbb S} \circ X_s^{-1} 
 \bigr) \bigr\} 
 \Bigr) \ud s = 0. 
\end{equation*} 
\end{lem}

\begin{proof}
Although the result is standard in the finite-dimensional setting, proving it in the infinite-dimensional case is more delicate. 
The main idea is borrowed from 
\cite{ImkellerReveillacRichter}
and relies on Theorem 6.27 in 
\cite{CinlarJacodProtterSharpe}. 
There is here an additional (little) difficulty, which comes from the fact that the martingale representative term $Z^{t,\mu,k}_s(x)$ depends on 
the extra variable $x \in {\mathbb S}$. For that reason, we feel more appropriate to 
give a sketch of the proof, based on the 
ideas of \cite{CinlarJacodProtterSharpe}. 
\vskip 4pt

\textit{First Step.} 
We start with several notations.
Throughout the proof, we call $D_n$ the collection of dyadic numbers $\{j/2^n, j \in {\mathbb Z}\}$. 
Moreover, for any $(t,x,\mu) \in [0,T] \times {\mathbb S} \times {\mathcal P}_2({\mathbb R})$, we call ${\mathbb P}^{t,x,\mu}$ the probability 
measure on $\Omega = {\mathscr C}([0,T],[0,T]) \times 
{\mathscr C}([0,T],U^2({\mathbb S}))
\times {\mathscr C}([0,T],L^2({\mathbb S}))$ equipped with the law of 
the concatenated paths $(t,F_\mu^{-1},0)_{0 \leq r \leq t}$ and 
$(s,X_s,W_s-W_t)_{t \leq s \leq T}$, 
where $(X_s)_{t \leq s \leq T}$ solves 
\eqref{eq:V}
with $F_{\mu}^{-1}$ as initial condition at time $t$. 

We think that there is no ambiguity in denoting the canonical process on the space $\Omega$ 
by $(\tau_s,X_s,W_s)_{0 \leq s \leq T}$. 
It is very important to notice that, although $U^2({\mathbb S})$ is equipped with an $L^2$ norm, we can 
impose 
the process
$(X_s)_{0 \leq s \leq T}$ evaluated at $x$,
namely
$(X_s(x))_{0 \leq s \leq T}$, to be jointly measurable in $s$, $x$ and $\omega$ (the latter denoting here the generic element of the canonical space). 
This follows 
from Remark 
\ref{rem:measurability}. In particular, this allows us to define the following functional: 
\begin{equation*} 
Y_s(x) = {\mathcal U}\bigl(\tau_s ,x,\textrm{\rm Leb}_{\mathbb S}  \circ X_{\tau_s}^{-1} \bigr) + \int_{\tau_0}^{s \vee \tau_0}
 \partial_x f\bigl(X_r(x),\textrm{\rm Leb}_{\mathbb S} \circ X_r^{-1}\bigr) \ud r, \quad s \in [0,T]. 
\end{equation*} 

Lastly,
with a slight abuse of notation, we do as 
in 
\eqref{eq:representation}  
and we call 
$((Z_s^{t,\mu,k}(x))_{t \leq s \leq T})_{k \in {\mathbb N}}$ 
the processes obtained by representing the left-hand side in 
\eqref{eq:representation} 
as a sum of stochastic integrals under ${\mathbb P}^{t,x,\mu}$. 
\vskip 4pt

\textit{Second Step.} 
For $x \in {\mathbb S}$, for any two integers $k,n \geq 1$ and any two reals $0 \leq s_1 < s_2 \leq T$, we let 
\begin{equation*} 
\begin{split}
&C_{s_1,s_2}^{k,n}(x) :=  \sum_{r_i \in D_n, s_1 \leq r_i < r_{i+1} <s_2}
 \bigl( Y_{r_{i+1}}(x)- Y_{r_i}(x) \bigr) 
\bigl( B^k_{r_{i+1}}- B^k_{r_i} \bigr),
\end{split}
\end{equation*} 
Following the proof of Lemma 3.28 in 
\cite{CinlarJacodProtterSharpe}, 
we can construct an increasing sequence of integers $(n_p(t,x,\mu))_{p \geq 1}$, measurable w.r.t. $(t,x,\mu)$, such that, under each 
${\mathbb P}^{t,x,\mu}$, the limit
\begin{equation*} 
\lim_{p \rightarrow \infty} C_{s_1,s_2}^{k,n_p(t,x,\mu)}(x)
=
\lim_{p \rightarrow \infty} C_{s_1,s_2}^{k,n_p(\tau_0,x,{\rm Leb}_{\mathbb S} \circ X_0^{-1})}(x), \quad 0 \leq s_1 < s_2 \leq T,
\end{equation*} 
exists (pointwise) almost surely
and then coincides with the quadratic co-variation between $(Y_s(x))_{0 \leq s \leq T}$ 
and $(B^k_s)_{0 \leq s \leq T}$.   
We then define, for $s_1,s_2 \in [0,T]$, 
\begin{equation*} 
\begin{split}
&A^k_{s_1,s_2}(x) :=  \limsup_{p \rightarrow \infty} 
C_{s_1,s_2}^{k,n_p(\tau_0,x,{\rm Leb}_{\mathbb S} \circ X_0^{-1})}(x),
%; 
%\quad A^{k,-}_{s_1,s_2}(x) := \liminf_{p \rightarrow \infty} 
%C_{s_1,s_2}^{k,n_p(\tau_0,x,{\rm Leb}_{\mathbb S} \circ X_0^{-1})}(x). 
\end{split}
\end{equation*}
and
\begin{equation*}
Z^{k}_{s}(x) := \liminf_{h \searrow 0, h \in {\mathbb Q}} \frac1h  A^{k}_{s,s+h}(x), 
\quad s \in [0,T], 
\end{equation*}
with the limit existing in $[-\infty,\infty]$. 
With probability 1 under ${\mathbb P}^{t,x,\mu}$, 
%$(A^{k}_{s_1,s_2}(x))_{0 \leq s_1 < s_2 \leq T}$
%and 
%$(A^{k,-}_{s_1,s_2}(x))_{0 \leq s_1 < s_2 \leq T}$ coincide, and similarly 
%for $(Z^{k,+}_s(x))_{0 \leq s \leq T}$ and 
%$(Z^{k,-}_s(x))_{0 \leq s \leq T}$.
%%We also put
%%\begin{equation*} 
%%Z_s(t,x) := \Bigl( Z_s^+(t,x) - Z_s^-(t,x) \Bigr) , \quad s \in [0,T]
%%\end{equation*} 
%
%
%
%In fact, under ${\mathbb P}_{t,x,\mu}$, 
%we  know that 
%$(A_{s_1,s_2}(t,x):=A_{s_1,s_2}^+(t,x) - A_{s_1,s_2}^- (t,x))_{0 \leq s_1,s_2 \leq T}$ is in fact additive and has finite variation. 
%Moreover, with probability 1 under ${\mathbb P}_{t,x,\mu}$,  
\begin{equation*} 
A^{k}_{s_1,s_2}(x)
= \int_{s_1}^{s_2} Z^{k}_r(x) \ud r, \quad t \leq s_1 < s_2 \leq T.  
\end{equation*}  
In particular,  
with probability 1 under 
${\mathbb P}^{t,x,\mu}$,
\begin{equation*} 
\int_t^T 
\vert Z_{r}^{k}(x)
- 
Z^{t,\mu,k}_r (x)
\vert \ud r
=0. 
\end{equation*} 
\vskip 4pt

\textit{Third Step.} 
We observe that $Z_{t}^{k}(x)$ is $\cap_{s >t} \sigma((\tau_r,X_r,W_r), \ 0 \le r \leq s)$--measurable.
Blumenthal 0-1 law says that, 
under ${\mathbb P}^{t,x,\mu}$, 
 $Z_{t}^{k}(x)$ is almost surely constant. And then 
 \begin{equation*} 
 {\mathbb P}^{t,x,\mu} \Bigl( \bigl\{ \vert Z_{t}^{k}(x) 
 \vert  = {\mathbb E}^{t,x,\mu} \bigl( \vert Z_{t}^{k}(x)
 \vert 
 \bigr)\bigr\} \Bigr) = 1. 
 \end{equation*} 
Letting 
\begin{equation*} 
\begin{split}
&\phi^k(t,x,\mu) :=  {\mathbb E}^{t,x,\mu} \bigl( \vert Z_{t}^{k}(x) \vert \bigr), 
\\
&\psi^k(t,x,\mu) :=  {\mathbb E}^{t,x,\mu} \bigl( Z_{t}^{k}(x)  \bigr) {\mathbf 1}_{\{ \phi^k(t,x,\mu) < \infty\}}, 
\end{split}
\end{equation*}
this yields
 \begin{equation*} 
 {\mathbb P}^{t,x,\mu} \Bigl( \bigl\{ Z_{t}^{k}(x) = \psi^k(t,x,\mu)\bigr\} \Bigr) = 1,
 \end{equation*}  
 when $\phi^k(t,x,\mu) < \infty$. 
 Assume for a while that the functions $\varphi^k$ and $\psi^k$ are jointly measurable (in $(t,x,\mu)$). Then, 
for any $s \in [t,T]$, 
$ {\mathbb P}^{t,x,\mu}$ almost surely, 
 \begin{equation*} 
 {\mathbb P}^{s,x,{\rm Leb}_{\mathbb S}\circ X_s^{-1}} \Bigl( \bigl\{ Z_{s}^{k}(x) =   \psi^k(s,x,
 {\rm Leb}_{\mathbb S}\circ X_s^{-1}
 )\bigr\} \Bigr) = 1, 
 \end{equation*} 
 when ${\mathbb E}^{s,x,{\rm Leb}_{\mathbb S}\circ X_s^{-1}}(\vert Z_s^k(x) \vert) < \infty$. 

By the conclusion of the second step, we have 
\begin{equation*} 
{\mathbb E}^{t,x,\mu} \int_t^T 
{\mathbb E}^{s,x,{\rm Leb}_{\mathbb S}\circ X_s^{-1}}(\vert Z_s^k(x) \vert)
\ud s 
= 
{\mathbb E}^{t,x,\mu} \int_t^T 
\vert Z_s^k(x) \vert
\ud s < \infty.
\end{equation*}
And then, 
this 
implies that the condition 
${\mathbb E}^{s,x,{\rm Leb}_{\mathbb S}\circ X_s^{-1}}(\vert Z_s^k(x) \vert) < \infty$ 
is satisfied almost everywhere under $\ud s \otimes \ud {\mathbb P}^{t,x,\mu}$. 
This
implies the claim (notice that it suffices to identify the integrand $Z^k(x)$
when the problem is set on the canonical space). 

It remains to check that the functions $\phi^k$ and  $\psi^k$ are measurable. This follows from a monotone class argument. Indeed, we observe that, for any 
subset $A$ in the Borel $\sigma$-field on $[0,T] \times {\mathbb R} \times {\mathcal P}_2({\mathbb R})$
and any $B$ in the Borel $\sigma$-field on $\Omega$, the mapping 
\begin{equation*} 
(t,x,\mu) \mapsto  \int_{\Omega} 
 {\mathbf 1}_A(t,x,\mu) {\mathbf 1}_B(\omega) 
 \ud 
 {\mathbb P}^{t,x,\mu} (\omega) 
= 
 {\mathbf 1}_A(t,x,\mu) 
{\mathbb P}^{t,x,\mu} (B)  
\end{equation*} 
is measurable. In particular, if we call ${\mathscr L}$ the collection of Borel
subsets $C$ 
of $[0,T] \times {\mathbb R} \times {\mathcal P}_2({\mathbb R}) \times \Omega$ such that 
\begin{equation*} 
(t,x,\mu) \mapsto  
{\mathbb P}^{t,x,\mu} (C)  
\end{equation*} 
is measurable, 
${\mathscr L}$ contains the Cartesian product of the two 
Borel $\sigma$-fields on the two spaces 
$[0,T] \times {\mathbb R} \times {\mathcal P}_2({\mathbb R})$ and 
$\Omega$. Now, it is clear that ${\mathscr L}$ is stable by monotone limits. This shows that 
${\mathscr L}$ is equal to the tensorial product  of the two 
Borel $\sigma$-fields on the two spaces 
$[0,T] \times {\mathbb R} \times {\mathcal P}_2({\mathbb R})$ and 
$\Omega$. Measurability of 
$\phi^k$ and $\psi^k$ easily follows.  
\end{proof} 

The purpose of the next lemma is to estimate the
functions $(\psi^k)_{k \in {\mathbb N}}$ in 
Lemma 
\ref{lem:2:5}. 
\begin{lem}
\label{lem:2:4}
With the same notation as in
\eqref{eq:representation}
and
Lemma 
\ref{lem:martingale}, we have, 
for every initial condition $(t,\mu)  \in [0,T] \times {\mathcal P}_2({\mathbb R})$ to \eqref{eq:V}, 
and every function $\theta \in L^2({\mathbb S})$, 
\begin{equation*}
\biggl\{ 
\sum_{k \in {\mathbb N}} 
\biggl( 
(1 \vee k)^{2 \lambda} 
\Bigl\vert 
\int_{\mathbb S} \theta(x) \psi^k(t,x,\mu) \ud x \Bigr\vert^2 \biggr)
\biggr\}^{1/2} \leq 
\sup_{t \leq s \leq T}
{\rm Lip} \biggl[ \int_{\mathbb S} \theta(x)  \, {\mathcal U}(s,x,\cdot)  \ud x \biggr],
\end{equation*} 
where 
\begin{equation*}
{\rm Lip} \biggl[ \int_{\mathbb S} \theta(x)  \, {\mathcal U}(s,x,\cdot)  \ud x \biggr]
:=
\sup_{\mu,\nu \in {\mathcal P}_2({\mathbb R}) : \mu \not = \nu}
{\mathbb W}_2(\mu,\nu)^{-1} \biggl\vert \int_{\mathbb S} 
\theta(x) \Bigl(  
 {\mathcal U}(s,x,\mu)
-
 {\mathcal U}(s,x,\nu)
 \Bigr) 
 \ud x 
\biggr\vert. 
\end{equation*} 
\end{lem} 

By 
item $(3)$ in 
Definition \ref{def:mathcalC}, we observe that the right-hand side in the last display of the 
statement is less than $C \| \theta \|_2$.

\begin{proof}
Throughout the proof, the initial condition $(t,\mu)$ is fixed. 
Recalling 
\eqref{eq:representation}, we merely write $(Z^j(x))_{j \in {\mathbb N}}$ for $(Z^{t,\mu,j}(x))_{j \in {\mathbb N}}$. 
\vskip 4pt 

\textit{First Step.}
For given $\beta \in {\mathbb R}$ and an infinite collection 
$(\ell_k)_{k \in {\mathbb N}} \in {\mathbb R}^{\mathbb N}$, 
with $\sum_{k \in {\mathbb N}} \ell_k^2 =1$
and with 
$\sum_{k \in {\mathbb N}} \ell_k^2 k^{2 \lambda}  < \infty$, 
and for $h \in [0,T-t)$, 
we consider the probability measure 
\begin{equation*} 
{\mathbb Q}^{\beta} := 
\exp \biggl( - \beta \sum_{k \in {\mathbb N}} \ell_k (1 \vee k)^\lambda \bigl( B_{t+h}^k - B_t ^k\bigr) - \frac{\beta^2}2 h
\sum_{k \in {\mathbb N}}\ell_k^2 (1 \vee k)^{2 \lambda} \biggr) \cdot {\mathbb P}. 
\end{equation*}
We know that, under ${\mathbb Q}^{\beta}$, the collection of processes
\begin{equation*} 
\Bigl( \tilde{B}_s^k := B_s^k + \beta \ell_k (1 \vee k)^\lambda \bigl(s \wedge (t+h)
-t 
\bigr)_{t \leq s \leq T} \Bigr)_{k \in {\mathbb N}} 
\end{equation*}
are independent Brownian motions. 

Fix $x \in {\mathbb S}$. 
Recalling that 
\begin{equation} 
\label{eq:U:martingale}
\begin{split}
{\mathcal U}\bigl(s,x,{\rm Leb}_{\mathbb S} \circ X_s^{-1} \bigr) 
&= {\mathcal U}(t,x,\mu) 
-
 \int_t^s
\partial_x f \bigl( X_r(x),{\rm Leb}_{\mathbb S} \circ X_r^{-1} \bigr)
\ud r
\\
&\hspace{15pt} + \sum_{k \in {\mathbb N}} 
\int_t^s Z_r^k(x) \ud B_r^k, \quad s \in [t,T], 
\end{split} 
\end{equation} 
we obtain 
\begin{equation*} 
\begin{split} 
{\mathcal U}\bigl(s,x,{\rm Leb}_{\mathbb S} \circ X_s^{-1} \bigr) 
&= {\mathcal U}(t,x,\mu) 
-
 \int_t^s
\partial_x f \bigl( X_r(x),{\rm Leb}_{\mathbb S} \circ X_r^{-1} \bigr)
\ud r
\\
&\hspace{5pt} - 
 \sum_{k \in {\mathbb N}}  \beta \int_t^s \ell_k (1 \vee k)^\lambda Z_r^k(x) \ud r
+
 \sum_{k \in {\mathbb N}} 
\int_t^s Z_r^k(x) \ud \tilde B_r^k, \quad s \in [t,t+h]. 
\end{split} 
\end{equation*}
 And then, choosing $s=t+h$ and taking expectation under ${\mathbb Q}^{\beta}$, 
 \begin{equation} 
 \label{eq:qbeta}
\begin{split}
& {\mathbb E}^{{\mathbb Q}^{\beta}} 
 \Bigl[
 {\mathcal U}\bigl(t+h,x,{\rm Leb}_{\mathbb S} \circ X_{t+h}^{-1} \bigr) 
-  {\mathcal U}(t,x,\mu) \Bigr] 
\\
&=
-
 {\mathbb E}^{{\mathbb Q}^{\beta}}  \int_t^{t+h}
\partial_x f \bigl( X_r(x),{\rm Leb}_{\mathbb S} \circ X_r^{-1} \bigr)
\ud r
 - \beta  \sum_{k \in {\mathbb N}} {\mathbb E}^{{\mathbb Q}^{\beta}} \int_t^{t+h} 
 \ell_k ( 1 \vee k )^\lambda
Z_r^k(x) \ud r.
\end{split} 
 \end{equation}
Next, we observe that, under ${\mathbb Q}^{\beta}$, 
\begin{equation*} 
\begin{split}
\ud X_s(x) &= -  {\mathcal V}\bigl(s,x,{\rm Leb}_{\mathbb S} \circ X_s^{-1} 
\bigr) \ud s 
+ \Delta X_s(x) \ud s -  \beta  \sum_{k \in {\mathbb N}}\ell_k e_k(x) \ud s
\\
&\hspace{15pt} + \sum_{k \in {\mathbb N}} \frac{1}{(1 \vee k)^\lambda} e_k(x) \ud \tilde B^k_s 
+   \ud \eta_s(x), \quad x \in {\mathbb S}, \quad s \in [t,t+h]. 
\end{split}
\end{equation*} 
The law of $(X_s(\cdot))_{t \leq s \leq t+h}$ under ${\mathbb Q}^\beta$ 
is the same as the law of 
$(X_s^\beta(\cdot))_{t \leq s \leq t+h}$ under ${\mathbb P}$, where 
\begin{equation*} 
\begin{split}
\ud X_s^\beta(x) &= - {\mathcal V}\bigl(s,x,{\rm Leb}_{\mathbb S} \circ (X^\beta)_s^{-1} 
\bigr) \ud s 
+ \Delta X_s^\beta(x) \ud s -  \beta \sum_{k \in {\mathbb N}} \ell_k e_k(x) \ud s
\\
&\hspace{15pt} + \sum_{k \in {\mathbb N}} \frac{1}{(1 \vee k)^\lambda} e_k(x) \ud   B^k_s 
+   \ud \eta^\beta_s(x), \quad x \in {\mathbb S}, \quad s \in [t,t+h], 
\end{split}
\end{equation*} 
with $X_t^{\beta}=X_t$ as initial condition
(uniqueness to the above equation can be shown as in 
\eqref{eq:uniqueness:1}). 
Therefore, 
by 
\eqref{eq:qbeta}, 
\begin{equation*} 
\begin{split}
& {\mathbb E}  
 \Bigl[ 
{\mathcal  U}\bigl(t+h,x,{\rm Leb}_{\mathbb S} \circ (X^\beta_{t+h})^{-1} \bigr) 
- {\mathcal  U}(t,x,\mu) \Bigr] 
\\
&= 
-
 {\mathbb E}  \int_t^{t+h} 
\partial_x f \bigl( X_r^\beta(x),{\rm Leb}_{\mathbb S} \circ (X_r^\beta)^{-1} \bigr)
\ud r
- \beta  \sum_{k \in {\mathbb N}} {\mathbb E}^{{\mathbb Q}^{\beta}} \int_t^{t+h} 
\ell_k (1 \vee k)^\lambda
Z_r^k(x) \ud r.
\end{split} 
\end{equation*} 
By the property 
\eqref{eq:U:martingale}, 
\begin{equation*} 
\begin{split}
 &{\mathbb E}  
 \Bigl[
 {\mathcal U}\bigl(t+h,x,{\rm Leb}_{\mathbb S} \circ (X^\beta_{t+h})^{-1} \bigr) 
-  {\mathcal U} \bigl(t+h,x,{\rm Leb}_{\mathbb S} \circ (X_{t+h})^{-1} \bigr)
  \Bigr]
  \\
  & =
  -
 {\mathbb E}  \int_t^{t+h} 
 \Bigl[
\partial_x f \bigl( X_r^\beta(x),{\rm Leb}_{\mathbb S} \circ (X_r^\beta)^{-1} \bigr)
-
\partial_x f \bigl( X_r (x),{\rm Leb}_{\mathbb S} \circ X_r^{-1} \bigr)
\Bigr]
\ud r 
\\
&\hspace{15pt}  - \beta  \sum_{k \in {\mathbb N}} {\mathbb E}^{{\mathbb Q}^{\beta}} \int_t^{t+h} 
\ell_k (1 \vee k)^\lambda
Z_r^k(x)  \ud r.
\end{split} 
\end{equation*}

\textit{Second Step.} Reproducing the stability analysis carried out in 
\eqref{eq:uniqueness:1}
and 
in Lemma
\ref{lem:U}, we 
can get an estimate for 
$\| X_{t+h}^{\beta} - X_{t+h} \|_2^2$ 
and then deduce 
by integrating in $x$ the above identity
that, for any smooth function 
$\theta : {\mathbb S} \rightarrow {\mathbb R}$, 
\begin{align}
%\biggl( \int_{\mathbb S} 
\biggl\vert \beta 
\sum_{k \in {\mathbb N}} \ell_k (1 \vee k)^\lambda
{\mathbb E}^{{\mathbb Q}^{\beta}} \int_t^{t+h} 
\biggl( 
\int_{\mathbb S} \theta(x)  Z_r^k(x)   \ud x \biggr)  \ud r 
\biggr\vert   %\biggr)^{1/2}  \leq 
&\leq
{\rm Lip} %_{L^2({\mathbb S})}  
\biggl[ \int_{\mathbb S} \theta(x)  \, {\mathcal U}(t+h,x,\cdot)  \ud x \biggr] 
e^{C_{\mathcal V} h}  \vert \beta \vert h 
\nonumber
\\
&\hspace{15pt} + C_{\mathcal V}   \vert \beta \vert \, \| \theta \|_2 h^2,
\label{eq:beta:qbeta} 
\end{align} 
where $C_{\mathcal V}$ depends on the Lipschitz constant of ${\mathcal V}$. 
The key point to get the above formula is to repeat 
\eqref{eq:uniqueness:1}, with $X'$ replaced by 
$X^\beta$ and to use the bound
\begin{equation*} 
\begin{split} 
\| X_{s} - X_{s}^\beta \|_2^2 
&\leq  
C_{\mathcal V} 
\int_t^{s} 
\| X_r - X_r^\beta \|_2^2 \ud r
+  2 \beta \sum_{k \in {\mathbb N}} 
\int_t^{s} \ell_k \langle X_r - X_r^{\beta}, e_k \rangle_2 \ud r
\\
&\leq 2 \vert \beta \vert e^{C_{\mathcal V} h} 
\int_t^{s}  \| X_r - X_r^\beta \|_2 \ud r, 
\end{split} 
\end{equation*} 
where we used the fact that $\sum_{k \in {\mathbb N}} \ell_k^2 =1$. 
We easily deduce that 
$\| X_{s} - X_{s}^\beta \|_2  \leq e^{C_{\mathcal V} h} \vert \beta \vert h$
for $s \in [t,t+h]$. 
Dividing by 
$\beta$ on both sides in 
\eqref{eq:beta:qbeta}
and then letting $\beta$ tend to $0$, we deduce that 
\begin{equation*} 
%\biggl( \int_{\mathbb S}
 \begin{split} 
 \biggl\vert \sum_{k \in {\mathbb N}} \ell_k
(1 \vee k )^{\lambda}  
   {\mathbb E}  \int_t^{t+h} 
\biggl( \int_{\mathbb S} \theta(x) Z_r^k(x)  \ud x \biggr) \, \ud r
\biggr\vert
%^2 \ud x \biggr)^{1/2} 
&\leq
{\rm Lip} \biggl[ 
 %_{L^2({\mathbb S})} 
 \int_{\mathbb S} \theta(x) \,   {\mathcal U}(t+h,x,\cdot) \ud x \biggr] e^{C_{\mathcal V} h} h
 \\
&\hspace{15pt}   + C_{\mathcal V}    \| \theta \|_2 h^2.
\end{split} 
\end{equation*} 
%Here, the reader may be tempted to derive an $x$-wise bound for 
%${\mathbb E} \int_t^{t+h} Z_r^k(x)  \ud r $. Instead, we prefer to work with integrated quantities because, at this stage, we just control 
%the Lipschitz constant of ${\mathcal U}(t+h,\cdot,\cdot)$, seen as a function from 
%$U^2({\mathbb S})$ to $L^2({\mathbb S})$.

We then argue as in the proof of Lemma
\ref{lem:martingale}.
Assuming that we are working on the canonical space $\Omega={\mathscr C}([0,T] ,L^2({\mathbb S}))$ equipped with the 
law ${\mathbb P}$ of a 
$Q$-Brownian motion, we consider 
a regular conditional probability $({\mathbb P}_{s,\omega})_{\omega \in \Omega}$ of ${\mathbb P}$ given the $\sigma$-field generated by 
$\sigma(W(r,\cdot) ; r \leq s)$ ($W$ here standing for the canonical process). Then, $\omega$ ${\mathbb P}$-almost surely, %under ${\mathbb P}_{s,\omega}$, 
\begin{equation*} 
\begin{split}
%\biggl( \int_{\mathbb S} 
\biggl\vert   \sum_{k \in {\mathbb N}} \ell_k
(1 \vee k )^{\lambda}   {\mathbb E}^{{\mathbb P}_{s,\omega}}  \int_s^{s+h} 
\biggl( \int_{\mathbb S} 
\theta(x) 
Z_r^k(x) \, \ud x \biggr) \, \ud r
\biggr\vert %^2 \ud x \biggr)^{1/2} 
&\leq {\rm Lip}
\biggl[
\int_{\mathbb S}
\theta(x)   \, %_{L^2({\mathbb S})} 
 {\mathcal U}(s+h,x,\cdot) \ud x  \biggr]
e^{C_{\mathcal V} h} h  \\
&\hspace{15pt} + C_{\mathcal V}   \| \theta \|_2 h^2.
\end{split}  
\end{equation*} 
\vskip 4pt

\textit{Third Step.}
Consider now a 
mesh $\{t=t_0<t_1<\cdots<t_n=T\}$
of $[t,T]$ of step $h$ together with 
a 
simple process $H : [t,T] \times  \Omega \rightarrow {\mathbb R}$ such that 
\begin{equation*} 
H_s = H_{t_i}, \quad s \in [t_i,t_{i+1}),
\end{equation*}
and $H_{t_i} :   \Omega \rightarrow {\mathbb R}$ is bounded and ${\mathcal F}_{t_i}/{\mathcal B}({\mathbb R})$-measurable. 
Then, 
the conclusion of the third step yields
\begin{equation*} 
\begin{split}
&\biggl\vert  \sum_{k \in {\mathbb N}} \ell_k
(1 \vee k )^{\lambda}
 {\mathbb E}^{{\mathbb P}_{t_i,\omega}}  \int_{t_i}^{t_{i+1}} 
\biggl[ H_{t_i} \biggl( \int_{\mathbb S}  \theta(x) Z_r^k(x) \, \ud x \biggr) \biggr] \ud r 
\biggr\vert 
\\
& \leq
\vert H_{t_i} \vert \, 
 {\rm Lip}
\biggl[
\int_{\mathbb S}
\theta(x) \,  %_{L^2({\mathbb S})} 
 {\mathcal U}(s+h,x,\cdot) \ud x  \biggr]
e^{C_{\mathcal V} h} h   + C_{\mathcal V}   \| \theta \|_2 h^2.
\end{split}
\end{equation*} 
Taking expectation and then summing over $i$, 
we deduce that 
\begin{equation*}
\begin{split} 
&\biggl\vert   
 \sum_{k \in {\mathbb N}} \ell_k
(1 \vee k )^{\lambda}
{\mathbb E}
  \int_t^T
\biggl[
  H_r
  \biggl( \int_{\mathbb S} 
\theta(x) 
Z_r^k(x) \ud x
\biggr)
\biggr] 
\ud r
\biggr\vert 
% \int_{\mathbb S} 
%\biggl\vert   {\mathbb E}
%\biggl[ 
%  \int_t^T
%  H_r 
%Z_r^k(x) \ud r  
%\biggr]
%\biggr\vert 
%\ud x
\\
&\leq 
\sup_{t \leq s \leq T}
 {\rm Lip}
\biggl[
\int_{\mathbb S}
\theta(x)  \, %_{L^2({\mathbb S})} 
 {\mathcal U}(s,x,\cdot) \ud x  \biggr]
e^{C_{\mathcal V} h}  
{\mathbb E} 
\int_t^T
 \vert H_r \vert \, \ud r    + C_{\mathcal V}   \| \theta \|_2 T h.
\end{split}
\end{equation*} 
By a standard result of approximation of progressively measurable processes, 
the above remains true for any 
${\mathbb F}$-progressively measurable process 
$(H_r)_{t \leq r \leq T}$
satisfying 
\begin{equation*} 
{\mathbb E} \int_t^T  \vert H_r \vert^2  \, \ud r < \infty,
\end{equation*}
where we used implicitly the fact that 
(see \eqref{eq:representation})
\begin{equation*} 
\sum_{k \in {\mathbb N}} 
{\mathbb E} 
\int_t^T
\biggl(  \int_{\mathbb S} \vert Z_r^k(x) \vert^2 \ud x \biggr) \ud r
< \infty. 
\end{equation*} 
Letting $h$ tend to $0$, 
we get 
\begin{equation*}
\begin{split} 
&\biggl\vert 
{\mathbb E}
  \int_t^T
\biggl[
  H_r
  \biggl( \int_{\mathbb S} 
\theta(x) 
\biggl\{ 
 \sum_{k \in {\mathbb N}} \ell_k
(1 \vee k )^{\lambda}  
Z_r^k(x) \biggr\} \ud x
\biggr)
\biggr] 
\ud r
\biggr\vert 
\\
&=
\biggl\vert 
 \sum_{k \in {\mathbb N}} \ell_k
(1 \vee k )^{\lambda}  
{\mathbb E}
  \int_t^T
\biggl[
  H_r
  \biggl( \int_{\mathbb S} 
\theta(x) 
Z_r^k(x) \ud x
\biggr)
\biggr] 
\ud r
\biggr\vert 
% \int_{\mathbb S} 
%\biggl\vert   {\mathbb E}
%\biggl[ 
%  \int_t^T
%  H_r 
%Z_r^k(x) \ud r  
%\biggr]
%\biggr\vert 
%\ud x
\\
& \leq 
\sup_{t \leq s \leq T}
 {\rm Lip}
\biggl[
\int_{\mathbb S}
\theta(x)  \, %_{L^2({\mathbb S})} 
 {\mathcal U}(s,x,\cdot)  \ud x  \biggr]
{\mathbb E} 
\int_t^T
 \vert H_r \vert \, \ud r.  
\end{split}
\end{equation*} 
Writing the integral over ${\mathbb S}$ in the left-hand side as 
the inner product 
$\langle \theta,
\sum_{k \in {\mathbb N}} 
\ell_k (1 \vee k )^\lambda
Z_r^k \rangle_2$ in $L^2({\mathbb S})$, we deduce that 
\begin{equation*} 
\int_t^T
{\mathbb P} 
\Bigl( \Bigl\{ 
\Bigl\vert
\bigl\langle \theta, \sum_{k \in {\mathbb N}} 
\ell_k (1 \vee k )^\lambda
Z_r^k \bigr\rangle_2 
\Bigr\vert
> R(\theta)
\Bigr\}
\Bigr) 
 \ud r = 0, 
\end{equation*}
with 
\begin{equation*} 
R(\theta) := \sup_{t \leq s \leq T}  {\rm Lip} \biggl[ \int_{\mathbb S} \theta(x) \, {\mathcal U}(s,x,\cdot)  \ud x \biggr].
\end{equation*}
Back to Lemma 
\ref{lem:2:5}, this says that we can replace $\psi^k(t,x,\mu)$ by 
$$\psi^k(t,x,\mu) 
\prod_{p \in {\mathbb N}} 
\prod_{q \in {\mathbb N}} 
{\mathbf 1}_{ \Bigl\{
\bigl\vert 
\sum_{j \in {\mathbb N}} 
\ell_j^{(p)} (1 \vee j )^\lambda
 \langle \psi^j(t,\cdot,\mu), \theta^{(q)} \rangle_2
 \bigr\vert  \leq R(\theta^{(q)})\Bigr\}},$$ 
for countably many $(\theta^{(q)})_{q \in {\mathbb N}}$'s
and countably many sequences $(\ell^{(p)})_{p \in {\mathbb N}}$'s. 
By a separability argument, we then have 
\begin{equation*} 
\Bigl\vert 
\sum_{k \in {\mathbb N}}
\ell_k ( 1 \vee k)^\lambda 
\langle \psi^k(t,\cdot,\mu),\theta \rangle_2 \Bigr\vert \leq R(\theta),
\end{equation*} 
for all $\theta \in L^2({\mathbb S})$ and all $(\ell_k)_{k \in {\mathbb N}}$
with 
$\sum_{k \in {\mathbb N}} \ell_k^2 =1$
and 
$\sum_{k \in {\mathbb N}} (1 \vee k)^{2\lambda} \ell_k^2 < \infty$.
By a new approximation argument, the last condition can be dropped.  
This completes the proof. 
\end{proof}

\subsection{Analysis over an arbitrary fixed time duration} 
\label{subse:3:3}
In order to extend existence and uniqueness from small to arbitrary time durations, one needs to provide an \textit{a priori} bound for fixed points of 
$\Phi$. 
In order to clarify the analysis, we denote, for any $t \in (0,T)$, by ${\mathcal C}_t$ the space of functions ${\mathcal V}:[t,T] \times {\mathbb S} \times {\mathcal P}_2({\mathbb R}) \rightarrow {\mathbb R}$ 
that satisfy Definition 
\ref{def:mathcalC} but on the time interval $[t,T]$. With a slight abuse of notation, we then write 
$\Phi({\mathcal V})$ for the image of ${\mathcal V}$ by the restriction of the mapping $\Phi$ 
to the interval $[t,T]$.

\begin{lem}
\label{lem:phi}
Let 
$t \in (0,T)$ and
${\mathcal V}$ be an element of ${\mathcal C}_t$ 
satisfying 
\begin{equation*}
\bigl\| 
 \nabla_x {\mathcal V}\bigl(s,\cdot,{\rm Leb}_{\mathbb S} \circ X^{-1} 
 \bigr) 
 \bigr\|_2^2 \leq C + \frac{C}{T-s} \| X \|_2^2, \quad X \in U^2({\mathbb S}), 
 \quad s \in [t,T),
\end{equation*}
for some constant $C \geq 0$. 

Then,
for any $\theta \in L^2({\mathbb S})$, 
the function $(t,\mu) \mapsto 
\langle \theta,{\mathcal U}(t,\cdot,\mu) \rangle_2$, with ${\mathcal U} :=\Phi({\mathcal V})$, can be represented   as 
\begin{equation*} 
\begin{split}
&\bigl\langle \theta,{\mathcal U}(t,\cdot,\mu)\bigr\rangle_2 = 
{\mathbb E} \biggl[ 
\int_{\mathbb S} 
\theta(x) \partial_x g \Bigl( X_T^0(x), {\rm Leb}_{\mathbb S} \circ 
(X_T^0)^{-1} \Bigr) \ud x
\\
&\hspace{5pt} + \int_t^T \biggl( \int_{\mathbb S} 
\theta(x) 
\partial_x f \Bigl( X_s^0(x), {\rm Leb}_{\mathbb S} \circ 
(X_s^0)^{-1} \Bigr) \ud x \biggr) \ud s
\\
&\hspace{5pt} +
\int_t^T 
\int_{\mathbb S} 
{\mathcal V} \bigl( s, y, 
{\rm Leb}_{\mathbb S} \circ 
(X_s^0)^{-1}
\bigr) 
\Psi^\theta  \bigl(s,y,
{\rm Leb}_{\mathbb S} \circ 
(X_s^0)^{-1} 
 \bigr) \ud y \, \ud s 
 \biggr], 
 \quad (t,\mu) \in [0,T] \times {\mathcal P}_2({\mathbb R}), 
 \end{split} 
\end{equation*} 
for a %universal constant $c>0$ (independent of $C$) and a 
measurable function $\Psi^\theta : [t,T] \times  {\mathbb S} \times 
{\mathcal P}_2({\mathbb R}) \rightarrow 
{\mathbb R}$ satisfying
\begin{equation*} 
\int_{\mathbb S}
\Bigl\vert 
  \Psi^\theta \bigl(s,y,{\rm Leb}_{\mathbb S} \circ 
X^{-1} \bigr) 
\Bigr\vert^2 \ud y 
 \leq 
\sup_{t \leq s \leq T}
{\rm Lip} \biggl[ \int_{\mathbb S} \theta(x)  \, {\mathcal U}(s,x,\cdot)  \ud x \biggr]^2, \quad (s,X) \in [t,T] \times U^2({\mathbb S}). 
\end{equation*} 
We recall that, above, $(X^0_s)_{t \leq s \leq T}$ is the solution to \eqref{eq:V} with $0$-drift and  
$(t,\mu)$ as initial condition and 
that ${\rm Lip}$
is as in the statement of 
Lemma \ref{lem:2:4}. 
\end{lem}

\begin{proof}
\textit{First Step.}
Back to \eqref{eq:V} (with $(t,\mu)$ as initial condition), we write 
\begin{equation*} 
\begin{split}
{\mathcal V}\bigl(s,x,\text{\rm Leb}_{\mathbb S} \circ X_s^{-1} \bigr) \ud s 
 + \ud W_s(x)
& = \sum_{k \in {\mathbb N}} 
\frac{ e_k(x) 
}{(1 \vee k)^\lambda} \Bigl( (1 \vee k)^\lambda \widehat{\mathcal V}^k(s, \text{\rm Leb}_{\mathbb S} \circ X_s^{-1} ) 
\ud s
 + \ud B_s^k \Bigr)
 \\
 &=: \sum_{k \in {\mathbb N}} 
 e_k(x) \ud \widetilde{B}_s^k,
 \end{split}
\end{equation*} 
where $\widehat{\mathcal V}^k(s,\mu)$ is the Fourier coefficient number $k$ of ${\mathcal V}(s,\cdot,\mu)$. 
This Fourier coefficient can be estimated in two ways:
\begin{equation*}
\begin{split}
&\bigl\vert \widehat{\mathcal V}^k(s,\mu) \bigr| \leq  \| {\mathcal V} \|_\infty, \quad k \in {\mathbb N}, 
\\
&\bigl\vert \widehat{\mathcal V}^k(s,\mu) \bigr| \leq \frac{1}{k} \bigl| \widehat{\nabla {\mathcal V}}^k(s,\mu) \bigr|, \quad k \geq 1. 
\end{split} 
\end{equation*}And then, we can find a constant $C$ (possibly different from the one in the statement) such that 
\begin{equation*} 
\begin{split}
 \sum_{k \in {\mathbb N}} 
\bigl\vert (1 \vee k)^\lambda \widehat{\mathcal V}^k(s, \text{\rm Leb}_{\mathbb S} \circ X_s^{-1} )
\bigr\vert^2
\leq C + C \sum_{k \geq 1} \frac{k^{2\lambda}}{k^{1+\delta}} 
\bigl\vert
 \widehat{\nabla {\mathcal V}}^k(s, \text{\rm Leb}_{\mathbb S} \circ X_s^{-1} )
\bigr\vert^{1+\delta},
\end{split}
\end{equation*} 
where $\delta$ is a parameter in $(0,1)$ whose value is fixed next.   
By Cauchy-Schwarz inequality, we obtain 
\begin{equation*} 
\begin{split}
&\sum_{k \in {\mathbb N}} 
\bigl\vert (1 \vee k)^\lambda \widehat{\mathcal V}^k(s, \text{\rm Leb}_{\mathbb S} \circ X_s^{-1} )
\bigr\vert^2
\\
&\leq C + C 
\biggl( \sum_{k \geq 1} 
\bigl[ \frac{k^{2\lambda}}{k^{1+\delta}} 
\bigr]^{2/(1-\delta)}
 \biggr)^{(1-\delta)/2}
\biggl( \sum_{k \geq 1} 
\bigl\vert
 \widehat{\nabla {\mathcal V}}^k(s, \text{\rm Leb}_{\mathbb S} \circ X_s^{-1} )
\bigr\vert^{2} \biggr)^{(1+\delta)/2}.
\end{split}
\end{equation*}
We observe that $2 \lambda <2$ and then we can choose $\delta \in (0,1)$ such that $1+\delta > 2 \lambda$. 
This says that $2\lambda - (1+\delta) < 0$.
Since $2/(1-\delta)$ tends to $\infty$ as $\delta$ tends to $1$, we can increase the value of $\delta$
(still in $(0,1)$) so that 
$2[2\lambda - (1+\delta)]/(1-\delta) < -1$.  
We obtain (for a new value of $C$ depending on $\delta$) 
\begin{equation*}
\begin{split} 
 \sum_{k \in {\mathbb N}} 
\bigl\vert (1 \vee k)^\lambda \widehat{\mathcal V}^k(s, \text{\rm Leb}_{\mathbb S} \circ X_s^{-1} )
\bigr\vert^2
\ud s 
&\leq C + C 
\biggl( \sum_{k \geq 1} 
\bigl\vert
 \widehat{\nabla {\mathcal V}}^k(s, \text{\rm Leb}_{\mathbb S} \circ X_s^{-1} )
\bigr\vert^{2} \biggr)^{(1+\delta)/2}
\\
&= C + C  
\bigl\|
 \widehat{\nabla {\mathcal V}}(s, \text{\rm Leb}_{\mathbb S} \circ X_s^{-1} )
\bigr\|_2^{1+\delta}
\\
&\leq C + \frac{\| X_s \|_2^{1+\delta}}{(T-s)^{(1+\delta)/2}},
\end{split}
\end{equation*} 
with the last line following from 
the assumption in the statement (of Lemma \ref{lem:phi}). 
And then, 
 \begin{equation}
  \label{eq:enegery:girsanov}
\begin{split}
\int_0^T  \sum_{k \in {\mathbb N}} 
\bigl\vert (1 \vee k)^\lambda \widehat{\mathcal V}^k(s, \text{\rm Leb}_{\mathbb S} \circ X_s^{-1} )
\bigr\vert^2
\ud s 
&\leq C  \Bigl( 1+ \sup_{0 \leq s \leq T} \| X_s \|_2^{1+\delta} \Bigr).
\end{split}
\end{equation} 
By \eqref{eq:main:existence:uniqueness}, the right-hand has finite exponential moments of any order, from which we deduce that Novikov criterion is satisfied. Hence, we can consider the new probability measure 
\begin{equation*} 
{\mathbb Q} := 
\exp \biggl( 
-
 \sum_{k \in {\mathbb N}} 
(1 \vee k)^\lambda \widehat{\mathcal V}^k(s, \text{\rm Leb}_{\mathbb S} \circ X_s^{-1} )
\ud  B_s^k 
- 
\frac12 
 \sum_{k \in {\mathbb N}} 
\Bigl\vert (1 \vee k)^\lambda \widehat{\mathcal V}^k(s, \text{\rm Leb}_{\mathbb S} \circ X_s^{-1} )
\bigr\vert^2 \ud s \biggr) \cdot {\mathbb P}. 
\end{equation*}
Under ${\mathbb Q}$, the law of $(X_s)_{0 \le s \le T}$ is the same as the law of $X^0$ under ${\mathbb P}$, see
\cite[Theorem 10.14]{dapratozabczyk}.
\vskip 4pt

\textit{Second Step.}
We now observe that, under ${\mathbb Q}$,  
\eqref{eq:representation} becomes
\begin{equation*} 
\begin{split}
&\partial_x g \bigl( X_T(x),{\rm Leb}_{\mathbb S} \circ X_T^{-1} \bigr)
+ \int_t^T
\partial_x f \bigl( X_s(x),{\rm Leb}_{\mathbb S} \circ X_s^{-1} \bigr)
\ud s
\\
&\hspace{15pt} = 
{\mathcal U}(t,x,\mu) 
-
\sum_{k \in {\mathbb N}} 
\int_t^T 
Z^k_s(x) 
 (1 \vee k)^\lambda \widehat{\mathcal V}^k(s, \text{\rm Leb}_{\mathbb S} \circ X_s^{-1} )
\ud s
+ \sum_{k \in {\mathbb N}} \int_t^T Z^k_s(x) 
\ud \widetilde 
B_s^k.
\end{split} 
\end{equation*} 
(The second term on the right-hand side is well-defined thanks to 
\eqref{eq:enegery:girsanov})
Here, we make use of Lemma
\ref{lem:2:5}. 
Taking 
inner product with a function $\theta \in L^2({\mathbb S})$ and then 
expectation 
on both sides (under ${\mathbb Q}$), we get 
\begin{equation*} 
\begin{split}
\bigl\langle
\theta, 
{\mathcal U}(t,\cdot,\mu)
\bigr\rangle_2
&= {\mathbb E}^{\mathbb Q} \biggl[ 
\int_{\mathbb S} 
\theta(x) 
\partial_x g \bigl( X_T(x),{\rm Leb}_{\mathbb S} \circ X_T^{-1} \bigr)
\ud x
\\
&\hspace{15pt} 
+ \int_t^T
\biggl( \int_{\mathbb S} 
\theta(x) 
\partial_x f 
 \bigl( X_s(x),{\rm Leb}_{\mathbb S} \circ X_s^{-1} \bigr)
\ud x \biggr) \ud s
%+
%\sum_{k \in {\mathbb N}} 
%Z^k_s(x) 
% (1 \vee k)^\lambda \widehat{V}^k(s, \text{\rm Leb}_{\mathbb S} \circ X_s^{-1} )
%\ud s
%\biggr]
%\\
%&={\mathbb E}^{\mathbb Q} \biggl[ \partial_x g \bigl( X_T(x),{\rm Leb}_{\mathbb S} \circ X_T^{-1} \bigr)
\\
&\hspace{15pt} +
\sum_{k \in {\mathbb N}} 
\int_t^T
\bigl\langle \theta,
\psi^k(s,\cdot, \text{\rm Leb}_{\mathbb S} \circ X_s^{-1} ) 
\bigr\rangle_2
 (1 \vee k)^\lambda \widehat{\mathcal V}^k(s, \text{\rm Leb}_{\mathbb S} \circ X_s^{-1} )
\ud s
\biggr].
\end{split}
\end{equation*} 
And then,
\begin{equation*} 
\begin{split}
\bigl\langle \theta, {\mathcal U}(t,\cdot,\mu)
\bigr\rangle_2
&= {\mathbb E} \biggl[ \int_{\mathbb S} 
\theta(x) \partial_x g \bigl( X_T^0(x),{\rm Leb}_{\mathbb S} \circ (X_T^0)^{-1} \bigr)
\ud x
\\
&\hspace{15pt} 
+ \int_t^T
\biggl( 
\int_{\mathbb S} 
\theta(x) 
\partial_x f 
 \bigl( X^0_s(x),{\rm Leb}_{\mathbb S} \circ (X^0_s)^{-1} \bigr)
\ud x \biggr) \ud s 
\\
&\hspace{15pt} +
\sum_{k \in {\mathbb N}} 
\int_t^T
\bigl\langle \theta, 
\psi^k(s,\cdot, \text{\rm Leb}_{\mathbb S} \circ (X^0)_s^{-1} ) 
\bigr\rangle_2
 (1 \vee k)^\lambda \widehat{\mathcal V}^k(s, \text{\rm Leb}_{\mathbb S} \circ (X^0)_s^{-1} )
\ud s
\biggr].
\end{split}
\end{equation*} 
We justify that the last term is well-defined. 
By Lemma 
\ref{lem:2:4}, we 
indeed have, for any $X \in U^2({\mathbb S})$,  
\begin{equation*} 
\begin{split} 
\sum_{k \in {\mathbb N}}
 (1 \vee k)^{2\lambda}
 \bigl\langle 
 \theta, 
\psi^k(s,\cdot,\text{\rm Leb}_{\mathbb S} \circ X^{-1} )
\bigr\rangle_2^2  
&\leq 
\sup_{t \leq s \leq T}
{\rm Lip} \biggl[ \int_{\mathbb S} \theta(x)  \, {\mathcal U}(s,x,\cdot)  \ud x \biggr]^2
\\
&\leq 
C 
\| \theta \|_2^2 ,  
\end{split} 
\end{equation*}
with the last line following from the fact that 
${\mathcal U}$ belongs to ${\mathcal C}$. 
%This implies 
%\begin{equation*} 
%\biggl\|
%\sum_{k \in {\mathbb N}}
% (1 \vee k)^\lambda
% \psi^k(s,\cdot,\text{\rm Leb}_{\mathbb S} \circ X^{-1} )
%\bigr\|_2 \leq 
%\frac{C}{(1 \vee k)}.
%\end{equation*} 
This allows us to 
let
\begin{equation*} 
 \Psi^\theta  \bigl(s,y,
{\rm Leb}_{\mathbb S} \circ 
X^{-1} 
 \bigr) 
 := \sum_{k \in {\mathbb N}} 
 (1 \vee k)^\lambda  
 \bigl\langle \theta, 
\psi^k(s,\cdot,\text{\rm Leb}_{\mathbb S} \circ X^{-1} )
\bigr\rangle_2
 e_k(y), \quad y \in {\mathbb S}. 
\end{equation*} 
As a function of $y$, it belongs to $L^2({\mathbb S})$. 
Moreover, 
\begin{equation*} 
\int_{\mathbb S}
\bigl\vert 
 \Psi^\theta\bigl(s,y,{\rm Leb}_{\mathbb S} \circ 
X^{-1} \bigr) 
\bigr\vert^2 \ud y 
 \leq 
\sup_{t \leq s \leq T}
{\rm Lip} \biggl[ \int_{\mathbb S} \theta(x) \,  {\mathcal U}(s,x,\cdot)  \ud x \biggr]^2.
\end{equation*} 
This completes the proof.
\end{proof} 

We now turn to the proof of the main statement. 
\begin{thm}
\label{main:thm}
Under the assumption of Theorem 
\ref{main statement}, there exists a unique field ${\mathcal U} \in {\mathcal C}$
solving the fixed point ${\mathcal U}=\Phi({\mathcal U})$.  
\end{thm}

\begin{proof}
The principle of the construction is well known, see \cite{Delarue2002209}. 
The point is to start from the terminal time $T$ and to construct ${\mathcal U}$ on the interval $[T-\delta,T]$, for some $\delta >0$, 
by means of 
Proposition 
\ref{prop:2:2}
and then to iterate. 

The challenge is to control the length 
$\delta$ along the iteration. 
As made clear in the statement of Proposition 
\ref{prop:2:2}, this is the same as controlling the 
Lipschitz constant 
of ${\mathcal U}$ (in the variable $\mu$, with the function ${\mathcal U}$ being seen as a mapping from $[0,T] \times U^2({\mathbb S})$ into $U^2({\mathbb S}) \subset L^2({\mathbb S})$) along the induction. 
Assuming that ${\mathcal U}$ solves the fixed point on some interval $[t,T]$, for $t \in (0,T)$, we then need to prove an \textit{a priori} 
bound for the Lipschitz constant of ${\mathcal U}$ in $\mu$ that is uniform with respect to the time argument in $[t,T]$.
For the sake of clarity, notice that the Lipschitz constant that needs to be controlled is 
\begin{equation}
\label{eq:Lip:U2:L2}
{\rm Lip}_{L^2({\mathbb S})} {\mathcal U}(s,\cdot,\cdot) 
:=
\sup_{\mu,\nu \in {\mathcal P}_2({\mathbb R}) : \mu \not = \nu}
{\mathbb W}_2(\mu,\nu)^{-1} \biggl( 
\int_{\mathbb S} 
\bigl\vert 
 {\mathcal U}(s,x,\mu)
-
 {\mathcal U}(s,x,\nu)
 \bigr\vert^2 
  \ud x 
\biggr)^{1/2},
\end{equation} 
which is indeed the Lipschitz constant of ${\mathcal U}(s,\cdot,\cdot)$, seen as a function from 
$U^2({\mathbb S})$ into $L^2({\mathbb S})$.

To proceed, we indeed apply Lemma
\ref{lem:phi} and we use 
the regularization
property proven in \cite{delarueHammersley2022rshe}
for the dynamics without drift $X^0$ (see the paragraph ``Notations'' in Introduction). 
Also, we make use of the notation introduced in the beginning of Subsection 
\ref{subse:3:3}. Whenever 
${\mathcal U}$ is a fixed point of the mapping $\Phi$ (with the latter being restricted to the time interval $[t,T]$ in an obvious manner), 
we then observe from Theorem 
\ref{thm:main:existence:uniqueness} 
that 
${\mathcal U}$ satisfies the main condition in the statement of Lemma 
\ref{lem:phi}. To do so, 
it suffices to combine the representation 
of ${\mathcal U}=\Phi({\mathcal U})$
with the first bound in 
Theorem 
\ref{thm:main:existence:uniqueness} (written with 
$(s,X)$ instead of $(0,X_0)$ as initial condition). We get, for $s \in [t,T]$ and $X \in U^2({\mathbb S})$, 
\begin{equation*} 
\begin{split}
\bigl\| \nabla_x {\mathcal U} \bigl(s,\cdot,{\rm Leb}_{\mathbb S} \circ X^{-1} 
\bigr)\bigr\|_2 
&\leq 
{\mathbb E} \biggl[ \| \nabla X_T
\|_2
+
\int_s^T 
\|\nabla X_r
\|_2 
\ud r
\biggr]
 \\
 &\leq C + \frac{C}{\sqrt{T-s}}  \| X \|_2 + \int_s^T \frac{C}{\sqrt{r-s}}   \| X \|_2 \ud r
 \leq C+ \frac{C}{\sqrt{T-s}}  \| X \|_2,
\end{split} 
\end{equation*} 
with $C$ depending on $T$ and ${\mathcal V}$. 
By the way, we also notice that we have a bound for $\| {\mathcal U} \|_\infty$ (in terms of known parameters), thanks
to 
\eqref{eq:eq:U}. 
%
%
%The first condition therein is just a consequence of the representation formula 
%\eqref{eq:eq:U}
%together with 
%\eqref{eq:main:existence:uniqueness}. 
%By the way, we notice that we have a bound for $\| U \|_\infty$ thanks
%to 
%\eqref{eq:eq:U}. 

Invoking 
Lemma \ref{lem:phi}
and 
%deduce that 
%\begin{equation*} 
%\begin{split}
%{\mathcal U}(t,x,\mu) &= 
%{\mathbb E} \biggl[ \partial_x g \Bigl( X_T^0(x), {\rm Leb}_{\mathbb S} \circ 
%(X_T^0)^{-1} \Bigr)
%\\
%&\hspace{15pt} -
%\int_t^T 
%\int_{\mathbb S} 
%{\mathcal U} \bigl( s, y, 
%{\rm Leb}_{\mathbb S} \circ 
%(X_s^0)^{-1}
%\bigr) 
%\Psi  \bigl(s,x,y,
%{\rm Leb}_{\mathbb S} \circ 
%(X_s^0)^{-1} 
% \bigr) \ud y \, \ud s 
% \biggr], 
% \end{split} 
%\end{equation*} 
%for any $(x,\mu) \in {\mathbb S} \times {\mathcal P}_2({\mathbb R})$. 
recalling that $({\mathscr P}^0_s)_{s \geq 0}$ denotes the semi-group generated by 
$X^0$, we obtain, for any function $\theta \in L^2({\mathbb S})$, 
\begin{equation}
\label{eq:equation:mild:X0} 
\begin{split}
\bigl\langle \theta, {\mathcal U}(t,\cdot,\mu) \bigr\rangle_2&= {\mathscr P}^0_{T-t} \Bigl(
  {\mathcal G}^\theta(\cdot)  
 \Bigr) 
+ 
\int_t^T 
{\mathscr P}^0_{s-t} 
\Bigl(
  {\mathcal F}^\theta(\cdot,\mu)  
 \Bigr) 
 \ud s
\\
&\hspace{15pt} +
\int_t^T 
{\mathscr P}^0_{s-t} 
\biggl( 
\int_{\mathbb S} 
{\mathcal U} \bigl( s, y, 
\cdot
\bigr) 
\Psi^\theta  (s,y,
\cdot)
 \ud y \biggr)  \ud s,  
 \end{split} 
\end{equation} 
with the notation
\begin{equation*} 
  {\mathcal F}^\theta(\mu)  
= \int_{\mathbb S} \theta(x) 
\partial_x f \bigl( F_{\mu}^{-1}(x), \mu \bigr) \ud x, 
\quad 
  {\mathcal G}^\theta(\mu)  
= \int_{\mathbb S} \theta(x) 
\partial_x g \bigl( F_{\mu}^{-1}(x), \mu \bigr) \ud x.
\end{equation*}

We then apply 
Theorem 5.11 in 
\cite{delarueHammersley2022rshe}, which says that there exists a constant $C>0$ such that, for any bounded measurable function 
$\Phi : {\mathcal P}_2({\mathbb R}) \rightarrow {\mathbb R}$ and any $t \in (0,T]$, 
\begin{equation}
\label{eq:regularisation} 
{\rm Lip} \bigl( {\mathscr P}_t^0 \Phi \bigr) \leq \frac{C}{t^{(1+\lambda)/2}} \| \Phi \|_\infty,
\end{equation}
where ${\rm Lip}$ is used to denote the Lipschitz constant of $\Phi$ (as usual ${\mathcal P}_2({\mathbb R})$ is equipped with ${\mathbb W}_2$). 
Then,
inserting 
\eqref{eq:regularisation} 
into
\eqref{eq:equation:mild:X0}, we obtain 
(for a constant $C$ only depending 
on known parameters) 
\begin{equation*} 
\begin{split} 
{\rm Lip}\Bigl[ 
\bigl\langle \theta, {\mathcal U}(t,\cdot,\cdot) \bigr\rangle_2
\Bigr] 
&\leq
C \| \theta \|_1
+
 \frac{C \| \theta \|_1}{(T-t)^{(1+\lambda)/2}}
\\
&\hspace{15pt} + C
 \int_t^T 
\frac{1}{(s-t)^{(1+\lambda)/2}} 
\sup_{\mu \in {\mathcal P}_2({\mathbb R})}
\biggl( 
\Bigl\vert 
\int_{\mathbb S} 
{\mathcal U} \bigl( s, y, 
\mu
\bigr) 
\Psi^\theta  (s,y,
\mu)
 \ud y 
 \Bigr\vert \biggr)  
\ud s. 
\end{split}
\end{equation*} 
Above, 
${\rm Lip} [ \langle \theta, {\mathcal U}(t,\cdot,\cdot) \rangle_2]$ denotes the Lipschitz constant of $\mu \mapsto \langle \theta, {\mathcal U}(t,\cdot,\mu) \rangle_2$, seen as a function from 
${\mathcal P}_2({\mathbb R})$ to ${\mathbb R}$. 
%This is different from the notation used in 
%Lemmas
%\ref{lem:2:4} 
%and \ref{lem:phi}. In particular, 
%one can check that
%\begin{equation}
%\label{eq:Lip:L2:R}
%{\rm Lip}_{L^2({\mathbb S})} 
%\ {\mathcal U}(t,\cdot,\cdot) 
%\leq \biggl( \int_{\mathbb S} 
%\Bigl[
%{\rm Lip} 
%\ {\mathcal U}(t,x,\cdot) 
%\Bigr]^2 \ud x \biggr)^{1/2}. 
%\end{equation}
%(By an obvious separability argument of the space 
%${\mathcal P}_2({\mathbb R})$, 
%the integrand in the right-hand side is well measurable in $x$.)
%Then,
Using Lemma 
\ref{lem:phi} to get a bound for
 $y \mapsto 
\int_{\mathbb S} 
\vert \Psi^\theta(s,y,\mu) \vert^2 \ud y $, we obtain 
(the value of the constant $C$ being allowed to vary from line to line) 
\begin{equation*} 
\begin{split}
&{\rm Lip}  \Bigl[ 
\bigl\langle \theta, {\mathcal U}(t,\cdot,\cdot) \bigr\rangle_2
\Bigr]
\leq
 \frac{C \| \theta \|_1}{(T-t)^{(1+\lambda)/2}}
+ C
 \int_t^T 
\frac1{(s-t)^{(1+\lambda)/2}} 
\sup_{s \leq r \leq T} 
{\rm Lip}
 \Bigl[ 
\bigl\langle \theta, {\mathcal U}(r,\cdot,\cdot) \bigr\rangle_2
\Bigr]
\ud s.
\end{split}
\end{equation*}
And then, for any 
$S$ fixed $(t,T)$, 
\begin{equation}
\label{eq:induction:delta} 
\begin{split}
{\rm Lip} \Bigl[ 
\bigl\langle \theta, {\mathcal U}(t,\cdot,\cdot) \bigr\rangle_2
\Bigr]
&\leq \frac{C \| \theta \|_1}{(T-t)^{(1+\lambda)/2}}
+ C (S-t)^{(1-\lambda)/2}
\sup_{t \leq s \leq S} 
{\rm Lip}
\Bigl[  \bigl\langle \theta, {\mathcal U}(s,\cdot,\cdot) \bigr\rangle_2 \Bigr]
\\
&\hspace{15pt} + C \sup_{S \leq s \leq T} 
{\rm Lip}
\Bigl[  \bigl\langle \theta, {\mathcal U}(s,\cdot,\cdot) \bigr\rangle_2 \Bigr].
\end{split}
\end{equation}
Importantly, $C$ is independent of $t$ and $\theta$. 
%Hence, \eqref{eq:Lip:L2:R}
%yields
%\begin{equation} 
%\label{eq:induction:delta}
%\begin{split}
%{\rm Lip}_{L^2({\mathbb S})}   \ {\mathcal U}(t,\cdot,\cdot) 
%&\leq \frac{c}{(T-t)^{(1+\lambda)/2}}
%+ C (S-t)^{(1-\lambda)/2}
%\sup_{t \leq s \leq S} 
%\Bigl[
%{\rm Lip}_{L^2({\mathbb S})} 
%\  {\mathcal U}(s,\cdot,\cdot) \Bigr]
%\\
%&\hspace{15pt} + C \sup_{S \leq s \leq T} 
%\Bigl[
%{\rm Lip}_{L^2({\mathbb S})} 
%\  {\mathcal U}(s,\cdot,\cdot) \Bigr].
%\end{split}
%\end{equation}

Next, we
recall 
\eqref{eq:Lip:U2:L2}
and we notice that 
\begin{equation}
\label{eq:induction:delta:22} 
{\rm Lip} \Bigl[ 
\bigl\langle \theta, {\mathcal U}(s,\cdot,\cdot) \bigr\rangle_2
\Bigr]
\leq \| \theta \|_2
{\rm Lip}_{L^2({\mathbb S})} {\mathcal U}(s,\cdot,\cdot). 
\end{equation} 
We then choose $\delta$ such that $C \delta^{(1-\lambda)/2}  \leq 1/2$. 
Assume also that we have a bound $C_\ell$, $\ell \geq 1$, for 
$\sup_{T-\ell \delta \leq s \leq T} 
[{\rm Lip}_{L^2({\mathbb S})}  {\mathcal U}(s,\cdot,\cdot) ]$. 
By Lemma 
\ref{lem:U} and by 
\eqref{eq:induction:delta:22}, we already have a bound for $C_1$ that only depends on the 
properties of $f$ and $g$, provided 
$\delta$ is small enough. 
If $\ell \delta <T-t$, 
then, for any $s \in [T- (\ell+1)\delta,T - \ell \delta]$, $s\geq t$, 
\eqref{eq:induction:delta}
yields
\begin{equation*} 
%{\rm Lip}_{L^2({\mathbb S})}   {\mathcal U}(s,\cdot,\cdot)
\sup_{\| \theta \|_2 \leq 1} 
{\rm Lip} \Bigl[ 
\bigl\langle \theta, {\mathcal U}(s,\cdot,\cdot) \bigr\rangle_2
\Bigr]
  \leq \frac{C}{\delta^{(1+\lambda)/2}}
+ \frac12 
\sup_{s \leq r \leq T-\ell \delta} 
\sup_{\| \theta \|_2 \leq 1} 
{\rm Lip} \Bigl[ 
\bigl\langle \theta, {\mathcal U}(r,\cdot,\cdot) \bigr\rangle_2 \Bigr] + C C_\ell.
\end{equation*}
Taking the supremum over $s \in [t \vee (T-(\ell+1)\delta),T-\ell \delta]$, 
 we get a  bound $C_{\ell+1}$ for
$$\sup_{t \vee (T-(\ell+1) \delta) \leq s \leq T} 
\sup_{\| \theta \|_2 \leq 1} 
{\rm Lip} \Bigl[ 
\bigl\langle \theta, {\mathcal U}(s,\cdot,\cdot) \bigr\rangle_2
\Bigr].$$
Now, it is standard to prove that 
$$\sup_{\| \theta \|_2 \leq 1} 
{\rm Lip} \Bigl[ 
\bigl\langle \theta, {\mathcal U}(s,\cdot,\cdot) \bigr\rangle_2
\Bigr]
=
{\rm Lip}_{L^2({\mathbb S})} {\mathcal U}(s,\cdot,\cdot).$$ 
 In this way, we get an \textit{a priori} bound for the Lipschitz constant 
 of $ {\mathcal U}$, seen as a function from 
$U^2({\mathbb S})$ into $L^2({\mathbb S})$, along its iterative backward construction. 
\end{proof}

\subsection{Existence and uniqueness of a mean field equilibrium (first part of Theorem \ref{main statement})}

Since we already have 
proven
Theorem \ref{main:thm}, it remains to establish 
Proposition 
\ref{prop:fixed:point} in order to prove the existence and uniqueness part in 
Theorem 
\ref{main statement}.

\begin{proof}[Proof of Proposition 
\ref{prop:fixed:point}]

The fact that a fixed point to $\Phi$ induces a mean field equilibrium
is a consequence of Lemma 
\ref{lem:martingale} and then Proposition 
\ref{prop:2:7:new}.  

We then need to prove that uniqueness of 
a fixed point to $\Phi$ implies uniqueness of the mean field equilibrium. 
In addition
to a fixed point ${\mathcal U}$, we consider another equilibrium ${\mathcal U}'$ (as in Definition 
\ref{def:mathcalC}). 
We call $(X_t')_{0 \le t \le T}$ 
the solution 
to
\eqref{eq:V} driven by ${\mathcal U}'$. 
By Proposition 
\ref{prop:2:7:new}, we know that, 
for almost every $t \in [0,T]$,
${\mathbb P}$-almost surely, 
for all $x \in {\mathbb S}$,
\begin{equation}
\label{eq:pontryagin:2}
\begin{split}
{\mathcal U}'\bigl(t,x,{\rm Leb}_{\mathbb S} \circ (X_t')^{-1} 
\bigr) &= {\mathbb E} \Bigl[ \partial_x g \Bigl( X_T'(x), {\rm Leb}_{\mathbb S} \circ (X_T')^{-1} \Bigr) 
\\
&\hspace{15pt} 
+ \int_t^T
\partial_x f \bigl( X_s'(x),{\rm Leb}_{\mathbb S} \circ (X_s')^{-1} \bigr)
\ud s
\, \vert \, {\mathcal F}_t \Bigr]. 
\end{split}
\end{equation} 
Considering as before 
a regular conditional probability $({\mathbb P}_{t,\omega})_{\omega \in \Omega}$ of ${\mathbb P}$ given the $\sigma$-field generated by 
$\sigma(W_r ; r \leq t, x \in {\mathbb S})$, we have that, $\omega$ ${\mathbb P}$-almost surely, under ${\mathbb P}_{t,\omega}$, $(X_r')_{s \leq r \leq T}$ can be regarded as the unique solution of \eqref{eq:V} driven by ${\mathcal U}'$ and initialized from 
$(t,X_t'(\omega))$. By expressing the above conditional expectation through the regular conditional probability, 
we obtain:
\begin{equation*} 
\begin{split}
{\mathcal U}'\bigl(t,x,{\rm Leb}_{\mathbb S} \circ (X_t'(\omega))^{-1} 
\bigr) &= 
\int_{\Omega} 
\partial_x g \bigl( 
X_T'(\omega')(x),
{\rm Leb}_{\mathbb S} \circ (X_T'(\omega'))^{-1} 
\bigr)
\ud {{\mathbb P}_{t,\omega}}(\omega')
\\
&\hspace{15pt} + 
\int_t^T \int_{\Omega} 
\partial_x f \bigl( 
X_s'(\omega')(x),
{\rm Leb}_{\mathbb S} \circ (X_s'(\omega'))^{-1} 
\bigr)
\ud {{\mathbb P}_{t,\omega}}(\omega') \, \ud s.
\end{split}
\end{equation*}
Under the same probability measure, we can write
\begin{equation*} 
\begin{split} 
{\mathcal U}\bigl(t,x,{\rm Leb}_{\mathbb S} \circ (X_t'(\omega))^{-1} 
\bigr) &= 
\int_{\Omega} 
\partial_x g \bigl( 
X_T(\omega')(x),
{\rm Leb}_{\mathbb S} \circ (X_T(\omega'))^{-1} 
\bigr)
d{{\mathbb P}_{t,\omega}}(\omega')
\\
&\hspace{15pt} + 
\int_t^T \int_{\Omega} 
\partial_x f \bigl( 
X_s(\omega')(x),
{\rm Leb}_{\mathbb S} \circ (X_s(\omega'))^{-1} 
\bigr)
\ud {{\mathbb P}_{t,\omega}}(\omega') \, \ud s,
\end{split} 
\end{equation*}where $(X_s)_{t \leq s \leq T}$ is the solution to \eqref{eq:V} driven by ${\mathcal U}$ and initialized from 
$(t,X_t'(\omega))$. 
Repeating the proof of Lemma 
\ref{lem:U}, we obtain that, 
for $T-t \leq C_{\mathcal U}/2$, 
${\mathcal U}(t,\cdot,\cdot)$ and ${\mathcal U}'(t,\cdot,\cdot)$ should coincide. 
Restarting the analysis from time $T-C_{\mathcal U}/2$, with ${\mathcal U}(T-C_{\mathcal U}/2,x,\mu)$ playing the role of 
$\partial _xg(X(x),\mu)$ for $X \sim \mu$, we obtain 
the identification of ${\mathcal U}$ and ${\mathcal U}'$ on 
$[T-2\times C_{\mathcal U}/2,T]$. Arguing inductively, we can identify ${\mathcal U}$ and ${\mathcal U}'$ on the interval $[0,T]$, which 
proves that there is a unique mean field equilibrium, given by ${\mathcal U}$. 
\end{proof} 

\section{Representation of ${\mathcal U}$}
\label{sec:intrinsic} 

Let ${\mathcal U}$ be given by Theorem 
\ref{main:thm} (equivalently, by the first part of Theorem \ref{main statement}). Solve 
\eqref{eq:V}
with $0$ as initial time (and with some initial state $\mu \in {\mathcal P}_2({\mathbb R})$). 

The purpose of this section is to show 
the second 
part in 
Theorem \ref{main statement}: along the 
path $(X_t(\cdot))_{0 \le t \le T}$ (which takes values in $U^2({\mathbb S})$), 
 the field 
${\mathcal U}$ can be represented in a mean-field form, namely, for 
almost every $\omega \in \Omega$ and 
almost every $t \in [0,T]$, 
 ${\mathcal U}(t,x,{\rm Leb}_{\mathbb S} \circ X_t^{-1})$ can be 
 written in the form 
\begin{equation}
\label{eq:tildeU} 
{\mathcal U}\bigl(t,x,{\rm Leb}_{\mathbb S} \circ X_t^{-1}\bigr) = \tilde {\mathcal U}(t,X_t(x),{\rm Leb}_{\mathbb S} \circ X_t^{-1}),
\end{equation} 
for a measurable mapping $\tilde {\mathcal U} : [0,T] \times {\mathbb R} \times {\mathcal P}_2({\mathbb R}) \rightarrow {\mathbb R}$. 

\subsection{Formula when $X$ is invertible} 
The proof of 
\eqref{eq:tildeU} relies on the following generic principle. 
\begin{lem}
\label{lem:X:invertible}
Let $X \in U^2({\mathbb S})$ be a continuous non-decreasing mapping such that 
\begin{equation}
\label{eq:lem:4:1} 
\int_{\mathbb S} \int_{\mathbb S} {\mathbf 1}_{\{ X(x) = X(y)\}} \ud x \, \ud y =0. 
\end{equation} 
Then, the function $X : [0,1/2] \rightarrow {\mathbb R}$ is one-to-one from $[0,1/2]$ on its image $X({\mathbb S})$. Denoting 
by 
$X^{-1}$ the inverse (which is measurable), it holds, for any $t \in [0,T]$, 
\begin{equation*} 
{\mathcal U}(t,x,{\rm Leb}_{\mathbb S} \circ X^{-1}) 
= \tilde {\mathcal U} \bigl(t,X(x),{\rm Leb}_{\mathbb S} \circ X^{-1}\bigr), 
\end{equation*} 
with 
\begin{equation*} 
\tilde {\mathcal U}(t,y,{\rm Leb}_{\mathbb S} \circ X^{-1})
= 
\left\{ 
\begin{array}{ll}
{\mathcal U} \bigl(t,X^{-1}(y),{\rm Leb}_{\mathbb S} \circ X^{-1}\bigr) &\quad y \in X({\mathbb R})
\\
{\mathcal U} \bigl(t,0,{\rm Leb}_{\mathbb S} \circ X^{-1}\bigr) 
 &\quad y < X(0)
\\
{\mathcal U} \bigl(t,1/2,{\rm Leb}_{\mathbb S} \circ X^{-1}\bigr) 
 &\quad y > X(1/2)
\end{array}
\right..
\end{equation*} 
\end{lem}
\begin{proof}
We know that ${\mathcal U}$ is a non-decreasing continuous function on $[0,1/2]$. Assume that it is not strictly decreasing, meaning that there exist 
$x_1$ and $x_2$ in $[0,1/2]$, with $x_1<x_2$, such that 
${\mathcal U}(x_1)={\mathcal U}(x_2)$. Then, for all $z \in [x_1,x_2]$, ${\mathcal U}(z)={\mathcal U}(x_1)={\mathcal U}(x_2)$. 
Letting $I=[x_1,x_2]$, we have
\begin{equation*} 
\int_{\mathbb S} \int_{\mathbb S} {\mathbf 1}_{\{ X(x) = X(y)\}} \ud x \, \ud y 
\geq 
\int_{I} \int_{I} {\mathbf 1}_{\{ X(x) = X(y)\}} \ud x \, \ud y = \vert I \vert^2 >0, 
\end{equation*}  
which is a contradiction. Therefore, $X$ is continuous one-to-one from $[0,1/2]$ onto its image. 
This makes it possible to define $\tilde {\mathcal U}$ as in the statement. 

Now, for $x \in [0,1/2]$,
\begin{equation*}
\tilde {\mathcal U}\bigl(t,X(x),{\rm Leb}_{\mathbb S} \circ X^{-1}\bigr)= 
{\mathcal U} \bigl(t,X^{-1}(X(x)),{\rm Leb}_{\mathbb S} \circ X^{-1}\bigr)
= {\mathcal U} (t,x,{\rm Leb}_{\mathbb S} \circ X^{-1}). 
\end{equation*} 
When $x \in [-1/2,0]$, we have
\begin{equation*} 
\begin{split}
{\mathcal U} (t,x,{\rm Leb}_{\mathbb S} \circ X^{-1})
=
{\mathcal U} (t,-x,{\rm Leb}_{\mathbb S} \circ X^{-1})
&=
\tilde {\mathcal U}\bigl(t,X(-x),{\rm Leb}_{\mathbb S} \circ X^{-1}\bigr)
\\
&=
\tilde {\mathcal U}\bigl(t,X(x),{\rm Leb}_{\mathbb S} \circ X^{-1}\bigr).
\end{split}
\end{equation*} 
This completes the proof. 
\end{proof}

We recall from 
\eqref{eq:inverse:generalise} that the function $x \in [0,1] \mapsto X(x/2)$ is the quantile function of the 
probability 
measure $\mu := {\rm Leb}_{\mathbb S} \circ X^{-1}$. With $X^{-1}$ defined as in the previous lemma, 
we get 
\begin{equation*} 
\bigl( 2 X^{-1}\bigr) \circ X \bigl( \frac{x}2 \bigr) =  x, \quad x \in [0,1], 
\end{equation*} 
which shows that the function $2X^{-1}$ is the cumulative distribution function of $\mu$, which we denote by 
$F_\mu$. 
Therefore, 
$\tilde {\mathcal U}$ can be rewritten as
\begin{equation}
\label{eq:tildeU:2} 
\tilde  {\mathcal U}(t,y,\mu)
= 
 {\mathcal U} \Bigl(t,\frac{F_\mu(y)}2,\mu \Bigr) \quad y \in  {\mathbb R}.
\end{equation} 
We can prove that the function 
\begin{equation*} 
(t,y,\mu) \mapsto \Bigl(t,\frac{F_\mu(y)}2,\mu \Bigr)
\end{equation*}
is measurable. This shows that $\tilde {\mathcal U}$ is measurable. Indeed, we notice that, for any $a\geq 0$, 
\begin{equation*} 
\begin{split}
&\bigl\{ (\mu,y) \in {\mathcal P}_2({\mathbb R}) \times {\mathbb R} : F_\mu(y) \geq a \bigr\}
\\
&= \bigl\{ (\mu,y) \in {\mathcal P}_2({\mathbb S}) \times {\mathbb R} : \mu \bigl((-\infty,y]\bigr) \geq a \bigr\}
\\
&= 
\bigcap_{n \in {\mathbb N} \setminus \{0\}} 
\bigcup_{z \in {\mathbb Q}} 
 \bigl\{ 
\mu \in {\mathcal P}_2({\mathbb R}) : \mu \bigl( (-\infty,z] \bigr) \geq a \bigr\} 
\times
\big\{ y \in {\mathbb R} : z- \tfrac1{n} \leq y \leq z \bigr\}.
\end{split}
\end{equation*} 
It is easy to see that the set on the last line above belongs to the product of the two $\sigma$-fields on 
${\mathcal P}_2({\mathbb R})$ and ${\mathbb R}$ respectively. 
Importantly, we notice that the definition of $\tilde  {\mathcal U}$
in \eqref{eq:tildeU:2} always makes sense, even if the 
assumption of Lemma 
\ref{lem:X:invertible} is not satisfied. This plays a key role in the sequel. 

Notice also that, 
whenever 
$\mu$ has no atoms, $F_\mu$ is continuous so that 
$\tilde{\mathcal U}$ is continuous in $y$. 
Observe that
so is the case for $\mu = \textrm{\rm Leb}_{\mathbb S} \circ X^{-1}$ when 
$X$
satisfies
\eqref{eq:lem:4:1}.

\subsection{Almost sure invertibility of the solution}

Here is now the second step in the construction of 
a mean field representative of the process $({\mathcal U}(s,x,{\rm Leb}_{\mathbb S} \circ X_s^{-1}))_{0 \leq s \leq T}$. 

\begin{lem}
Let ${\mathcal U}$ be given by Theorem 
\ref{main:thm}. Solve 
\eqref{eq:V}
with $0$ as initial time (and with some initial state $\mu \in {\mathcal P}_2({\mathbb R})$). Then, for almost every 
$\omega \in \Omega$ and for almost every $t \in [0,T]$, 
\begin{equation*} 
\int_{\mathbb S} \int_{\mathbb S} {\mathbf 1}_{\{ X_t(x) = X_t(y)\}} \ud x \,  \ud y =0. 
\end{equation*} 
\end{lem} 

\begin{proof} 
By the same Girsanov transformation as in Lemma 
\ref{lem:phi}, we can in fact prove the result for the driftless equation 
\eqref{eq:V}.
Then, 
for a given $\varepsilon >0$, consider a smooth positive function $\psi : {\mathbb R} \rightarrow {\mathbb R}$ such that 
\begin{equation*} 
 {\mathbf 1}_{\{\vert y \vert \leq \varepsilon\}}
 \leq \psi(y) 
 \leq  {\mathbf 1}_{\{\vert y \vert \leq 2\varepsilon\}}, \quad y \in {\mathbb R}. 
\end{equation*} 
Obviously, we can choose $\psi$ to be even. Also, we can find a function $\Psi$ such that $\Psi''=\psi$. Choosing 
$\Psi(0)=\Psi'(0)=0$, we have 
$\vert \Psi'(x) \vert \leq 2\varepsilon$ for $x \in {\mathbb R}$ and then 
$\vert \Psi(x) \vert \leq 2 \varepsilon \vert x \vert$. Clearly, $\Psi$ is even. 

Next, we consider the following functional
on ${\mathcal P}_2({\mathbb R})$:  
\begin{equation*} 
\Xi(\mu) := \frac12 \int_{{\mathbb R}} \int_{{\mathbb R}} \Psi(x-y) \ud \mu(x) \ud \mu(y), \quad \mu \in {\mathcal P}_2({\mathbb R}). 
\end{equation*} 
Since $\Psi$ is smooth,  $\Xi$ has continuous first and second-order derivatives (with respect to $\mu$). They are given by 
\begin{equation*}
\begin{split} 
&\partial_\mu \Xi(\mu)(x) =  \int_{\mathbb R} 
\Psi'(x-y) \ud \mu(y), 
\\
&\partial_x \partial_\mu \Xi(\mu)(x) =  \int_{\mathbb R} 
\psi(x-y) \ud \mu(y),
\\
&\partial_\mu^2 \Xi(\mu)(x,y) = -
\psi(x-y). 
\end{split} 
\end{equation*} 
We believe that those notations are by now standard. We refer to Chapter 5 in the book 
\cite{CarmonaDelarueI}. 

By It\^o's formula proven in \cite{delarueHammersley2022ergodicrshe}, we obtain 
\begin{equation*}
\begin{split}
{\mathbb E} 
&\Bigl[ \Xi\bigl( {\rm Leb}_{\mathbb S} \circ X_t^{-1}  \bigr) 
-
\Xi\bigl(  {\rm Leb}_{\mathbb S} \circ X_0^{-1}  \bigr) \Bigr]
\\
&= 
  - {\mathbb E} \int_0^t 
\int_{\mathbb S} \int_{\mathbb S} 
\psi \bigl(X_s(x) - X_s(y) \bigr) \bigl[ \nabla X_s(x) \bigr]^2 dx \, \ud y \, ds
\\
&\hspace{15pt} +
\frac12
\sum_{k \geq 1} k^{-2 \lambda} 
{\mathbb E}  \int_0^t  \int_{\mathbb S} 
\int_{\mathbb S} 
\psi \bigl(X_s(x) - X_s(y) \bigr)
e_k^2(x) 
\ud x \, \ud y \, \ud s
    \\
&\hspace{15pt} 
- \frac12 \sum_{k \geq 1} k^{-2\lambda}  {\mathbb E}   \int_0^t \int_{\mathbb S}\int_{\mathbb S}% \partial^2_{\mu}
\psi \bigl(X_s(x) - X_s(y) \bigr)
e_k(x) e_k(y) 
\ud x \, \ud y \, \ud s,
\end{split}
\end{equation*}
which yields 
%can be rewritten
\begin{equation*}
\begin{split}
&\frac14
\sum_{k \geq 1} k^{-2\lambda} 
{\mathbb E}  \int_0^t  \int_{\mathbb S} 
\int_{\mathbb S} 
{\mathbf 1}_{\{ X_s(x) = X_s(y) \}} 
\bigl\vert
e_k(x) 
-
e_k(y) 
\bigr\vert^2  \ud x \, \ud y \, \ud s
\\
&\leq C \varepsilon +  {\mathbb E} \int_0^t 
\int_{\mathbb S} \int_{\mathbb S} 
{\mathbf 1}_{\{ \vert X_s(x) - X_s(y) \vert \leq 2 \varepsilon \}} 
 \bigl[ \nabla X_s(x) \bigr]^2 \ud x \, \ud y \, \ud s.
\end{split}
\end{equation*}
Letting $\varepsilon$ tend to $0$, we get 
\begin{equation}
\label{eq:epsilon:tends:to:0}
\begin{split}
&\frac14
\sum_{k \geq 1} k^{-2 \lambda} 
{\mathbb E}  \int_0^t  \int_{\mathbb S} 
\int_{\mathbb S} 
{\mathbf 1}_{\{ X_s(x) = X_s(y) \}} 
\bigl\vert
e_k(x) 
-
e_k(y) 
\bigr\vert^2 \ud x \, \ud y \, \ud s
\\
&\leq
{\mathbb E} \int_0^t 
\int_{\mathbb S} \int_{\mathbb S} 
{\mathbf 1}_{\{ X_s(x) = X_s(y) \}} 
 \bigl[ \nabla X_s(x) \bigr]^2 \ud x \, \ud y \, \ud s.
\end{split}
\end{equation}
We claim that the last line on the right-hand side is equal to $0$. Take indeed $y \in {\mathbb S}$ and consider, for an element $h \in H^1({\mathbb S})
 \cap U^2({\mathbb S})$, the integral 
 \begin{equation*} 
 \int_{\mathbb S} 
{\mathbf 1}_{\{  h(x) = h(y) \}} 
 \bigl[ \nabla h(x) \bigr]^2 \ud x.
 \end{equation*} 
 We claim that it is zero.
 Without any loss of generality, we can assume (by symmetry and periodicity) that 
 $y \in [0,1/2]$. Then, 
 the set $\{ x \in {\mathbb S} : h(x)=h(y)\}$ must be made of one or two intervals since
 $h$ is non-decreasing on $[0,1/2]$
 and non-increasing on $[-1/2,0]$.
When there are two intervals, they are obtained one from each other by means of the mapping 
$x \mapsto -x$.  
  We then call 
$I$ the interval containing at least one non-negative real. 
If the length of $I$ is zero, then the proof of the claim is done, i.e. 
the above integral is indeed zero. If not, 
we write
 \begin{equation*} 
 \int_{\mathbb S} 
{\mathbf 1}_{\{  h(x) = h(y) \}} 
 \bigl[ \nabla h(x) \bigr]^2 \ud x
 \leq 
 \int_I
{\mathbf 1}_{\{  h(x) = h(y) \}} 
 \bigl[ \nabla h(x) \bigr]^2 \ud x +  \int_{-I}
{\mathbf 1}_{\{  h(x) = h(y) \}} 
 \bigl[ \nabla h(x) \bigr]^2 \ud x,
 \end{equation*} 
 but we observe that, on the interior of $I$, $\nabla h$ must be zero (and similarly on $-I$). This completes the proof the claim. Back to 
 \eqref{eq:epsilon:tends:to:0}, we obtain
 \begin{equation*}
\begin{split}
&\frac14
\sum_{k \geq 1} k^{-2\lambda} 
{\mathbb E}  \int_0^t  \int_{\mathbb S} 
\int_{\mathbb S} 
{\mathbf 1}_{\{ X_s(x) = X_s(y) \}} 
\bigl\vert
e_k(x) 
-
e_k(y) 
\bigr\vert^2 \ud x \, \ud y \, \ud s
=0.
\end{split}
\end{equation*}
Recalling that $e_k(x) = \sqrt{2} \cos(2 \pi k x)$,  
this may be rewritten as
\begin{equation*}
\begin{split}
\sum_{k \geq 1}  k^{-2\lambda} 
{\mathbb E} \int_0^t \int_{\mathbb S} \int_{\mathbb S}  {\mathbf 1}_{\{ \vert X_s(x) - X_s(y) \vert =0 \}}
\sin^2\bigl( 2 \pi k (x-y) \bigr)
\sin^2 \bigl( 2 \pi k (x+y) \bigr) \ud x \, \ud y \, \ud s = 0. 
\end{split} 
\end{equation*} 
So, ${\mathbb P}$ almost surely, for almost every $t \in [0,T]$ and almost every $x \in {\mathbb S}$, 
\begin{equation*}
\begin{split}
\sum_{k \geq 1}  k^{-2\lambda} 
  \int_{\mathbb S}  {\mathbf 1}_{\{ \vert X_s(x) - X_s(y) \vert =0 \}}
\sin^2\bigl( 2 \pi k (x-y) \bigr)
\sin^2 \bigl( 2 \pi k (x+y) \bigr) \ud y=0.  
\end{split} 
\end{equation*} 
For a given $k \in {\mathbb N} \setminus \{0\}$, the function 
$y \in {\mathbb S} \mapsto 
\sin^2( 2 \pi k (x-y) )
\sin^2( 2 \pi k (x+y) )$ can only have $x$ as a zero on a small interval around $x$ (since 
the function is a trigonometric polynomial and has a finite number of zeros on the circle). 
This means that we can find $\delta >0$ (random, depending on $t$ and $x$) such that 
 $$\int_{\mathbb S}  {\mathbf 1}_{\{ \vert X_s(x) - X_s(y) \vert =0 \}}
 {\mathbf 1}_{\{ d_{\mathbb S}(x,y) < \delta \}}
 \ud y =0,$$
where $ d_{\mathbb S}$ is the distance on ${\mathbb S}$. 

 Using again the fact that the collection of points $y$ such that $X_s(y)=X_s(x)$ is an interval containing 
 $x$, we deduce that 
 $$\int_{\mathbb S}  {\mathbf 1}_{\{ \vert X_s(x) - X_s(y) \vert =0 \}}
 \ud y =0.$$
 The above holds true ${\mathbb P}$ almost-surely, for almost every $t$ and for almost every $x$. 
 The result easily follows. 
\end{proof}

\subsection{Main statement}

By combining the last two subsections, we get
\begin{thm}
\label{eq:main:statement:tildeU}
Let ${\mathcal U}$ be given by Theorem 
\ref{main:thm}
and call 
\begin{equation*}
\tilde {\mathcal U}(t,y,\mu)
= 
{\mathcal U} \Bigl(t,\frac{F_\mu(y)}2,\mu \Bigr) \quad y \in  {\mathbb R}.
\end{equation*} 
 Also, 
solve 
\eqref{eq:V}
with $0$ as initial time (and with some initial state $\mu \in {\mathcal P}_2({\mathbb R})$)
and call the solution $(X_t)_{0 \leq t \leq T}$. 
Then, $(X_t)_{0 \leq t \leq T}$ solves 
\begin{equation}
\label{eq:V:2}
\ud X_t(x) = - \tilde{\mathcal U} \bigl(t,X_t(x),\text{\rm Leb}_{\mathbb S} \circ X_t^{-1} \bigr) \ud t + \Delta X_t(x) \ud t + \ud W_t(x) + \ud \eta_t(x), \quad x \in {\mathbb S}, 
\end{equation} 
where, for almost every $t \in [0,T]$, 
${\mathbb P}$-almost surely,
\begin{equation*} 
\begin{split}
\tilde{\mathcal U} \bigl(t,X_t(x),\text{\rm Leb}_{\mathbb S} \circ X_t^{-1} \bigr)
&= {\mathbb E} \biggl[ \partial_x g\bigl(X_T(x),\text{\rm Leb}_{\mathbb S} \circ X_T^{-1} \bigr)
\\
&\hspace{15pt} + \int_t^T \partial_x f\bigl(X_s(x),\text{\rm Leb}_{\mathbb S} \circ X_s^{-1} \bigr) 
\ud s
\, \vert \, {\mathcal F}_t
\biggr], \quad x \in {\mathbb S}.
\end{split}
\end{equation*}
\end{thm} 

\subsection{Connection with standard mean field games}
The reader may wonder about the connection between our construction and the one that is usually followed in
standard mean field games without common noise. Intuitively, the field $\widetilde {\mathcal U}$ identified in the statement of 
Theorem \ref{eq:main:statement:tildeU} should be regarded as 
$(t,x,\mu) \mapsto \partial_x {\mathcal V}(t,x,\mu)$, where ${\mathcal V}$ is the solution of the so-called 
master equation in mean field games, see for instance  \cite{cardaliaguet2019master,CarmonaDelarueI,CarmonaDelarueII}. 
In the absence of common noise, the construction of the field ${\mathcal V}$ is just restricted to cases satisfying certain monotonicity properties (like Lasry-Lions monotonicity either or displacement monotonicity). Because the displacement monotonicity condition is more in line 
with the Pontryagin principle, which we invoked in 
the proof of Proposition 
\ref{prop:2:7:new}, we here assume (in this paragraph only and in order to fix the setting) 
\begin{equation}
\label{eq:monotonicity:assumption} 
{\mathbb E}
\Bigl[ \Bigl( X - X' \Bigr) 
\Bigl( 
\partial_x f\bigl(X,{\mathcal L}(X)\bigr) 
- 
\partial_x f \bigl( X',{\mathcal L}(X') \bigr) 
\Bigr) \Bigr] \geq 0,
\end{equation} 
and similarly for $g$, for any two random variables $X$ and $X'$
that are constructed on an arbitrary probability space, with respective laws ${\mathcal L}(X)$ and 
${\mathcal L}(X')$.

We then consider the same mean field game as the one described in 
Subsection \ref{subse:def:mean field game}, but without 
the additional forcing $\Delta X_t \ud t + \ud W_t$ in the 
equation \eqref{eq:V}. Instead, we just postulate that 
\begin{equation}
\label{eq:V:bis}
\ud X_t(x) = - {\mathcal V}\bigl(t,x,\text{\rm Leb}_{\mathbb S} \circ X_t^{-1} \bigr) \ud t + \ud \eta_t(x), \quad
  t \in [0,T], \quad x \in {\mathbb S}, 
\end{equation} 
where, as before, $(\eta_t)_{0 \leq t \leq T}$ is a reflection term with values in $L^2({\mathbb S})$ such that, for any
$u \in U^2({\mathbb S})$, the path $(\langle \eta_t,u \rangle)_{0 \leq t \le T}$ is non-decreasing. 
As far as ${\mathcal V}$ is concerned, we assume it to belong to the 
class ${\mathcal C}$ introduced in Definition 
\ref{def:mathcalC}. Following Brenier's original work \cite{Brenier1}, one can show that 
\eqref{eq:V:bis}
 has 
 a unique solution 
$(X_t)_{0 \leq t \leq T}$ with a continuous path from $[0,T]$ into $U^2({\mathbb S})$. It is obtained as the limit of 
the following discrete scheme
\begin{equation}
\label{eq:discrete:scheme:deterministic} 
\frac{\ud}{\ud t} X_t^{(h)}(x) =  - {\mathcal V}\bigl(t,x,{\rm Leb}_{\mathbb S} \circ (X_{t_n}^{(h)})^{-1} \bigr), \quad
  t \in [t_n,t_{n+1}), \quad x \in {\mathbb S}, 
\end{equation}  
for a
time mesh $(t_n)_{n=0,\cdots,N}$  of the interval $[0,T]$ with uniform step size $h$. At time $t_{n+1}$, $X_{t_{n+1}}^{(h)}$ is obtained by
rearrangement: 
\begin{equation}
\label{eq:rearrange}
X_{t_{n+1}}^{(h)} = \Bigl( X_{t_{n+1}-}^{(h)} \Bigr)^\star,
\end{equation}
with the symbol $-$ in the above time index denoting the left-limit at point $t_{n+1}$ (which obviously exists here) 
and the symbol $\star$ is the rearrangement operation. Intuitively, we rearrange any function on ${\mathbb S}$ in the form of a function, which has the same distribution under the Lebesgue measure, but which belongs to $U^2({\mathbb S})$. Details may be found in  \cite{Brenier1,delarueHammersley2022rshe}. 

Here is the main result of this subsection. It  shows that, in absence of common noise, our notion of mean field game coincides with the usual one: 
\begin{prop}
\label{prop:MFG}
On the top of the assumptions used in 
Theorem 
\ref{main statement}, assume that 
\eqref{eq:monotonicity:assumption} is in force. Then, the mean field game, as defined in 
Definition 
\ref{subse:def:mean field game} but with $(X_t)_{0 \leq t \leq T}$ satisfying  
equation \eqref{eq:discrete:scheme:deterministic}, has a unique solution ${\mathcal U} \in {\mathcal C}$. 
This solution coincides with the unique mean field game equilibrium that is understood in the usual sense. 

In clear, 
 the flow of distributions 
 $(\mu_t)_{0 \le t \le T}$ defined by
\begin{equation*} 
\mu_t := \textrm{\rm Leb}_{\mathbb S} \circ X_t^{-1},
\end{equation*}  
when $(X_t)_{0 \leq t \leq T}$ is driven by ${\mathcal U}$, is the unique 
continuous flow $(\nu_t)_{0 \leq t \leq T}$ from $[0,T]$ to ${\mathcal P}_2({\mathbb R})$, with $\nu_0=\mu_0$, such that 
the optimal trajectory of the minimization problem 
\begin{equation*}
\inf 
\widetilde J\Bigl((\varphi_t)_{0 \le t \le T};(\nu_t)_{0 \le t \le T} \Bigr), 
\end{equation*}
with
\begin{equation}
\label{eq:cost:mfg} 
\begin{split}
\widetilde J\Bigl((\varphi_t)_{0 \le t \le T};(\nu_t)_{0 \le t \le T} \Bigr)
&:=
 \int_{\mathbb S}   g\bigl(\Phi_T(x),\nu_T \bigr) \ud x
\\
&\hspace{15pt} +  \int_0^T 
 \int_{\mathbb S}
 \Bigl(
f\bigl(\Phi_t(x),\nu_t \bigr) 
+
\frac12 
\vert \varphi_t(x)\vert^2 \Bigr)\ud t \, \ud x,
\end{split}
\end{equation} 
has exactly $(\nu_t)_{0 \le t \le T}$ as marginal laws. Above, 
the minimization is performed over measurable mappings $\varphi : (t,x) \in [0,T] \times {\mathbb S} \mapsto 
\varphi_t(x) \in {\mathbb R}$ satisfying 
$\int_0^T
\int_{\mathbb S} \vert \varphi_t(x) \vert^2 \ud t \ud x< \infty$
and $(\Phi_t)_{0 \le t \le T}$ is 
\begin{equation*}
\frac{\ud}{\ud t} 
\Phi_t(x) = 
\varphi_t(x), \quad t \in [0,T], \quad x \in {\mathbb S},
\end{equation*} 
where $\Phi_0$ is any measurable mapping from ${\mathbb S}$ to ${\mathbb R}$ such that 
${\rm Leb}_{\mathbb S} \circ \Phi_0^{-1} = \mu_0$. The flow of marginal laws of any trajectory is then 
given by $( {\rm Leb}_{\mathbb S} \circ \Phi_t^{-1} )_{0 \le t \le T}$. 
\end{prop}

\begin{proof} 
{ \ }

\textit{First Step.}
By \cite[Theorem 3.32]{CarmonaDelarueI}, there exists a unique mean field equilibrium (in the standard sense) under
the standing assumption. Moreover, it admits a decoupling field, in the sense that there exists a mapping 
\begin{equation*} 
{\mathcal W} : [0,T] \times {\mathbb R} \times {\mathcal P}_2({\mathbb R}) \rightarrow {\mathbb R}, 
\end{equation*} 
which is Lipschitz continuous in the second and third arguments, uniformly in time, such that the equilibrium 
$(\nu_t)_{0 \le t \le T}$ 
of the mean field game (which has not been yet identified with $(\mu_t)_{0 \le t \le T}$)
is the solution of the (first-order) Fokker-Planck equation:
\begin{equation*} 
\partial_t \nu_t - \textrm{\rm div}_x \bigl( {\mathcal W}(t,\cdot,\nu_t) \nu_t \bigr)=0, \quad t \in [0,T],
\end{equation*} 
with $\nu_0=\mu_0$ as initial condition. 
In fact, one can show (this is a consequence of the analysis carried out in \cite[Theorem 3.32]{CarmonaDelarueI})
that ${\mathcal W}$ is non-decreasing in the argument $x$. In short, this is a consequence of the 
fact that ${\mathcal W}$ has the same monotonicity structure as $\partial_x f$ and $\partial_x g$. 

Alternatively, $(\nu_t)_{0 \le t \le T}$ can be regarded as the flow of marginal laws of 
the solution to the McKean-Vlasov equation:
\begin{equation*}
\frac{\ud}{\ud t} Y_t(x) = - {\mathcal W}\bigl(t,Y_t(x),\textrm{\rm Leb}_{\mathbb S} \circ Y_t^{-1} \bigr),
\quad t \in [0,T],  \quad x \in {\mathbb S}, 
\end{equation*} 
with any initial condition 
$Y_0 : {\mathbb S} \rightarrow {\mathbb R}$ with $\mu_0$ as distribution. 
\vskip 4pt

\textit{Second Step.} In particular, one can choose 
$Y_0(x) = F_{\mu_0}^{-1}(x)$, for $x \in {\mathbb S}$, 
so that $Y_0 \in U^2({\mathbb S})$. We claim that, under this choice, 
$Y_t$ belongs to $U^2({\mathbb S})$ for any $t \in [0,T]$. The proof is twofold. First, we observe that 
$Y_t(x)=Y_t(-x)$, by uniqueness of the solution to the ODE
\begin{equation*} 
\dot{y}_t = - {\mathcal W}\bigl(t,y_t,\nu_t \bigr),
\quad t \in [0,T],
\end{equation*} 
with $y_0=Y_0(x)=Y_0(-x)$ as initial condition. Second, we notice that, for any $x,x' \in (0,1/2)$, 
with $x<x'$, 
\begin{equation} 
\label{eq:y:xprime:x}
\begin{split}
\frac{\ud}{\ud t} 
\biggl( \frac{Y_t(x') - Y_t(x)}{x'-x} \biggr) &=-  \frac{ {\mathcal W}(t,Y_t(x'),\nu_t) - {\mathcal W}(t,Y_t(x),\nu_t)}{x'-x}
\\
&=-  Z_t(x,x') \biggl( \frac{Y_t(x') - Y_t(x)}{x'-x} \biggr),
\end{split}
\end{equation} 
with 
\begin{equation*} 
Z_t(x,x') := 
\left\{
\begin{array}{ll}
\displaystyle 
\frac{ {\mathcal W}(t,Y_t(x'),\nu_t) - {\mathcal W}(t,Y_t(x),\nu_t)}{Y_t(x') - Y_t(x)} \quad &\textrm{if} \quad 
Y_t(x') \not = Y_t(x), 
\\
0 \quad &\textrm{if} \quad 
Y_t(x') = Y_t(x).
\end{array}
\right.
\end{equation*}
By the Lipschitz property of ${\mathcal W}$, we have $\vert Z_t(x,x') \vert \leq C$, for some constant $C \geq 0$
related with the Lipschitz constant 
of ${\mathcal W}$. And then we deduce 
from the `linear' ODE 
\eqref{eq:y:xprime:x} that $Y_t(x')-Y_t(x)$ has the same sign as $Y_0(x') - Y_0(x)$
and is thus non-negative. Moreover, 
\begin{equation*} 
\sup_{0 \leq t \leq T} \vert Y_t(x') - Y_t(x) \vert \leq \exp(CT) \vert Y_0(x')-Y_0(x) \vert, 
\end{equation*} 
which shows that each $Y_t$ is right-continuous in $0$ and left-continuous on $(0,1/2]$, because %since 
$Y_0$ is also right-continuous in $0$ and left-continuous on $(0,1/2]$. This completes the proof that 
$Y_t$ belongs to $U^2({\mathbb S})$ for each $t \in [0,T]$, that is $Y_t(x) = F_{\nu_t}^{-1}(x)$ for any $(t,x) \in [0,T] \times {\mathbb S}$. 
We thus rewrite the equation for $(Y_t(x))_{0 \le t \le T}$ in the form 
\begin{equation} 
\label{eq:widetilde:W} 
\frac{\ud}{\ud t} Y_t(x) = - {\mathcal W} \bigl(t, F_{\nu_t}^{-1}(x), \nu_t \bigr), \quad t \in [0,T], 
\quad x \in {\mathbb S}. 
\end{equation} 
This prompts us to let:
\begin{equation*} 
\widetilde{W}(t,x,\mu) = {\mathcal W} \bigl( t,F_{\mu}^{-1}(x), \mu \bigr), \quad (t,x,\mu) \in [0,T] \times 
{\mathbb S} \times {\mathcal P}_2({\mathbb R}). 
\end{equation*} 
Next, \eqref{eq:widetilde:W} can be reformulated as
\begin{equation}
\label{eq:widetilde:W:2} 
\begin{split}
\frac{\ud}{\ud t} Y_t(x) &= - \widetilde{\mathcal W} \bigl(t, x, \nu_t \bigr)
\\
&=  - \widetilde{\mathcal W} \bigl(t, x, {\rm Leb}_{\mathbb S} \circ Y_t^{-1} \bigr), \quad t \in [0,T], 
\quad x \in {\mathbb S}. 
\end{split}
\end{equation} 

\textit{Third Step.} 
The point is to regard 
\eqref{eq:widetilde:W:2}
as a reflected equation of the form 
\eqref{eq:V:bis}.
We first observe that, for any two $\mu,\mu' \in {\mathcal P}_2({\mathbb R}^d)$, 
\begin{equation} 
\label{eq:lipschitz:W}
\begin{split} 
\int_{\mathbb S} \bigl\vert 
\widetilde{W}(t,x,\mu)
-
\widetilde{W}(t,x,\mu') \bigr\vert^2 \ud x
&=
\int_{\mathbb S} \bigl\vert 
{\mathcal W} \bigl( t,F_{\mu}^{-1}(x), \mu \bigr)
-
{\mathcal W} \bigl( t,F_{\mu'}^{-1}(x), \mu' \bigr) \bigr\vert^2 \ud x
\\
&\leq C \int_{\mathbb S} \vert F_{\mu}^{-1}(x) - 
F_{\mu'}^{-1}(x) \vert^2 \ud x + 
C {\mathbb W}_2^2(\mu,\mu')
\\
&\leq C{\mathbb W}_2^2(\mu,\mu'),
\end{split} 
\end{equation} 
where the constant $C$ is related with the Lipschitz constant 
of ${\mathcal W}$. The third line follows 
from the solution of the ${\mathbb W}_2$-optimal transport 
problem in dimension 1, see 
\cite{villani2009}.

It remains to notice that, by composition of ${\mathcal W}$ and $F_{\mu}^{-1}$, the function 
$x \in {\mathbb S} \mapsto \tilde{\mathcal W}(t,x,\mu)$ 
belongs to $U^2({\mathbb S})$. Moreover, by 
\cite[Eq. (3.46)]{CarmonaDelarueI} and by boundedness of 
$\partial_x f$ and $\partial_x g$, 
the function ${\mathcal W}$ is bounded. 
This shows that 
$\widetilde{\mathcal W}$ is in the class ${\mathcal C}$.

Finally, since $Y_t$ belongs to $U^2({\mathbb S})$ for each $t \in [0,T]$, 
we can write 
\begin{equation*}
\ud Y_t(x) =  - \widetilde{\mathcal W} \bigl(t, x, {\rm Leb}_{\mathbb S} \circ Y_t^{-1} \bigr) + \ud \widetilde \eta_t(x), \quad t \in [0,T], 
\quad x \in {\mathbb S}, 
\end{equation*} 
with $(\widetilde \eta_t)_{0 \leq t \leq T} \equiv 0$. 
\vskip 4pt

\textit{Fourth Step.} 
We now show that the function $\widetilde{\mathcal W}$ satisfies Definition 
\ref{def:mfg:common}. 
Using the same notation as in 
\eqref{eq:cost}
and 
\eqref{eq:controlled:dynamics}, we consider the cost 
\begin{equation*}
J\Bigl((\gamma_t)_{0 \le t \le T};(Y_t)_{0 \le t \le T} \Bigr)
\end{equation*} 
which  (implicitly) 
requires
to associate with 
the control 
$(\gamma_t)_{0 \le t \le T}$ 
the dynamics
\begin{equation*}
\ud\Gamma_t(x) = \ud Y_t(x) +  \Bigl( \gamma_t(x) + \widetilde{\mathcal W}\bigl(t,x,{\rm Leb}_{\mathbb S} \circ Y_t^{-1} \bigr) \Bigr) \ud t, \quad t \in [0,T] \ ; 
\quad \Gamma_0(x) =Y_0(x), \quad x \in {\mathbb S}. 
\end{equation*}
By 
\eqref{eq:widetilde:W:2}, we get 
\begin{equation*}
\ud\Gamma_t(x) =  \gamma_t(x)  \ud t, \quad t \in [0,T] \ ; 
\quad \Gamma_0(x) =Y_0(x), \quad x \in {\mathbb S}. 
\end{equation*} 
Back to 
\eqref{eq:cost:mfg}, 
we deduce that 
\begin{equation*} 
J\Bigl((\gamma_t)_{0 \le t \le T};(Y_t)_{0 \le t \le T} \Bigr)
= 
\widetilde J\Bigl((\gamma_t)_{0 \le t \le T};(\nu_t)_{0 \le t \le T} \Bigr). 
\end{equation*} 
Since $(\nu_t)_{0 \le t \le T}$ is a mean field equilibrium (in the classical sense), we deduce that, under the initial condition 
$F_{\mu_0}^{-1}$, 
the optimal trajectory to the functional in the right-hand side
(which is unique under the standing convexity properties of $f$ and $g$ in $x$) is 
exactly $(Y_t)_{0 \le t \le T}$. The optimal control is 
\begin{equation*} 
\gamma_t(x) = - 
{\mathcal W}\bigl( t,Y_t(x), \mu \bigr)=-
{\mathcal W}\bigl( t,F_{\mu}^{-1}(x), \mu \bigr)
= - \widetilde{\mathcal W}(t,x,\mu), \quad (t,x) \in [0,T] \times {\mathbb S}. 
\end{equation*} 
This proves that the optimal control to the left-hand side is 
$(t,x) \mapsto -\widetilde{\mathcal W}(t,x,{\rm Leb}_{\mathbb S} \circ Y_t^{-1})$, hence
proving that 
$\widetilde{\mathcal W}$ satisfies Definition 
\ref{def:mfg:common}. 
\vskip 4pt

\textit{Fifth Step.} 
It remains to notice that there exists a unique solution to the mean field game defined in 
Definition 
\ref{subse:def:mean field game}.  
The proof is very similar to that of Proposition 
\ref{prop:fixed:point}. Indeed, 
if ${\mathcal U}$ is a solution and because 
$\widetilde{\mathcal W}$ 
is also a solution, we know from 
Proposition 
\ref{prop:2:2} that it must coincide 
with $\widetilde {\mathcal W}$ on 
$[T-\delta,T]$, for some $\delta >0$ depending on the Lipschitz constants of coefficients
$f$ and $g$. In this way, we have a representation formula 
for ${\mathcal U}(T-\delta,\cdot,\cdot)$. 
Thanks to 
\eqref{eq:lipschitz:W}, we have a Lipschitz property (in the measure argument) for this new boundary condition at time 
$T-\delta$. This allows us to iterate the proof and then to identify ${\mathcal U}$ 
and $\widetilde{\mathcal W}$ on the whole domain. 
%and that
%\eqref{eq:widetilde:W} 
%has a unique solution (for $F_{\mu_0}^{-1}$ as initial condition). Then, 
%we claim that
%\begin{equation*}
%\textrm{\rm Leb}_{\mathbb S} \circ \widetilde{Y}_t^{-1} = \nu_t.  
%\end{equation*} 
%
%Now, with the same notation $\widetilde{\mathcal U}$ as before, we consider the equation:
%\begin{equation}
%\label{eq:widetilde:X} 
%\frac{\ud}{\ud t} \widetilde{X}_t(x) =  - \widetilde{\mathcal U}\bigl(t,\widetilde{X}_t(x),\mu_t \bigr), \quad x \in {\mathbb S}, 
%\quad t \in [0,T].  
%\end{equation} 
%where 
%$\mu_t := {\rm Leb}_{\mathbb S} \circ \widetilde{X}_t^{-1}$. 
%We claim from the fact that $\widetilde{\mathcal U}$ is non-decreasing in the second argument that 
%\begin{equation*}
%\begin{split}
%&(y-y') \cdot \Bigl( \widetilde{\mathcal U}(t,y,\mu) - \widetilde{\mathcal U}(t,y',\mu) \Bigr)
%\geq 0,
%\end{split}
%\end{equation*}
%which proves that the equation is well-posed, in the sense that 
%\eqref{eq:widetilde:X} has a unique solution for any initial condition. As a result (this may be seen as a consequence of Yamada-Watanabe theorem),  
%the flow probability measures $(\tilde \mu_t := \textrm{\rm Leb}_{\mathbb S} \circ \widetilde X_t^{-1})_{t \geq 0}$ only depends on 
%the initial distribution $\tilde \mu_0 = \textrm{\rm Leb}_{\mathbb S} \circ \widetilde X_0^{-1}$. 
%
%We then consider the following variant of \eqref{eq:discrete:scheme:deterministic}:
%\begin{equation*} 
%\frac{\ud}{\ud t} \widetilde X_t^{(h)}(x) =  - {\mathcal U}\bigl(t,x, \mu_t \bigr), \quad x \in {\mathbb S}, \quad t \in [t_n,t_{n+1}). 
%\end{equation*} 
\end{proof}

\section{Appendix: proof of Theorem
\ref{thm:main:existence:uniqueness}}
\label{appendix}

The goal of this section is prove that equation 
\eqref{eq:V} has a solution when ${\mathcal V}$ in the class ${\mathcal C}$ (see Definition 
\ref{def:mathcalC}).

\subsection{Approximation scheme for proving existence}

The proof of existence achieved in 
\cite[Proposition 4.14]{delarueHammersley2022rshe} when ${\mathcal V}=0$ relies on 
a discretization scheme. We follow the same lines below, but we pay a special care in explaining the 
additional difficulties due to the presence of ${\mathcal V}$. 
Notice that another strategy based on  
Girsanov thereom would be conceivable, see for instance the forthcoming work 
\cite{delarueHammersley2022ergodicrshe}.

Following \cite{delarueHammersley2022rshe}, we consider the following approximation, along a uniform
time mesh $(t_n)_{n=0,\cdots,N}$  of the interval $[0,T]$. We call $h \in (0,T)$ the step size of the mesh. 

Iterating on the value of $n$, we then define: 
\begin{equation*}
\ud X_t^{(h)}(x) = - {\mathcal V}\bigl(t,x,\text{\rm Leb}_{\mathbb S} \circ (X_{t_n}^{(h)})^{-1} \bigr) \ud t + \Delta X_t^{(h)}(x) \ud t + \ud W_t(x), \quad t \in [t_n,t_{n+1}),
\end{equation*} 
the value of $X_{t_n}^{(h)}$ being given by the previous iteration of the scheme and the next value $X_{t_{n+1}}^{(h)}$ being given by
rearrangement: 
$$X_{t_{n+1}}^{(h)}= \bigl( X_{t_{n+1}-}^{(h)} \bigr)^\star.$$ We recall that the symbol $-$ in the above time index denotes the left-limit at point $t_{n+1}$. Moreover, the symbol $\star$ is the rearrangement operation, see \eqref{eq:rearrange}. 

One of the very key point in the proof is the analogue of 
Lemma 3.4 in 
\cite{delarueHammersley2022rshe}.

\begin{lem} 
\label{lem:1:1}
There exists a constant $C_T$, independent of $h$, such that 
\begin{equation*} 
{\mathbb E} 
\Bigl[ \bigl\|\nabla X_{t_n}^{(h)} \bigr\|_2^2 \Bigr]
\leq C_T \Bigl[ 1 +  \min \Bigl( \frac{1}{nh} \| X_0 \|_2^2 ,  \bigl\| \nabla X_0 \bigr\|_2^2 \Bigr) \Bigr].
\end{equation*} 
\end{lem}

\begin{proof}
We follow the proof of Lemma 3.4 in \cite{delarueHammersley2022rshe}.
We consider an additional random variable $U$, uniformly distributed on $(0,1)$
and independent of $W$. 
By Lemma 3.3 in \cite{delarueHammersley2022rshe}, 
we get the first line below, which holds true for any $r \geq 0$. The second line below is standard algebra. 
\begin{align}
&{\mathbb E} 
\Bigl[ \bigl\| \nabla \bigl(  e^{r U \Delta} X_{t_n}^{(h)} \bigr) \bigr\|_2^2 \Bigr] 
\nonumber
\\
&\leq 
{\mathbb E} 
\biggl[ \biggl\|
\nabla \biggl( e^{(r U + h) \Delta } X_{t_{n-1}}^{(h)} + \int_{t_{n-1}}^{t_n} e^{(r U + t_n-s) \Delta} \Bigl[ - {\mathcal V}\bigl(s,\cdot,
\text{\rm Leb}_{\mathbb S} \circ (X_{t_{n-1}}^{(h)})^{-1} \bigr) \ud s + \ud W_s \Bigr]
\biggr)
\biggr\|_2^2 \biggr] \nonumber
\\
&\leq (1+h) 
{\mathbb E} 
\biggl[ \biggl\|
\nabla \biggl( e^{( r U + h) \Delta } X_{t_{n-1}}^{(h)} + \int_{t_{n-1}}^{t_n} e^{(rU + t_n-s) \Delta}  \ud W_s 
\biggr)
\biggr\|_2^2 \biggr] 
\label{eq:tightness:1}
\\
&\hspace{10pt} + 
\bigl( 1+ \frac1{h} \bigr) {\mathbb E} 
\biggl[ \biggl\|
\nabla \biggl(  \int_{t_{n-1}}^{t_n} e^{(r U + t_n-s) \Delta}  \Bigl[ {\mathcal V}\bigl(s,\cdot,
\text{\rm Leb}_{\mathbb S} \circ (X_{t_{n-1}}^{(h)})^{-1} \bigr)    \Bigr]
\ud s \biggr)
\biggr\|_2^2 \biggr].
\nonumber
\end{align} 
And then, 
by orthogonality of the stochastic integral with the $\sigma$-field 
${\mathcal F}_{t_{n-1}}$
and by Jensen inequality, 
we obtain the first inequality below. 
%The second
%inequality is just a consequence of the contracting property of the heat semi-group in 
%$L^2({\mathbb S})$. 
\begin{equation*}
\begin{split}
&{\mathbb E} 
\Bigl[ \bigl\| \nabla \bigl( e^{r U \Delta } 
 X_{t_n}^{(h)} \bigr) \bigr\|_2^2 \Bigr] 
\\
&\leq 
 (1+h) 
{\mathbb E} 
\Bigl[ \bigl\|
\nabla \bigl( e^{(r U + h) \Delta } X_{t_{n-1}}^{(h)} 
\bigr) \bigr\|_2^2 
\Bigr] 
+ (1+h) {\mathbb E} 
\biggl[ 
\biggl\|
 \int_{t_{n-1}}^{t_n} \nabla e^{(r U + t_n-s) \Delta}  \ud W_s 
\biggr\|_2^2 \biggr]
\\
&\hspace{10pt} + 
\bigl( 1+ h \bigr) {\mathbb E} 
\biggl[ 
 \int_{t_{n-1}}^{t_n}
\Bigl\|
\nabla \Bigl(  e^{(rU +t_n-s) \Delta}  \bigl[ {\mathcal V}\bigl(s,\cdot,
\text{\rm Leb}_{\mathbb S} \circ (X_{t_{n-1}}^{(h)})^{-1}\bigr) 
   \bigr]
   \Bigr)
   \Bigr\|_2^2
   \ud s
    \biggr].
%\\
%&\leq 
% (1+h) 
%{\mathbb E} 
%\Bigl[ \bigl\|
%\nabla \bigl( e^{ (r+ h) U \Delta } X_{t_{n-1}}^{(h)} 
%\bigr) \bigr\|_2^2 
%\Bigr] 
%+ (1+h) {\mathbb E} 
%\biggl[ 
%\biggl\|
% \int_{t_{n-1}}^{t_n} \nabla e^{(  t_n-s) \Delta}  \ud W_s 
%\biggr\|_2^2 \biggr]
%\\
%&\hspace{10pt} + 
%\bigl( 1+ h \bigr) {\mathbb E} 
%\biggl[ 
% \int_{t_{n-1}}^{t_n}
%\Bigl\|
%\nabla \Bigl(  e^{( t_n-s) \Delta}  \bigl[  {\mathcal V}(s,\cdot,
%\text{\rm Leb}_{\mathbb S} \circ (X_s^{(h)})^{-1}) 
%   \bigr]
%   \Bigr)
%   \Bigr\|_2^2
%   \ud s
%    \biggr].    
\end{split}
\end{equation*} 
We then replace $t_n$ by $t_m$ and we choose $r=(n-m)h$, which makes it possible to iterate. 
We obtain, for $m \in \{1,\cdots,n\}$, 
\begin{equation*} 
\begin{split}
 {\mathbb E} 
\Bigl[ \bigl\| \nabla    X_{t_n}^{(h)}   \bigr\|_2^2 \Bigr] 
&\leq 
 (1+h)^{n-m+1}  
{\mathbb E} 
\Bigl[ \bigl\|
\nabla \bigl( e^{(n-m+1) h U \Delta } X_{t_{m-1}}^{(h)}
\bigr) \bigr\|_2^2 
\Bigr] 
\\
&\hspace{10pt} +
\sum_{k=m}^{n}
(1+h)^{n-k+1}  
\biggl\{
{\mathbb E} 
\biggl[ \biggl\| \int_{t_{k-1}}^{t_{k}}
\nabla e^{(t_{k}-s + (n-k) h U) \Delta} \ud W_s \biggr\|_2^2 \biggr]
\\
&\hspace{20pt} +
{\mathbb E} \biggl[ \int_{t_{k-1}}^{t_{k}} 
\Bigl\| \nabla \Bigl(  
e^{(t_k-s + (n-k) h U) \Delta}\bigl[ {\mathcal V}\bigl(s,\cdot,
\text{\rm Leb}_{\mathbb S} \circ (X_{t_{k-1}}^{(h)})^{-1}\bigr) 
   \bigr]
   \Bigr)
   \Bigr\|_2^2
   \ud s
   \biggr] \biggr\}.
\end{split}
\end{equation*} 
The sum on the second line is handled as in Lemma 
3.4 in 
\cite{delarueHammersley2022rshe}. 
Choosing $m=1$, we deduce that, for 
$\delta \in (0,\lambda-1/2)$,  
\begin{equation*} 
\begin{split}
 {\mathbb E} 
\Bigl[ \bigl\| \nabla    X_{t_n}^{(h)}   \bigr\|_2^2 \Bigr] 
&\leq 
 (1+h)^{n}  
{\mathbb E} 
\Bigl[ \bigl\|
\nabla \bigl( e^{n h U \Delta } X_{0} 
\bigr) \bigr\|_2^2 
\Bigr] 
+c_{\delta,\lambda} (hn)^{\delta}
\\
&\hspace{-15pt} +
\sum_{k=1}^{n}
(1+h)^{n-k+1}  
{\mathbb E} \biggl[ \int_{t_{k-1}}^{t_{k}} 
\Bigl\| \nabla \Bigl(  
e^{(t_k-s + (n-k) h U) \Delta}\bigl[ {\mathcal V}\bigl(s,\cdot,
\text{\rm Leb}_{\mathbb S} \circ (X_{t_{k-1}}^{(h)})^{-1}\bigr) 
   \bigr]
   \Bigr)
   \Bigr\|_2^2
   \ud s
   \biggr].
\end{split}
\end{equation*} 
The really new term is the sum on the second line above. 
We proceed as follows. 
\begin{equation*}
\begin{split}
\Bigl\| \nabla \Bigl(  
e^{(t_k-s + (n-k) h U) \Delta}\bigl[ {\mathcal V}\bigl(s,\cdot,
\text{\rm Leb}_{\mathbb S} \circ (X_{t_{k-1}}^{(h)})^{-1} \bigr) 
   \bigr]
   \Bigr)
   \Bigr\|_\infty
 \leq \frac{ c \| {\mathcal V} \|_\infty}{\sqrt{t_k - s + (n-k) h U}},
\end{split}
\end{equation*}
and
\begin{equation*}
\begin{split}
\Bigl\| \nabla \Bigl(  
e^{(t_k-s + (n-k) h U) \Delta}\bigl[ {\mathcal V}\bigl(s,\cdot,
\text{\rm Leb}_{\mathbb S} \circ (X_{t_{k-1}}^{(h)})^{-1} \bigr) 
   \bigr]
   \Bigr)
   \Bigr\|_1
\leq c \| {\mathcal V} \|_{\rm TV},
\end{split}
\end{equation*} 
with $\| \cdot \|_{\rm TV}$ denoting the total variation norm. 
Since ${\mathcal V}$ is non-decreasing on $(0,1/2)$ and symmetric with respect to $0$, 
we have 
$\|{\mathcal V} \|_{\rm TV} \leq 4 \| {\mathcal V} \|_\infty$. And, by H\"older inequality, 
\begin{equation}
\label{eq:tightness:2}
\begin{split}
\Bigl\| \nabla \Bigl(  
e^{(t_k-s + (n-k) h U) \Delta}\bigl[ {\mathcal V}\bigl(s,\cdot,
\text{\rm Leb}_{\mathbb S} \circ (X_{t_{k-1}}^{(h)})^{-1}\bigr) 
   \bigr]
   \Bigr)
   \Bigr\|_2^2
\leq  \frac{ c \| {\mathcal V} \|^2_\infty}{\sqrt{t_k - s + (n-k) h U}}. 
\end{split}
\end{equation} 
Then, 
\begin{equation}
\label{eq:tightness:3}
\begin{split}
&\int_{t_{k-1}}^{t_k} 
\Bigl\| \nabla \Bigl(  
e^{(t_k-s + (n-k) h U) \Delta}\bigl[ {\mathcal V}\bigl(s,\cdot,
\text{\rm Leb}_{\mathbb S} \circ (X_{t_{k-1}}^{(h)})^{-1} \bigr) 
   \bigr]
   \Bigr)
   \Bigr\|_2^2 \ud s 
\\
&\leq  \int_{t_{k-1}}^{t_k} 
\frac{ c \| {\mathcal V} \|_\infty}{\sqrt{t_k - s + (n-k) h U}} \ud s
=   \sqrt{h} \int_0^1 
\frac{ c \| {\mathcal V} \|_\infty}{\sqrt{ r + (n-k)  U}} \ud r
\\
&\leq   \frac{\sqrt{h}}{\sqrt{U}} \int_0^1 
\frac{ c \| {\mathcal V} \|_\infty}{\sqrt{ r + n-k }} \ud r
= 2c \frac{\sqrt{h}}{\sqrt{U}}  \| {\mathcal V} \|_\infty \Bigl[  \sqrt{ 1 + n-k}
-
 \sqrt{ n-k} \Bigr]. 
\end{split}
\end{equation} 
Eventually, 
\begin{equation*}
\begin{split}
&{\mathbb E} \sum_{k=1}^n \int_{t_{k-1}}^{t_k} 
\Bigl\| \nabla \Bigl(  
e^{(t_k-s + (n-k) h U) \Delta}\bigl[ {\mathcal V}\bigl(s,\cdot,
\text{\rm Leb}_{\mathbb S} \circ (X_{t_{k-1}}^{(h)})^{-1} \bigr) 
   \bigr]
   \Bigr)
   \Bigr\|_2^2 \ud s
   \\
   &\leq C  \sqrt{h} 
   \sum_{k=1}^n \Bigl[  \sqrt{ 1 + n-k}
-
 \sqrt{ n-k} \Bigr]= 2c \sqrt{h n}. 
\end{split}
\end{equation*}
In order to complete the proof, it remains to see that
\begin{equation*} 
{\mathbb E} 
\Bigl[ \bigl\|
\nabla \bigl( e^{n h U \Delta } X_{0} 
\bigr) \bigr\|_2^2 
\Bigr] \leq C_T 
\min \Bigl( \frac{1}{nh} \| X_0 \|_2^2 ,  \bigl\| \nabla X_0 \bigr\|_2^2 \Bigr), 
\end{equation*}  
see the end of the proof of 
  Lemma 
3.4 in 
\cite{delarueHammersley2022rshe}. 
\end{proof}

\subsection{Tightness}
The next step is to verify the various ingredients for proving tightness,  
as developed in \cite[Subsection 3.3]{delarueHammersley2022rshe}:
 \begin{prop}
\label{prop:tightness}
Denoting
by $(\tilde X^{(h)}_t)_{0 \le t \le T}$ the linear interpolation of 
$(X^{(h)}_{t_n})_{n=0,\cdots,N}$, the family $\{ \tilde X^{(h)}\}_{\textcolor{black}{h \in (0,1)}}$ induces a tight family of probability measures on the space  $\cC([0,T],L^2_{\rm sym}(\mathbb S))$ of continuous functions from 
$[0,T]$ into $L^2_{\rm sym}(\mathbb S)$. 
\end{prop}

\begin{proof}
The key is to adapt Proposition 3.7  in \cite{delarueHammersley2022rshe}. 
We follow, line by line, the computations therein. We start with 
\begin{equation*} 
\begin{split} 
\bigl\| X_{t_n}^{(h)}  - e^{nh \Delta} X_0 \bigr\|_2^2
&\leq 
\biggl\| e^{h \Delta} 
X_{t_{n-1}}^{(h)} 
- e^{nh \Delta} X_0 
\\
&\hspace{15pt} -
 \int_{t_{n-1}}^{t_n} e^{(t_n-s) \Delta} {\mathcal V}\bigl(s, \cdot, {\rm Leb}_{\mathbb S} \circ (X_{t_{n-1}}^{(h)})^{-1} \bigr) \ud s
		+  \int_{t_{n-1}}^{t_n} e^{ (t_n-s) \Delta} \ud W_s 
		\biggr\|_2^2. 
\end{split} 
\end{equation*} 
Proceeding as in the derivation of 
\eqref{eq:tightness:1}, we get 
\begin{equation*} 
\begin{split}
\bigl\| X_{t_n}^{(h)}  - e^{nh \Delta} X_0 \bigr\|_2^2
&\leq 
(1+h) 
\biggl\| e^{h \Delta} 
X_{t_{n-1}}^{(h)} 
- e^{nh \Delta} X_0 	+  \int_{t_{n-1}}^{t_n} e^{ (t_n-s) \Delta} \ud W_s 
		\biggr\|_2^2
\\
&\hspace{15pt} +
\bigl( 1 + \frac1{h} \bigr) 
\biggl\| 
 \int_{t_{n-1}}^{t_n} e^{(t_n-s) \Delta} {\mathcal V}\bigl(s, \cdot, {\rm Leb}_{\mathbb S} \circ (X_{t_{n-1}}^{(h)})^{-1} \bigr) \ud s
 \biggr\|_2^2
 \\
 &\leq (1+h) 
\biggl\| e^{h \Delta} 
X_{t_{n-1}}^{(h)} 
- e^{nh \Delta} X_0 	+  \int_{t_{n-1}}^{t_n} e^{ (t_n-s) \Delta} \ud W_s 
		\biggr\|_2^2
 + Ch, 
\end{split} 
\end{equation*} 
for a constant $C$ independent of $h$. 
We expand the above in the form 
\begin{equation*} 
\begin{split} 
\bigl\| X_{t_n}^{(h)}  - e^{nh \Delta} X_0 \bigr\|_2^2
&\leq 
(1+h) 
\bigl\| X_{t_{n-1}}^{(h)}  - e^{(n-1)h \Delta} X_0 \bigr\|_2^2
+ 
(1+h) 
\biggl\| 
 \int_{t_{n-1}}^{t_n} e^{ (t_n-s) \Delta} \ud W_s 
		\biggr\|_2^2
		\\
		&\hspace{15pt} 
		+ 
		2 (1+h) \Bigl\langle e^{h \Delta} X_{t_{n-1}}^{(h)} - e^{n h \Delta} X_0, 
 \int_{t_{n-1}}^{t_n} e^{ (t_n-s) \Delta} \ud W_s
 \Bigr\rangle_2 
 + Ch. 
\end{split} 
\end{equation*} 
By iterating, we obtain 
\begin{equation*} 
\begin{split} 
\bigl\| X_{t_n}^{(h)}  - e^{nh \Delta} X_0 \bigr\|_2^2
&\leq 
\sum_{k=1}^{n} (1+h)^{n-k+1} \biggl[ 
\biggl\| 
 \int_{t_{k-1}}^{t_k} e^{ (t_k-s) \Delta} \ud W_s 
		\biggr\|_2^2
		\\
&\hspace{15pt} 		+
		2   \Bigl\langle e^{h \Delta} X_{t_{k-1}}^{(h)} - e^{k h \Delta} X_0, 
 \int_{t_{k-1}}^{t_k} e^{ (t_k-s) \Delta} \ud W_s
 \Bigr\rangle_2 + C h 
 \biggr], 
\end{split}
\end{equation*}
which decomposition is very similar to 
\cite[(3.31)]{delarueHammersley2022rshe}. 
Following 
\cite[(3.32)]{delarueHammersley2022rshe}, we then let, 
for 
$n \in  \{0,\cdots,
 \lfloor T/h \rfloor \}$,
\begin{equation*} 
\begin{split}
&T_n^1 := \sum_{k=1}^{n} (1+h)^{n-k+1} 
\biggl\| 
 \int_{t_{k-1}}^{t_k} e^{ (t_k-s) \Delta} \ud W_s 
		\biggr\|_2^2, 
		\\
&T_n^2 := 2 \sum_{k=1}^{n} (1+h)^{n-k+1} 
 \Bigl\langle e^{h \Delta} X_{t_{k-1}}^{(h)} - e^{k h \Delta} X_0, 
 \int_{t_{k-1}}^{t_k} e^{ (t_k-s) \Delta} \ud W_s
 \Bigr\rangle_2,
\end{split} 
\end{equation*} 
And then, by the same computations as in 
\cite{delarueHammersley2022rshe}, we have, for any exponent $p \geq 1$, 
\begin{equation*} 
\begin{split}
&{\mathbb E} 
\Bigl[ 
\bigl\vert T_n^1 - T_m^1 \bigr\vert^p 
\Bigr] 
\leq 
C_p \bigl( h (n-m) \bigr)^p, \quad m,n \in \bigl\{0,\cdots,
 \lfloor T/h \rfloor \bigr\},
\end{split} 
\end{equation*} 
with the constant $C_p$ being allowed to depend on $p$.  
Moreover, with an obvious adaptation of Lemma 3.2 and Corollary 3.5 in 
\cite{delarueHammersley2022rshe} and with the same computations as in 
\cite[(3.33)]{delarueHammersley2022rshe}, we also have 
\begin{equation*} 
{\mathbb E} 
\Bigl[ 
\bigl\vert T_n^2 - T_m^2 \bigr\vert^p 
\Bigr] 
\leq 
C_p \bigl( h (n-m) \bigr)^{p/2}, \quad m,n \in \bigl\{0,\cdots,
 \lfloor T/h \rfloor \bigr\}.
 \end{equation*} 
%
%
%
%
% To do this, we let
% $A_0^{(h)} =0$ and then, by induction:
% \begin{equation*} 
% \begin{split}
% A_{t_n}^{(h)} &:=  \int_{t_{n-1}}^{t_n} e^{(t_n-s) \Delta} {\mathcal V}\bigl(s, \cdot, {\rm Leb}_{\mathbb S} \circ (X_s^{(h)})^{-1} \bigr) \ud s + e^{h \Delta} A_{t_{n-1}}^{(h)}. 
% \\
% &=
% \sum_{k=1}^n  \int_{t_{k-1}}^{t_k} e^{[ (n-k) h + (t_k-s)] \Delta} {\mathcal V}\bigl(s, \cdot, {\rm Leb}_{\mathbb S} \circ (X_s^{(h)})^{-1}\bigr) \ud s.
% \end{split}
% \end{equation*} 
%Importantly, each $A_{t_n}^{(h)}$ belongs to ${U}^2({\mathbb S})$. 
%We then have the following observation, very close to
%\eqref{eq:tightness:1}:
%\begin{equation*}
%	\begin{split}
%		\bigl\| X_{t_n}^{(h)} - A_{t_n}^{(h)} - e^{nh \Delta} X_0 \bigr\|_2^2
%		&\leq    \biggl\| 
%		e^{h \Delta} \Bigl(  X_{t_{n-1}}^{(h)} - A_{t_{n-1}}^{(h)} \Bigr) - e^{nh \Delta} X_0 
%		+  \int_{t_{n-1}}^{t_n} e^{ (t_n-s) \Delta} \ud W_s 
%		\biggr\|_2^2. 		
%	\end{split}
%\end{equation*}
Then it is easy to follow displays (3.34) to (3.39) in \cite{delarueHammersley2022rshe}. We deduce
\begin{equation*} 
\bigl\| \tilde X_t^{(h)} -   X_0 \bigr\|_2 \leq \Xi^{(h)} t^{\alpha/2} + w(t),
\end{equation*} 
where $(\tilde X_t^{(h)})_{0 \leq t \leq T}$ is the piecewise-linear interpolation of 
$(X_{t_n})_{n =0,\cdots,N}$. 
Here, $\Xi^{(h)}$ is a random variable with moments of any order $p \geq 1$ that are bounded independently of $h$
and $t \mapsto w(t)$ is random function that is bounded and that tends almost surely to $0$ as $t$ tends to $0$.

The next step is to study the analogue of (3.43) in 
\cite{delarueHammersley2022rshe}. Proceeding as before, we get 
\begin{equation*}
	\begin{split}
		 \left\lVert    X^{(h)}_{t_n}-   
		 e^{(n-m)h\Delta}   
		 X^{(h)}_{t_m}   \right\rVert_{2}^{2p}   
		 &\leq     
		 \biggl\| e^{h\Delta}  X^{(h)}_{t_{n-1}}
		  -
 \int_{t_{n-1}}^{t_n} e^{(t_n-s) \Delta} {\mathcal V}\bigl(s, \cdot, {\rm Leb}_{\mathbb S} \circ (X_{t_{n-1}}^{(h)})^{-1} \bigr) \ud s
\\
&\hspace{15pt}		  +\int_{t_{n-1}}^{t_n} e^{(t_n-s)\Delta}\ud W_s -e^{(n-m)h\Delta}  X^{(h)}_{t_m}   \biggr\|_{2}^{2p},
		%=  & \E  \left[ \left\lVert \left(e^{h\Delta}X^h_{n-1}-e^{(n-m)h\Delta}X^{h}_{m}\right)+\int_0^he^{(h-s)\Delta}d^{n}_s  \right\rVert_{2}^{2p} \right] \\
%		& =    \bigg( \left\lVert e^{h\Delta}X^h_{n-1}-e^{(n-m)h\Delta}X^{h}_{m}\right\rVert_{2}^{2} \\
%		& \quad \quad +2\left\langle  e^{h\Delta}X^h_{n-1}-e^{(n-m)h\Delta}X^{h}_{m}  , \int_0^he^{(h-s)\Delta}d^{n}_s   \right\rangle  + \left\lVert \int_0^he^{(h-s)\Delta}d^{n}_s  \right\rVert_{2}^{ 2}\bigg)^p , \\
%		%
	\end{split} 
\end{equation*}			
which prompts us to let
\begin{equation*}
\begin{split}
\hat{X}_s^{(h),n-1} &:= 
e^{s \Delta} 
\bigl[ X_{t_{n-1}}^{(h)} - e^{(t_{n-1} - t_m) \Delta} X_{t_m}^{(h)} \Bigr] 
\\
& \hspace{-10pt} 
- \int_0^{s} e^{(s-r) \Delta} {\mathcal V}\bigl( r + t_{n-1}, \cdot,  {\rm Leb}_{\mathbb S} \circ (X_{ t_{n-1}}^{(h)})^{-1} \bigr)
\ud r 
+ \int_0^s e^{(s-r) \Delta} \ud \bigl[ W_{t_{n-1} + r} - W_{t_{n-1}} \bigr]. 
\end{split} 
\end{equation*} 
By 
Theorem 6 in \cite{salavatiZangeneh2016pthPowerMaxIneqStochConv}, we obtain, for $t_n \geq t_m \geq t_N$ for some integer $N \geq 1$, 
\begin{equation*}
\begin{split}
&{\mathbb E} \biggl[ 
\Bigl\| X_{t_n}^{(h)} 
-
		 e^{(t_n-t_m)\Delta}   
		 X^{(h)}_{t_m}   
		 \Bigr\|_2^{2p} \, \vert \, {\mathcal F}_{t_N} \biggr] 
		 \\
&\leq
{\mathbb E} \biggl[ 
\Bigl\| X_{t_{n-1}}^{(h)} 
-
		 e^{(t_{n-1}-t_m) \Delta}   
		 X^{(h)}_{t_m}   
		 \Bigr\|_2^{2p} \, \vert \, {\mathcal F}_{t_N} \biggr]  
		 \\
		 & \hspace{15pt} 
		  + C h 
		 \biggl( 
		 		 {\mathbb E} \biggl[ 
\Bigl\| X_{t_{n-1}}^{(h)} 
-
		 e^{(t_{n-1}-t_m) \Delta}   
		 X^{(h)}_{t_m}   
		 \Bigr\|_2^{2p-1} \, \vert \, {\mathcal F}_{t_N} \biggr]  
		 \\
&\hspace{30pt} 		 +
		 {\mathbb E} \biggl[ 
\Bigl\| X_{t_{n-1}}^{(h)} 
-
		 e^{(t_{n-1}-t_m) \Delta}   
		 X^{(h)}_{t_m}   
		 \Bigr\|_2^{2(p-1)} \, \vert \, {\mathcal F}_{t_N} \biggr]  
		 + h^{p-1} 
		 \biggr) 
		 \\
		 &\leq (1+h)  {\mathbb E} \biggl[ 
\Bigl\| X_{t_{n-1}}^{(h)} 
-
		 e^{(t_{n-1}-t_m) \Delta}   
		 X^{(h)}_{t_m}   
		 \Bigr\|_2^{2p} \, \vert \, {\mathcal F}_{t_N} \biggr]  
		 \\
		 & \hspace{15pt} 
		  + C h 
		 \biggl( 		 
		 {\mathbb E} \biggl[ 
\Bigl\| X_{t_{n-1}}^{(h)} 
-
		 e^{(t_{n-1}-t_m) \Delta}   
		 X^{(h)}_{t_m}   
		 \Bigr\|_2^{2(p-1)} \, \vert \, {\mathcal F}_{t_N} \biggr]  
		 + h^{p-1} 
		 \biggr), 
\end{split}
\end{equation*} 
the second line following from H\"older and Young inequalities. 
Up to the pre-factor $(1+h)$ on the penultimate line, we recover the main inequality coming after (3.43) in 
\cite{delarueHammersley2022rshe}. This makes it possible to reproduce the computations therein and eventually to obtain the analogue of 
\cite[(3.44)]{delarueHammersley2022rshe}.  
We then get
\cite[(3.45)]{delarueHammersley2022rshe}.

Lastly, we need to readapt the third step in the proof of \cite[Proposition 3.7]{delarueHammersley2022rshe}. 
By returning to the proof of Lemma \ref{lem:1:1},
see in particular 
\eqref{eq:tightness:1}, 
 we obtain 
\begin{equation*}
\begin{split}
&\bigl\| \nabla \bigl( e^{(n-m) h U \Delta} X_{t_m} \bigr) \bigr\|_2^2 
\\
&\leq (1+h) 
 \biggl\|
\nabla \biggl( e^{[ (n-m) h U + h] \Delta } X_{t_{m-1}} + \int_{t_{m-1}}^{t_m} e^{[ (n-m) h U + t_m-s] \Delta}  \ud W_s 
\biggr)
\biggr\|_2^2
\\
&\hspace{15pt} + 
\bigl( 1+ \frac1{h} \bigr)  
  \biggl\|
\nabla \biggl(  \int_{t_{n-1}}^{t_n} e^{[ (n-m) h U + t_m-s] \Delta}  \Bigl[ {\mathcal V}(s,\cdot,
\text{\rm Leb}_{\mathbb S} \circ (X_{t_{n-1}}^{(h)})^{-1})    \Bigr]
\ud s \biggr)
\biggr\|_2^2 
\\
&\leq (1+h) 
 \biggl\|
\nabla \biggl( e^{[ (n-m) h U + h] \Delta } X_{t_{m-1}} + \int_{t_{m-1}}^{t_m} e^{[ (n-m) h U + t_m-s] \Delta}  \ud W_s 
\biggr)
\biggr\|_2^2
\\
&\hspace{15pt} + 
\bigl( 1+ h  \bigr)  
  \int_{t_{m-1}}^{t_m}
   \biggl\|
\nabla
 \biggl( e^{[ (n-m) h U + t_m-s ] \Delta}  \Bigl[ {\mathcal V}(s,\cdot,
\text{\rm Leb}_{\mathbb S} \circ (X_{t_{m-1}}^{(h)})^{-1})    \Bigr]
  \biggr)
\biggr\|_2^2 \ud s,  
\end{split}
\end{equation*} 
and then, by \eqref{eq:tightness:2}
and
\eqref{eq:tightness:3},  
\begin{equation*}
\begin{split}
&\bigl\| \nabla \bigl( e^{(n-m) h U \Delta} X_{t_m} \bigr) \bigr\|_2^2 
\\
&\leq (1+h) 
 \biggl\|
\nabla \biggl( e^{[ (n-m) h U + h] \Delta } X_{t_{m-1}} + \int_{t_{m-1}}^{t_m} e^{[ (n-m) h U + t_m-s] \Delta}  \ud W_s 
\biggr)
\biggr\|_2^2
\\
&\hspace{10pt}  + (1+h) 
 \int_{t_{m-1}}^{t_m}
 \frac{C}{\sqrt{t_m - s + (n-m) h U}} \ud s
 \\
&\leq (1+h) 
 \Bigl\|
\nabla \Bigl( e^{ [n-(m-1)] h U  \Delta } X_{t_{m-1}}
\Bigr)
\Bigr\|_2^2
+
(1+h) 
 \biggl\|
\nabla \biggl(  \int_{t_{m-1}}^{t_m} e^{[ (n-m) h U + t_m-s] \Delta}  \ud W_s 
\biggr)
\biggr\|_2^2
\\
&\hspace{10pt} +
2 (1+h) 
 \biggl\langle
\nabla \biggl( e^{[ (n-m) h U + h] \Delta } X_{t_{m-1}} \biggr), \biggl( \int_{t_{m-1}}^{t_m} e^{[ (n-m) h U + t_m-s] \Delta}  \ud W_s 
\biggr)
\biggr\rangle_2
\\
&\hspace{10pt} +  C  (1+h)  \frac{\sqrt{h}}{\sqrt{U}} \bigl[ \sqrt{1+n-m} - \sqrt{n-m} \bigr].
\end{split}
\end{equation*} 
And then, by induction, 
\begin{equation*}
\begin{split}
&\bigl\| \nabla \bigl( e^{(n-m) h U \Delta} X_{t_m} \bigr) \bigr\|_2^2 
\\
&\leq (1+h)^\ell
 \Bigl\|
\nabla \Bigl( e^{ [n-(m-\ell)] h U  \Delta } X_{t_{m-\ell}}
\Bigr)
\Bigr\|_2^2
\\
&\hspace{10pt} +
\sum_{k=m-\ell+1}^m
(1+h)^{m+1-k} 
 \biggl\|
\nabla \biggl(  \int_{t_{k-1}}^{t_k} e^{[ (n-k) h U + t_k-s] \Delta}  \ud W_s 
\biggr)
\biggr\|_2^2
\\
&\hspace{10pt} +
2 
\sum_{k=m-\ell+1}^m
(1+h)^{m+1-k} 
 \biggl\langle
\nabla \biggl( e^{[ (n-k) h U + h] \Delta } X_{t_{k-1}} \biggr), \biggl( \int_{t_{k-1}}^{t_k} e^{[ (n-k) h U + t_k-s] \Delta}  \ud W_s 
\biggr)
\biggr\rangle_2
\\
&\hspace{10pt} +  C 
\sum_{k=m-\ell+1}^m
(1+h)^{m+1-k} 
 \frac{\sqrt{h}}{\sqrt{U}} \bigl[ \sqrt{1+n-k} - \sqrt{n-k} \bigr].
\end{split}
\end{equation*}
Choosing $n=m$ and {$m-\ell=N$}, for some integer $N \geq 1$ (with $N \leq \lfloor T/h \rfloor$), performing the change of variable $j=m+1-k$, integrating with respect to $U$, taking the power $p \geq 1$ and then taking the conditional expectation given 
${\mathcal F}_{t_N}$, the second and third terms right above lead to similar quantities as $R^1$ and $R^2$ in the proof 
of  \cite[Proposition 3.7]{delarueHammersley2022rshe}. It then remains to observe that the last term is bounded by 
\begin{equation*} 
\begin{split}
&\biggl( 
{\mathbb E} \sum_{k=N+1}^m
(1+h)^{m+1-k} 
 \frac{\sqrt{h}}{\sqrt{U}} \bigl[ \sqrt{1+m-k} - \sqrt{m-k} \bigr]
\biggr)^p 
\\
&= \biggl( 
{\mathbb E} \sum_{j=1}^{m-N}
(1+h)^{j} 
 \frac{\sqrt{h}}{\sqrt{U}} \bigl[ \sqrt{j} - \sqrt{j-1} \bigr]
\biggr)^p
\\
&= c_p h^{p/2} 
\biggl( 
(1+h)^{m-N }
\sqrt{m-N} 
+
 \sum_{j=1}^{m-N-1}
\sqrt{j}  
 \bigl[  (1+h)^{j} - (1+h)^{j+1} \bigr]
\biggr)^p
\\
&= c_p h^{p/2} 
\biggl( 
(1+h)^{m-N }
\sqrt{m-N} 
-
h \sum_{j=1}^{m-N-1}
\sqrt{j}  
   (1+h)^{j}  
\biggr)^p
\\
&\leq C_{p,T}  \bigl[ h(m-N) \bigr]^{p/2},
\end{split} 
\end{equation*} 
which bound does not change the argument developed in 
the proof of \cite[Proposition 3.7]{delarueHammersley2022rshe}.
\end{proof}

\subsection{Existence of a solution and proof of \eqref{eq:main:existence:uniqueness}}
\label{subse:3:4} 
 
Following 
\cite{delarueHammersley2022rshe}, existence is proven by extracting a convergent subsequence of the family $((\tilde X_t^{(T/p)})_{0 \le t \le T})_{p \in {\mathbb N} \setminus \{0\}}$.
Very briefly, 
we decompose $X_{t_n}$ into 
$X_{t_n}:=
Y_{t_n}
+
V_{t_n}$,
where 
\begin{equation*} 
V_{t_{n+1}}
= e^{\Delta h} 
V_{t_n} 
- \int_0^{h} 
e^{\Delta (h-s)}
{\mathcal V}\Bigl(t_n+s,\cdot,
\text{\rm Leb}_{\mathbb S} \circ (X_{t_n})^{-1}
\Bigr) \ud s 
+
\int_0^{h} 
e^{\Delta (h-s)} \ud W_{t_n +s}. 
\end{equation*} 
Then, we observe that 
\begin{equation*} 
\begin{split} 
X_{t_{n+1}} &= \biggl( e^{\Delta h} 
X_{t_n} 
- \int_0^{h} 
e^{\Delta (h-s)}
{\mathcal V}\Bigl(t_n+s,\cdot,
\text{\rm Leb}_{\mathbb S} \circ (X_{t_n})^{-1}
\Bigr) \ud s 
+
\int_0^{h} 
e^{\Delta (h-s)} \ud W_{t_n +s} \biggr)^\star
\\
&= 
\biggl( e^{\Delta h} 
V_{t_n} 
- \int_0^{h} 
e^{\Delta (h-s)}
{\mathcal V}\Bigl(t_n+s,\cdot,
\text{\rm Leb}_{\mathbb S} \circ (X_{t_n})^{-1}
\Bigr) \ud s 
\\
&\hspace{15pt} +
\int_0^{h} 
e^{\Delta (h-s)} \ud W_{t_n +s}
+ e^{h \Delta} \bigl( X_{t_n} - V_{t_n} \bigr) 
 \biggr)^\star
 \\
 &= 
 \Bigl( V_{t_{n+1}}  + e^{h \Delta} Y_{t_n} 
  \Bigr)^\star,
\end{split} 
\end{equation*} 
which leads to 
\begin{equation*} 
Y_{t_{n+1}} =  \Bigl( V_{t_{n+1}}  + e^{h \Delta} Y_{t_n} 
  \Bigr)^\star - V_{t_{n+1}}, 
\end{equation*} 
which is the analogue of (4.6) in
\cite{delarueHammersley2022rshe}. 
We then proceed as in the proof of Proposition 4.1 in 
\cite{delarueHammersley2022rshe}: roughly speaking, 
the process $(Y_{t_n})_{0 \leq n \leq \lfloor T/h \rfloor}$ induces the reflection term and 
 the process 
 $(V_{t_n})_{0 \leq n \leq \lfloor T/h \rfloor}$ generates the non-reflected part of the dynamics.

To prove the first inequality in \eqref{eq:main:existence:uniqueness}, we notice that, by Lemma 
\ref{lem:1:1},
\begin{equation*} 
{\mathbb E} 
\Bigl[ \bigl\|\nabla \bigl( e^{\varepsilon \Delta}  X_{t_n}^{(h)} \bigr) \bigr\|_2^2 \Bigr]
\leq C_T \Bigl[ 1+ \min \Bigl( \frac{1}{t_n} \| X_0 \|_2^2 , \bigl\| \nabla X_0 \bigr\|_2^2 \Bigr) \Bigr]. 
\end{equation*} 
Letting $h$ tend to $0$, we get 
\begin{equation*} 
{\mathbb E} 
\Bigl[ \bigl\|\nabla \bigl( e^{\varepsilon \Delta}  X_{t}  \bigr) \bigr\|_2^2 \Bigr]
\leq C_T \Bigl[ 1 +  \min \Bigl( \frac{1}{t} \| X_0 \|_2^2 , \bigl\| \nabla X_0 \bigr\|_2^2 \Bigr) \Bigr].
\end{equation*} 
Letting $\varepsilon$ to $0$, we get the result. 

As for the second inequality in \eqref{eq:main:existence:uniqueness}, 
its proof is similar to the one in 
\cite{delarueHammersley2022ergodicrshe}. We give it for the sake of completeness. 
It can be checked that the mean (over ${\mathbb S}$) satisfies
\begin{equation*} 
\ud \int_{\mathbb S} X_t(x) \ud x = 
- \int_{\mathbb S}{\mathcal V}(t,x,{\rm Leb}_{\mathbb S} \circ X_t^{-1})  \ud x \, \ud t
+ \ud B_t^0.
\end{equation*} 
Since ${\mathcal V}$ is bounded, the result
is straightforward for the sole mean, which we denote by $(\overline X_t:=\int_{\mathbb S} X_t(x) \ud x)_{0 \le t \le T}$. 
Next, we expand 
\begin{equation*}
\begin{split} 
&\ud \| X_t - \bar X_t \|_2^2 
\\
&= - 2 \langle X_t -\bar X_t,{\mathcal V}(t,\cdot,{\rm Leb}_{\mathbb S} \circ X_t^{-1} ) \rangle_2 \ud t 
- \| \nabla X_t \|_2^2 \ud t + 2 \langle X_t - \bar X_t, \ud W_t \rangle_2 + c_\lambda \ud t.
%\\
%&\leq 
%C \bigl( 1+ \| X_t - \bar X_t \|_2^2 \bigr) \ud t  + 2 \langle X_t - \bar X_t, \ud W_t \rangle_2,
\end{split}
\end{equation*} 
%and, then, for any $\varepsilon >0$, 
%\begin{equation*} 
%\begin{split} 
%\sup_{0 \le t \le T} 
%\exp \bigl( \varepsilon \| X_t - \bar X_t \|_2^2 \bigr) 
%
%\end{split} 
%\end{equation*} 
%
%
%
% by Burkholder-Davis-Gundy inequality, it is sufficient to prove that 
%\begin{equation*} 
%\sup_{0 \le t \le T} {\mathbb E} \bigl[ \exp \bigl( \varepsilon  \| X_t - \bar X_t\|_2^2 \bigr) 
%\bigr]
% < \infty,
%\end{equation*} 
%for some $\varepsilon >0$. 
Using Poincaré inequality, one can prove that, for any $\varepsilon \in (0,1)$, 
\begin{equation*} 
\begin{split}
\ud \bigl[ \exp \bigl( \varepsilon \| X_t - \bar X_t \|_2^2 \bigr) \bigr] 
&\leq 
C_\varepsilon 
\exp \bigl( \varepsilon \| X_t - \bar X_t \|_2^2 \bigr)
\bigl( 1+ \varepsilon^2 \| X_t  - \bar X_t\|_2^2 - c \varepsilon \|X_t - \bar X_t \|^2_2 \bigr)   \ud t  
\\
&\hspace{15pt} + 
2 \varepsilon
\exp \bigl( \varepsilon \| X_t - \bar X_t \|_2^2 \bigr)
 \langle X_t - \bar X_t, \ud W_t \rangle_2,
\end{split} 
\end{equation*} 
where $c$ denotes the constant in Poincaré inequality over 
${\mathbb S}$. Above, $C_\varepsilon$ may depend on $\varepsilon$.  
Choosing $\varepsilon$ small enough, we deduce that 
% by Burkholder-Davis-Gundy inequality, it is sufficient to prove that 
\begin{equation*} 
\sup_{0 \le t \le T} {\mathbb E} \bigl[ \exp \bigl( \varepsilon  \| X_t - \bar X_t\|_2^2 \bigr) 
\bigr]
 < \infty,
\end{equation*} 
for some $\varepsilon >0$. 
By Burkholder-Davis-Gundy inequality, 
we complete the proof (with a possibly smaller value of $\varepsilon$).

\bibliographystyle{alpha}
 
\bibliography{FullBibliographyWRPH}

%%%%%%%%%%%%%%%%%%%%%%%%%%%%%%%%%%%%%%%%%%%%%%%%%%%%%%%%%%%%%%%%%%%
%%                                                               %%
%% Use the two commands below for producing your bibliography    %%
%% with bibtex, then comment again the commands and include the  %%
%% content of the .bbl file in this file below the commands.     %%
%%                                                               %%
%%%%%%%%%%%%%%%%%%%%%%%%%%%%%%%%%%%%%%%%%%%%%%%%%%%%%%%%%%%%%%%%%%%

%\bibliographystyle{amsplain}
%\bibliography{yourbibfilename}

% add below the content of your .bbl file produced by bibtex.

%%%%%%%%%%%%%%%%%%%%%%%%%%%%%%%%%%%%%%%%%%%%%%%%%%%%%%%%%%%%%%%%%%%
%%                                                               %%
%% You may add acknowledgments (optional).                       %%
%%                                                               %%
%%%%%%%%%%%%%%%%%%%%%%%%%%%%%%%%%%%%%%%%%%%%%%%%%%%%%%%%%%%%%%%%%%%
\begin{acks}
Youssef Ouknine thanks UniCA for having hosted him during the preparation of this work. His stay was supported by Centre National de la Recherche Scientifique (CNRS, Call `postes rouges').
François Delarue acknowledges the financial support of the European
Research Council (ERC) under the European Union’s Horizon 2020 research and innovation
programme (AdG ELISA project, Grant agreement No. 101054746).
\end{acks}

%%%%%%%%%%%%%%%%%%%%%%%%%%%%%%%%%%%%%%%%%%%%%%%%%%%%%%%%%%%%%%%%%%%
%%                                                               %%
%% You have reached the end of your document.                    %%
%%                                                               %%
%%%%%%%%%%%%%%%%%%%%%%%%%%%%%%%%%%%%%%%%%%%%%%%%%%%%%%%%%%%%%%%%%%%

\end{document}